\documentclass{article}

\usepackage{amsthm}
\usepackage{amsmath}
\usepackage{centernot}
\usepackage{amssymb}
\usepackage[margin=1.15in]{geometry}
\usepackage{accents}

\newtheorem{theorem}{Theorem}[section]
\newtheorem{remark}{Remark}[section]
\newtheorem{lemma}{Lemma}[section]
\newtheorem{proposition}{Proposition}[section]
\newtheorem{corollary}{Corollary}[section]

\newcommand{\R}{\mathbb{R}}
\newcommand{\C}{\mathbb{C}}
\newcommand{\Z}{\mathbb{Z}}
\newcommand{\A}{\mathcal{A}}
\newcommand{\K}{\mathcal{K}}
\newcommand{\p}{{\rm p}}
\newcommand{\tomega}{{\widetilde\omega}}
\newcommand{\tvarphi}{{\widetilde\varphi}}
\newcommand{\tpsi}{{\widetilde\psi}}
\newcommand{\tu}{{\widetilde u}}
\newcommand{\tr}{{\widetilde r}}
\newcommand{\hrho}{{\widehat\rho}}
\newcommand{\hA}{{\widehat A}}
\newcommand{\x}{\langle{\rm x}\rangle}
\newcommand{\z}{\langle z\rangle}
\newcommand{\w}{\langle w\rangle}

\newcommand{\divop}{\mathop{\rm div}}
\newcommand{\curlop}{\mathop{\rm curl}}
\newcommand{\idmap}{{\rm id}}
\newcommand{\diff}{\mathop{\rm diff}\nolimits}
\newcommand{\Diff}{\mathop{\rm Diff}\nolimits}
\newcommand{\dt}[1]{\accentset{\mbox{\bfseries .}}{#1}}
\newcommand{\dtt}[1]{\accentset{\mbox{\bfseries ..}}{#1}}
\newcommand{\bolddot}{{\mbox{\bfseries .}}}
\newcommand{\norm}[1]{\left\lVert#1\right\rVert}
\newcommand{\abnorm}[1]{\left\lvert#1\right\rvert}
\newcommand{\vertiii}[1]{{\left\vert\kern-0.25ex\left\vert\kern-0.25ex\left\vert #1
\right\vert\kern-0.25ex\right\vert\kern-0.25ex\right\vert}}

\makeatother

\begin{document}

\title{On the asymptotic behavior of solutions of the 2d Euler equation}

\author{Saif Sultan and Peter Topalov\footnote{P.T. is partially supported by the Simons Foundation, Award \#526907}}

\date{}

\maketitle

\begin{abstract}
We prove that the 2d Euler equation is globally well-posed in a space of vector fields having spatial asymptotic 
expansion at infinity of any a priori given order. The asymptotic coefficients of the solutions are holomorphic functions of $t$,
do {\em not} involve (spacial) logarithmic terms, and develop even when the initial data has fast decay at infinity. 
We discuss the evolution in time of the asymptotic terms and their approximation properties. 

\vspace{0.5cm}

\end{abstract}

\section{Introduction}
The objective of this paper is to study the spatial asymptotic behavior at infinity of the solutions
of the 2d Euler equation on $\R^2$. We do this by introducing a specific scale of Banach spaces of vector fields on $\R^2$
that captures the asymptotic phenomena  arising at the spatial infinity of the solutions of the equation.
We prove that the 2d Euler equation is globally well-posed in this scale of spaces and show that the corresponding 
solutions are well approximated by their asymptotic part.

Recall that the incompressible Euler equation on $\R^d$ with $d\ge 2$ is given by
\begin{equation}\label{eq:euler}
\left\{
\begin{array}{l}
v_t+v\cdot\nabla v=-\nabla\p,\quad\divop v=0,\\
v|_{t=0}=v_0,
\end{array}
\right.
\end{equation}
where $v(\rm x,t)$ is the velocity of the fluid and $\p(\rm x,t)$ is the scalar pressure.
The divergence $\divop$ and the covariant derivative $\nabla$ are computed with respect to
the Euclidean metric on $\R^d$. 
It was proven in \cite{McOTo3} that the Euler equation \eqref{eq:euler} is locally well-posed in the class 
of asymptotic vector fields on $\R^d$. An asymptotic vector field can be written in the form
\begin{equation}\label{eq:asymptotic_expansion_logs}
v({\rm x})=a_0(\theta)+\frac{a_1^0(\theta)+a_1^1(\theta)\log r}{r}+...+
\frac{a_N^0(\theta)+...+a_N^N(\theta)(\log r)^N}{r^N}+f({\rm x}),\quad |{\rm x}|\ge 1/2,
\end{equation}
where $r\equiv r({\rm x}):=|{\rm x}|$, $\theta\equiv\theta({\rm x}):={\rm x}/|{\rm x}|$ is a point on the unit sphere 
$S^{d-1}$, $N\ge 0$ is a given integer, the coefficients $a_k^j$ for $0\le k\le N$ and $0\le j\le k$ are continuous vector-valued 
functions on the sphere $a_k^j : S^{d-1}\to\R^d$ having an additional regularity, and $f : \R^d\to\R^d$ is
the {\em remainder} that belongs to a weighted Sobolev space on $\R^d$ so that it is continuous and 
$f({\rm x})=o(1/r^N)$ as $r\to\infty$.  The class of asymptotic vector fields is denoted by $\A^{m,p}_{N;0}$ where
$m>d/p$ is a regularity exponent and $1<p<\infty$. (We refer for the details to \cite[Section 2]{McOTo3}.)
In this sense, any vector field in $\A^{m,p}_{N;0}$ has partial asymptotic expansion at infinity of order $N$.
There are two points worth mentioning. The first one is that the class $\A^{m,p}_{N;0}$ of asymptotic vector fields
is {\em natural} for the Euler equation for the following reason: if we take a generic initial data
$v_0$ of order $o\big(1/r^N\big)$ with $N> d+1$, the corresponding local solution of the Euler equation 
develops at spatial infinity an asymptotic expansion of the form \eqref{eq:asymptotic_expansion_logs} with {\em non-vanishing} 
leading asymptotic term $a(\theta)/r^{d+1}$ (cf. \cite[Example 2, Appendix B]{McOTo3} as well as, its generalization,
Corollary \ref{coro:asymptotics_pop-up} below). The second remark is that if we take a generic initial data of the form 
\eqref{eq:asymptotic_expansion_logs} with vanishing coefficients in front of the log terms then the local solution of the Euler equation 
develops an asymptotics expansion with non-vanishing log terms. 
These log terms appear because the space $\A^{m,p}_N$ with vanishing log terms imposes no restrictions on the spherical 
Fourier modes of the coefficients $a_k^0$, $0\le k\le N$ (cf. \cite[Example 1, Appendix B]{McOTo3} as well as 
Remark \ref{rem:real_components} below).

\medskip

As mentioned above, in this work we restrict our attention to the two dimensional case $d=2$. 
Our first objective is to answer the following questions:

\begin{itemize}
\item[1)] Is there a closed subspace of the asymptotic space $\A^{m,p}_{N;0}$ that is {\em invariant} with respect to 
the Euler equation, closed (within the scale) with respect to the commutator of vector fields, 
and has an asymptotic part {\em without} log terms?
\item[2)] Are the solutions of the Euler equation in such a space {\em global} in time?
\end{itemize}

We will show that the answer to the both questions is {\em positive}. An important ingredient of our analysis is
that we identify $\R^2$ with the complex plane $\C$ by setting $z=x+i y$ for $(x,y)\in\R^2$ and 
re-write the Euler equation \eqref{eq:euler} in complex form as
\begin{equation}\label{eq:euler_complex}
\left\{
\begin{array}{l}
u_t+u u_z+{\bar u} u_{\bar z}=-2\p_{\bar z},\quad\divop u=u_z+{\bar u}_{\bar z}=0,\\
u|_{t=0}=u_0,
\end{array}
\right.
\end{equation}
where $u : \C\to\C$ is the holomorphic component of the fluid velocity $v$ (i.e., $u(z,\bar z):=v_1(z,\bar z)+i v_2(z,\bar z)$
where $(v_1,v_2)$ are the components of the velocity vector field $v$)
and the subscripts $z$ and $\bar z$ denote the partial differentiations with respect to 
the Cauchy operators $\partial_z$ and $\partial_{\bar z}$ respectively (see Appendix \ref{appendix:complex_form}).
With this identification in mind, we define for integer $m>2/p$ and $N\ge 0$ the following asymptotic space of vector fields on $\R^2$,
\begin{align}\label{eq:Z-space}
\mathcal{Z}_N^{m,p}:=\Big\{ \chi(|z|)\sum_{0\le k+l\le N}\frac{a_{kl}}{z^k\bar{z}^l}+f
\,\Big|\,f \in W_{\gamma_N}^{m,p}\,\,\text{\rm and}\,\,a_{kl}\in\mathbb{C}\Big\},
\end{align}
where $\chi : \R\to\R$ is a $C^\infty$-smooth cut-off function such that $\chi(\rho)=1$ for $|\rho|\ge 2$,
$\chi(\rho)=0$ for $|\rho|\le 1$, and $0\le\chi(\rho)\le 1$ for $1\le |\rho|\le 2$, equipped with the norm
\begin{align}\label{eq:Z-norm}
\norm{u}_{\mathcal{Z}^{m,p}_N}:=\sum_{0\le k+l\le N} \abnorm{a_{kl}} + 
\norm{f}_{W^{m,p}_{\gamma_N}}
\end{align}
where we implicitly assume in the sums that $k,l\ge 0$.
Here $W^{m,p}_\delta$ denotes the weighted Sobolev space with weight $\delta\in\R$ and regularity exponent $m\ge 0$,
\[
W^{m,p}_\delta:=\big\{f\in H^{m,p}_{loc}(\C,\C)\,\big|\,
\z^{\delta+|\alpha|}\partial^\alpha f\in L^p\,\,\text{\rm for}\,\,|\alpha|\le m\big\},\quad
\partial^\alpha\equiv\partial_z^{\alpha_1}\partial_{\bar z}^{\alpha_2},
\]
supplied with the norm $\|f\|_{W^{m,p}_\delta}:=\sum_{|\alpha|\le m}\big\|\z^{\delta+|\alpha|}\partial^\alpha f\big\|_{L^p}$
where $\z:=\big(1+|z|^2\big)^{1/2}$. The asymptotic space $\mathcal{Z}_N^{m,p}$ is a {\em Banach algebra} with respect to 
the pointwise multiplication of complex-valued functions.
The weight $\gamma_N$ and the regularity $m>2/p$ in \eqref{eq:Z-space} are chosen so that the remainder 
$f\in W_{\gamma_N}^{m,p}$ is of order $o\big(1/r^N\big)$ as $r\to\infty$. More specifically, we choose $\gamma_N:=N+\gamma_0$ 
where the number $\gamma_0$ with $0<\gamma_0+(2/p)<1$ is fixed and does not depend on the choice of the order of 
the asymptotic $N\ge 0$ or the regularity exponents $m>2/p$. As mentioned above, $f \in W_{\gamma_N}^{m,p}$ with 
$m>d/p$ implies that $f$ is of order $o\big(1/r^N\big)$ as $r\to\infty$ (see e.g. Proposition \ref{prop:main_technical} $(ii)$ 
and Remark \ref{rem:R^d} in Appendix \ref{appendix:technical}). In this way, the elements $u$ of $\mathcal{Z}_N^{m,p}$ 
represent vector fields $v=(v_1,v_2)$ on $\R^2$ that have asymptotic expansion at infinity of the form \eqref{eq:asymptotic_expansion_logs} 
without log terms.

\begin{remark}\label{rem:real_components}
More specifically, since $\frac{1}{z^k\bar{z}^l}=\frac{e^{i(l-k)\vartheta}}{r^{k+l}}$ any of the components 
$v_1$ and $v_2$ of such vector fields have asymptotic expansions at infinity of the form
$\sum_{0\le k\le N}\frac{a_k(\vartheta)}{r^k}$ where the coefficients $a_k(\vartheta)$ have truncated Fourier series 
that involve only the Fourier modes $e^{i(k-2j)\vartheta}$, $0\le j\le k$. 
(Here $r>0$ and $0\le\vartheta<2\pi$ denote  the polar coordinates on $\R^2$.)
These specific restrictions on the spherical modes of the asymptotic coefficients of the solutions of the 2d Euler equation
result in the absence of logarithmic terms along the time evolution of the initial data.
\end{remark}

\noindent For $0\le n\le N+1$ introduce the space
\begin{equation}\label{e:Zn-space}
\mathcal{Z}_{n,N}^{m,p}:=\Big\{\chi\!\!\!\sum_{n\le k+l\le N}\frac{a_{kl}}{z^k\bar{z}^l}+f
\,\Big|\,f \in W_{\gamma_N}^{m,p}\,\,\text{\rm and}\,\,a_{kl}\in\C\Big\}
\end{equation}
where we {\em omit} the summation term if $n=N+1$ and set $\mathcal{Z}_{N+1,N}^{m,p}\equiv W^{m,p}_{\gamma_N}$.
One of the main result of this work is the following theorem proven in Section \ref{sec:global_solutions}.

\begin{theorem}\label{th:main_introduction}
Assume that $m>3+(2/p)$ where $1<p<\infty$.
Then for any $u_0 \in \mathcal{Z}^{m,p}_N$ with $N\ge 0$ the 2d Euler equation has a unique {\em global} in time solution 
$u\in C\big([0,\infty),\mathcal{Z}^{m,p}_N\big) \cap C^1\big([0,\infty), \mathcal{Z}^{m-1,p}_N\big)$
such that the pressure $\p\equiv\p(t)$ lies in $\mathcal{Z}^{m+1,p}_{1,N}$ for any $t\ge 0$. 
The solution depends continuously on the initial data $u_0\in\mathcal{Z}^{m,p}_N$ in the sense that 
for any given $T>0$ the data-to-solution map $u_0 \mapsto u$, 
$\mathcal{Z}^{m,p}_N\to C\big([0,T],\mathcal{Z}^{m,p}_N\big)\cap C^1\big([0,T],\mathcal{Z}^{m-1,p}_N\big)$,
is continuous. 
The coefficients $a_{kl} : [0,\infty)\to\C$, $0\le k+l\le N$, in the asymptotic expansion of the solution
\begin{equation}\label{eq:u_expansion}
u(t)=\chi\sum_{0\le k+l\le N}\frac{a_{kl}(t)}{z^k\bar{z}^l}+f(t),\quad f(t)\in W^{m,p}_{\gamma_N},
\end{equation}
are {\em holomorphic} functions of $t$ in an open neighborhood of $[0,\infty)$ in $\C$.
\end{theorem}

\begin{remark}
Note that the remainder $f : [0,\infty)\to W^{m,p}_{\gamma_N}$ of the solution \eqref{eq:u_expansion} given by 
Theorem \ref{th:main_introduction} is {\em not} necessarily holomorphic in time.
\end{remark}

\begin{remark}\label{rem:uniqueness}
Without the assumption that for $t\in[0,\infty)$ the pressure $\p(t)$ lies in $\mathcal{Z}^{m+1,p}_{1,N}$ the solutions of 
the 2d Euler equation are {\em not} unique. Moreover, one can show that for any initial data 
$u_0\in\mathcal{Z}^{m,p}_N$ and for any $0<\tau<\infty$ there exists a solution 
$u\in C\big([0,\tau),\mathcal{Z}^{m,p}_N\big)\cap C^1\big([0,\tau),\mathcal{Z}^{m-1,p}_N\big)$
of the 2d Euler equation that is defined on the interval $[0,\tau)$ but blows-up at $t=\tau$.
The condition on the pressure can be relaxed, e.g. one can require that $\p(t)\in L^\infty$ and
$\p_{\bar{z}}(t)=o(1)$ as $|z|\to\infty$ for any $t\in\R$.
In any case, by Lemma \ref{LemmaPressureTermAsymp} in Section \ref{sec:euler_equation}, 
the pressure $\p$ in Theorem \ref {th:main_introduction} will belong to 
$C\big([0,\infty),\mathcal{Z}^{m+1,p}_{1,N}\big)\cap C^1\big([0,\infty),\mathcal{Z}^{m,p}_{1,N}\big)$.
\end{remark}

The following statements give a general characterization of the time evolution of the asymptotic coefficients of the solutions
in Theorem \ref{th:main_introduction}. These statements are proven in Section \ref{sec:euler_equation}. 

\begin{proposition}\label{prop:integrals}
Let $u\in C\big([0,\infty), \mathcal{Z}^{m,p}_N\big)\cap C^1\big([0,\infty),\mathcal{Z}^{m-1,p}_N\big)$ be the solution of the 2d Euler equation
given by Theorem \ref{th:main_introduction}, $u(t)=\chi\sum_{0\le k+l\le N}\frac{a_{kl}(t)}{z^k\bar{z}^l}+f(t),\quad f(t)\in W^{m,p}_{\gamma_N}$.
Then, for any divergence free initial data $u_0\in\mathcal{Z}^{m,p}_N$, the asymptotic coefficient $a_{10}(t)$ vanishes and the coefficients 
$a_{00}(t)$ and $a_{01}(t)$ are independent of $t\in[0,\infty)$. In addition, if $N\ge 2$ then for any $t\in[0,\infty)$ we have that 
$a_{20}(t)=a_{11}(t)=0$ and $a_{02}(t)=a_{02}(0)+\overline{a_{00}(0)}\,a_{01}(0)\,t$.
\end{proposition}

\begin{remark}\label{rem:divergence_free_asymptotics}
One can easily see that the divergence free condition $a_z+\bar{a}_{\bar{z}}=0$ implies that the asymptotic coefficients
$a_{k0}$ for $1\le k\le N$ and $a_{k1}$ for $1\le k\le N-1$ vanish.
\end{remark}

As a direct consequence we obtain the following 2d analog of \cite[Theorem 2.3]{Cantor} (see also \cite{Con1}). 

\begin{corollary}\label{coro:cantor}
For $m>3+(2/p)$ and for any weight $\delta+(2/p)\in(0,1)\cup(1,2)\cup(2,3)$ the 2d Euler equation is 
globally well-posed in the weighted Sobolev space $W^{m,p}_\delta$.
\end{corollary}

We refer to Corollary \ref{coro:weighted_spaces} below for a generalization of this statement for any weight $\delta+(2/p)>3$
with $\delta+(2/p)\notin\Z$. We expect that Corollary \ref{coro:cantor} holds also for $\delta+(2/p)\in\{1,2\}$ but the proof
will require additional considerations. 
Let $(\cdot,\cdot)$ be the pairing $(u_1,u_2)=\int_{\R^2} u_1 u_2\,dx dy$ of complex valued functions on $\R^2$, and let
$\divop v$ and $\curlop v$ be the divergence and the curl of a (real) vector field $v$ on $\R^2$. 
The following proposition follows from Proposition \ref{prop:integrals} and 
Lemma \ref{lem:a_kl-derivatives} (cf. \cite[Example 2, Appendix B]{McOTo3}).

\begin{proposition}\label{prop:evolution_coefficients}
Assume that $m>3+(2/p)$. Then we have
\begin{itemize}
\item[(i)] For $0\le n\le 3$ the  asymptotic space $\mathcal{Z}^{m,p}_{n,N}$, $n\le N+1$, is invariant with respect to the 2d Euler equation. 
\item[(ii)] For any $u_0\in\mathcal{Z}^{m,p}_{n,N}$, $3<n\le N+1$, the solution
$u\in C\big([0,\infty), \mathcal{Z}^{m,p}_N\big)\cap C^1\big([0,\infty),\mathcal{Z}^{m-1,p}_N\big)$ of the 2d Euler equation
given by Theorem \ref{th:main_introduction} is of the form
\begin{equation}\label{eq:no_decay}
u(t)={\chi}\sum_{3\le k\le n-1}\frac{a_{0k}(t)}{\bar{z}^k}+g(t),\quad g(t)\in\mathcal{Z}^{m,p}_{n,N},
\end{equation}
where $a_{0k}(t)=\frac{1}{\pi}\big(\omega_0,\overline{\varphi}(t)^{k-1}\big)$, $3\le k\le n$, $\omega_0:=i \curlop u_0$
is the initial (complex) vorticity\footnote{Slightly abusing the notation, we will call $\omega:=i \curlop u$ the (complex) vorticity of $u$.}, 
and $\varphi\in C^1\big([0,\infty),\mathcal{ZD}^{m,p}_N\big)$ is the solution of the equation
$\dt\varphi=u\circ\varphi$, $\varphi|_{t=0}=\idmap$ (cf. Lemma \ref{lem:ode} in Section \ref{sec:cauchy_operator}). 
\item[(iii)] For any $3<n\le N+1$ there exists an open {\em dense} set $\mathcal{N}_{n,N}$ inside the closed linear space 
$\mathcal{\accentset{\circ}{Z}}^{m,p}_{n,N}:=\big\{w\in \mathcal{Z}^{m,p}_{n,N}\,\big|\,\divop w=0\big\}$ of divergence 
free vector fields in $\mathcal{Z}^{m,p}_{n,N}$ such that for any $u_0\in\mathcal{N}_{n,N}$ and for any $3\le k\le n$ the coefficient 
$a_{0k}(t)$ in \eqref{eq:no_decay} does {\em not} vanish for almost any $t>0$.
\end{itemize}
\end{proposition}

\begin{remark}
It follows from Proposition \ref{prop:evolution_coefficients} (ii), (iii), that the asymptotic space 
$\mathcal{Z}^{m,p}_{n,N}$ with $3<n\le N+1$ is {\em not} invariant with respect to the 2d Euler equation. 
The elements of $\mathcal{Z}^{m,p}_{2,N}$, $N\ge 1$, have finite kinetic energy ${\rm E}(u)\equiv\frac{1}{2}(u,\bar{u})$. 
By Proposition \ref{prop:integrals} and Proposition \ref{prop:evolution_coefficients} the asymptotic space 
$\mathcal{Z}^{m,p}_{2,N}$ is invariant with respect to the 2d Euler equation and for $N\ge 2$ has a constant in time leading asymptotic
coefficient $a_{02}(t)=a_{02}(0)$.
\end{remark}

By taking $n=N+1$ in Proposition \ref{prop:evolution_coefficients} and then by using that 
$\mathcal{Z}^{m,p}_{n,N}\equiv W^{m,p}_{\gamma_N}$ we obtain from  Proposition \ref{prop:evolution_coefficients} $(ii)$, $(iii)$,
the following generalization of Corollary \ref{coro:cantor}.

\begin{corollary}\label{coro:weighted_spaces}
For any $\delta+(2/p)>3$ such that $\delta+(2/p)\notin\Z$ and for any
initial data $u_0\in W^{m,p}_\delta$ there exists a unique solution 
$u\in C\big([0,\infty), \mathcal{Z}^{m,p}_{N_\delta}\big)\cap C^1\big([0,\infty),\mathcal{Z}^{m-1,p}_{N_\delta}\big)$
of the 2d Euler equation such that for any $t\in[0,\infty)$ we have that
\[
u(t)={\chi}\sum_{3\le k\le N_\delta}\frac{a_{0k}(t)}{\bar{z}^k}+f(t),\quad f(t)\in W^{m,p}_\delta,
\]
where $N_\delta$ is the integer part of $\delta+(2/p)$.
The solution depends continuously on the initial data $u_0\in W^{m,p}_\delta$ in the sense of 
Theorem \ref{th:main_introduction}. In particular, the weighted Sobolev spaces $W^{m,p}_\delta$ are {\em not} invariant
with respect to the 2d Euler equation if $\delta+(2/p)>3$ and $\delta+(2/p)\notin\Z$.
\end{corollary}

\medskip

Another application of Theorem \ref{th:main_introduction} is the global well-posedness of the 2d Euler equation in the Fr\'echet 
space of {\em ``symbols''} $\mathcal{Z}^\infty$ (cf. \cite{BS,KPST}). More specifically, denote by $\mathcal{Z}^\infty$ the space of
$C^\infty$-smooth vector fields $u$ on $\R^2$ that have an asymptotic expansion at infinity of infinite order
\begin{equation}\label{eq:infinite_asymptotic}
u(z,\bar{z})\sim\sum_{k+l\ge 0}\frac{a_{kl}}{z^k\bar{z}^l}\quad\text{as}\quad|z|\to\infty.
\end{equation}
The meaning of this asymptotic formula is that for any truncation number $N\ge 0$ and for any multi-index $\alpha\in\Z^2_{\ge 0}$
there exists a constant $C_{N,\alpha}>0$ such that 
\begin{equation}\label{eq:infinite_asymptotic'}
\sup_{z\in\C}\Big|\z^{N+1+|\alpha|}\,\partial^\alpha\Big(u-\chi\sum_{0\le k+l\le N}\frac{a_{kl}}{z^k\bar{z}^l}\Big)\Big|
\le C_{N,\alpha},\quad\partial^\alpha\equiv\partial_z^{\alpha_1}\partial_{\bar{z}}^{\alpha_2}.
\end{equation}
The conditions in \eqref{eq:infinite_asymptotic'} give a set of semi-norms on the space of $C^\infty$-smooth vector fields on $\R^2$
that define a Fr\'echet topology on $\mathcal{Z}^\infty$. The following theorem follows from Theorem \ref{th:main_introduction}, 
Proposition \ref{prop:main_technical} (ii) and Remark \ref{rem:R^d} in Appendix \ref{appendix:technical}.

\begin{theorem}\label{th:Z^infty}
The 2d Euler equation is globally well-posed in the Fr\'echet space $\mathcal{Z}^\infty$ of $C^\infty$-smooth vector fields
allowing asymptotic expansion at infinity of infinite order. In particular, the solution depends continuously on the initial data
$u_0\in\mathcal{Z}^\infty$ in the sense that for any $T>0$ the data-to-solution map  $u_0\mapsto u$,
$\mathcal{Z}^\infty\to C^1\big([0,T],\mathcal{Z}^\infty\big)$, is continuous. The asymptotic coefficients in the
infinite asymptotic expansion \eqref{eq:infinite_asymptotic} extend to holomorphic functions in $t$ in open neighborhoods of
$[0,\infty)$ in $\C$.
\end{theorem}

Let us briefly discuss this result. If we pick an initial data $u_0\in\mathcal{Z}^\infty$ that is a finite sum of asymptotic terms
$u_0=\chi\sum_{n\le k+l\le N}\frac{a_{kl}^0}{z^k\bar{z}^l}$, $0\le n\le N$, with non-vanishing coefficients, then, by Theorem \ref{th:Z^infty}, 
the corresponding solution of the 2d Euler equation will lie in the space of symbols $\mathcal{Z}^\infty$ for any $t\ge 0$. In particular,
the solution will satisfy the infinite set of decay conditions \eqref{eq:infinite_asymptotic'} and we will have that
$a_{kl}(t)\ne 0$ for any $n\le k+l\le N$ and for almost any $t\ge 0$. By Proposition \ref{prop:evolution_coefficients} $(ii)$, for $n>3$ 
the solution can develop lower order asymptotics of the form $\sum_{3\le k\le n-1}\frac{a_{0k}(t)}{\bar{z}^k}$. 
Moreover, by the last statement of Lemma \ref{lem:a_kl-derivatives} in Section \ref{sec:euler_equation}, we can perturb $u_0$
by an element of $C^\infty_c$ with support concentrated in a neighborhood of zero so that such lower order asymptotic terms appear 
and we have that $a_{0k}(t)\ne 0$ for $3\le k\le n-1$ and for almost any $t>0$. Note also that for any infinite sequence of coefficients 
$a_{kl}^0$, $k+l\ge 0$, that satisfy a divergence free condition (see Remark \ref{rem:divergence_free_asymptotics}) we can construct a 
divergence free $u_0\in\mathcal{Z}^\infty$ such that $u_0\sim\sum_{k+l\ge 0}\frac{a_{kl}^0}{z^k\bar{z}^l}$ by using arguments similar to
those used to prove the classical Borel theorem (see e.g. \cite[\S 3.3]{AG}, \cite[Proposition 3.5]{Subin} as well as the proof of 
Lemma \ref{lem:a_kl-derivatives} in Section \ref{sec:euler_equation}). For any such $u_0$ we obtain a unique solution of the 2d Euler equation 
in $\mathcal{Z}^\infty$ that depends continuously on the initial data in $\mathcal{Z}^\infty$. The following corollary 
(see Section \ref{sec:euler_equation} for the proof) shows that {\em infinite} number of spatial asymptotics terms of the solutions of the 
2d Euler equation appear generically from initial data in the Schwartz space.

\begin{corollary}\label{coro:asymptotics_pop-up}
There exists a {\em dense} set $\mathcal{N}$ inside the closed linear space of divergence 
free vector fields $\accentset{\circ}S:=\big\{w\in S\big(\R^2,\R^2\big)\,\big|\,\divop w=0\big\}$ in the Schwartz space 
$S\big(\R^2,\R^2\big)$ such that for any $u_0\in\mathcal{N}$ the solution $u\in C^1\big([0,\infty),\mathcal{Z}^\infty\big)$ 
given by Theorem \ref{th:Z^infty} satisfies $u(t)\sim\sum_{k\ge 3}\frac{a_{0k}(t)}{\bar{z}^k}$ where for any $k\ge 3$ we have that 
$a_{0k}(t)\ne 0$ for almost any $t>0$.
\end{corollary}

Note that Theorem \ref{th:main_introduction} also implies that for any $u_0\in\mathcal{Z}^{m,p}_N$ and for any $T>0$ 
the asymptotic part $\sum_{0\le k+l\le N}\frac{a_{kl}(t)}{z^k\bar{z}^l}$ of the solution \eqref{eq:u_expansion} 
of the 2d Euler equation {\em approximates} the solution 
$u\in C\big([0,\infty),\mathcal{Z}^{m,p}_N\big)\cap C^1\big([0,\infty),\mathcal{Z}^{m-1,p}_N\big)$
uniformly for $t\in [0,T]$ and $z$ in an open neighborhood of infinity. 
In fact, assume that $m>3+(2/p)$ and $N\ge 0$. 
Then, by the compactness of the interval $[0,T]$ and the continuity of the remainder 
$f : [0,T]\to W^{m,p}_{\gamma_N}$ in \eqref{eq:u_expansion} we conclude that
there exists $C_T>0$ such that $\|f(t)\|_{W^{m,p}_{\gamma_N}}<C_T$ for any $t\in[0,T]$.
By combining this with Proposition \ref{prop:main_technical} (ii) and Remark \ref{rem:R^d} in 
Appendix \ref{appendix:technical}, we then conclude that for any multi-index $\alpha\in\Z_{\ge 0}$,
$0\le|\alpha|<m-(2/p)$, we have that for any $R\ge 2$,
\begin{equation}\label{eq:approximation_at_infinity}
\Big|\partial^\alpha\Big(u(t)-\sum_{0\le k+l\le N}\frac{a_{kl}(t)}{z^k\bar{z}^l}\Big)\Big|\le
\left(\frac{C_T}{R^{\gamma_0+2/p}}\right)\left(\frac{1}{\z}\right)^{N+|\alpha|},
\end{equation}
uniformly for $t\in[0,T]$ and for $z\in\C$ such that $|z|\ge R$ and $R\ge 2$. Moreover, this formula
holds locally uniformly on the choice of $u_0\in\mathcal{Z}^{m,p}_N$. The question whether a long time version of
\eqref{eq:approximation_at_infinity} holds is open.

Multi dimensional analogs of the asymptotic spaces $\mathcal{Z}^{m,p}_N$ as well as local in time versions of the results above 
can also be proved but they require various technical restrictions on the spherical modes in the asymptotic spaces and will be discussed separately.

\medskip

\noindent{\em Related work.} There is a vast literature on the well-posedness of the incompressible Euler equation 
(see e.g.\ \cite{BardosTiti,Kato1} and the references therein). The existence of solution of the Euler equations in 
weighted Sobolev space was studied in \cite{Cantor}. Solutions of the KdV and mKdV in symbol classes on the line
are considered in \cite{BS,KPST}. Various asymptotic spaces and groups of asymptotic diffeomorphisms were defined 
and studied in \cite{McOTo1,McOTo2}. The local existence of solutions of the Euler equation in asymptotic spaces involving log terms 
was established in \cite{McOTo3}. Finally, note that specific spatial decay and even certain asymptotic expansions are known 
to naturally appear for the solutions of the Navier-Stokes equation with {\em rapidly decreasing} initial data 
(see e.g.\ \cite{DobrShaf,BranMe,Bran,KR} and the references therein).

\medskip

\noindent{\em Organization of the paper.} The paper is organized as follows. In Section \ref{sec:spaces} we define a class 
of diffeomorphisms $\mathcal{ZD}^{m,p}_N$ of $\R^2$ and prove that they form a topological group with Banach manifold 
structure modeled on $\mathcal{Z}^{m,p}_N$. The group $\mathcal{ZD}^{m,p}_N$ has additional regularity properties summarized 
in Theorem \ref{th:ZD}. At the end of the section we prove that there exists a natural real analytic homomorphism
$\hrho : \mathcal{ZD}^{m,p}_N\to\hA$ onto a finite dimensional Lie group $\hA$. 
In Section \ref{sec:cauchy_operator} we study in detail the properties of the Cauchy operators $\partial_z$ when acting
on the scale of spaces $\mathcal{Z}^{m,p}_N$. In particular, we prove Theorem \ref{th:Z->Z-tilde} that
characterizes the image of the Cauchy operator in this scale. The local existence and uniqueness of solutions 
of the 2d Euler equation in $\mathcal{Z}^{m,p}_N$, as well as the statements characterizing the behavior of
the asymptotic coefficients are proven in Section \ref{sec:euler_equation}. The global existence is proven in 
Section \ref{sec:global_solutions}. We conclude the paper with two appendices where we prove several technical results 
used in the main body of the paper.

\medskip

\noindent{\em Acknowledgments}: The second author is grateful to IMI BAS.

\section{Asymptotic spaces and maps} \label{sec:spaces}
In this section we study the properties of the Banach algebra $\mathcal{Z}^{m,p}_N$ and an associated group of diffeomorphisms 
$\mathcal{ZD}^{m,p}_N$ of $\R^2$. In particular, we prove in Theorem \ref{th:ZD} that $\mathcal{ZD}^{m,p}_N$
is a topological group with additional regularity properties of the composition and the inversion of the diffeomorphisms.
At the end of the section we construct a real analytic homomorphism of $\mathcal{ZD}^{m,p}_N$ onto a finite dimensional Lie 
group $\hA$ that is used in Section \ref{sec:euler_equation} to prove the analyticity of the asymptotic coefficients of the solutions 
of the 2d Euler equation given by Theorem \ref{th:main_introduction}.

The following two propositions summarize the properties of the asymptotic spaces $\mathcal{Z}_N^{m,p}$ and 
their reminder spaces $W^{m,p}_\delta$. For the proof we refer to Proposition \ref{prop:main_technical},
Proposition \ref{prop:W_R}, and Remark \ref{rem:R^d} in Appendix \ref{appendix:technical} (cf. also \cite[\S\,2]{McOTo2}). 

\begin{proposition}\label{prop:properties_W-spaces}
\begin{itemize}
\item[(i)] For any $0\le m_1\le m_2 $ and for any real weights $\delta_1\le\delta_2$ we have that
the inclusion map $W^{m_2,p}_{\delta_2}\to W^{m_1,p}_{\delta_1}$ is bounded.
\item[(ii)] For any $m\ge 0$, $\delta\in\R$, and for any $k\in\Z$ the map
$f\mapsto f\cdot\frac{\chi}{z^k}$, $W^{m,p}_\delta\to W^{m,p}_{\delta+k}$ 
is bounded. The same holds with $z$ replaced by $\bar{z}$.
Here $\chi(z,\bar{z})\equiv\chi(|z|)$ is a $C^\infty$-smooth cut-off function such that $\chi : \R\to\R$ and 
$\chi|_{(-R,R)}\equiv 1$ for some $R>0$.
\item[(iii)] For any regularity exponent $m\ge 0$, $\delta\in\R$, and a multi-index $\alpha\in\Z_{\ge 0}^2$
such that $|\alpha|\le m$ the map
$\partial^\alpha: W^{m,p}_\delta\to W^{m-|\alpha|,p}_{\delta+|\alpha|}$ is bounded.
\item[(iv)] For $m>2/p$ and for any weights $\delta_1,\delta_2\in\R$ the map 
$(f,g)\mapsto fg$, $W^{m,p}_{\delta_1}\times W^{m,p}_{\delta_2}\to W^{m,p}_{\delta_1+\delta_2+(2/p)}$
is bounded. In particular, for $\delta\ge 0$ the space $W^{m,p}_\delta$ is a Banach algebra.
\end{itemize}
\end{proposition}

Proposition \ref{prop:properties_W-spaces} implies

\begin{proposition}\label{prop:properties_Z-spaces}
Assume that $m>2/p$. Then
\begin{itemize}
\item[(i)] For any $0\le n_1\le n_2\le N_1\le N_2$ we have that the inclusion map 
$\mathcal{Z}^{m,p}_{n_2,N_2}\to\mathcal{Z}^{m,p}_{n_1,N_1}$ is bounded.
\item[(ii)] For any $0\le n\le N$, an integer $k\ge 0$, a regularity exponent $m>k+(2/p)$, and a multi-index $|\alpha|\le k$ the map 
$\partial^\alpha : \mathcal{Z}_{n,N}^{m,p}\to\mathcal{Z}_{n+|\alpha|,N+|\alpha|}^{m-|\alpha|,p}$
is bounded. 
\item[(iii)] For $m>2/p$ and for any $0\le n_1\le N_1$ and $0\le n_2\le N_2$ the map 
\[
(f,g)\mapsto fg,\quad\mathcal{Z}^{m,p}_{n_1,N_1}\times\mathcal{Z}^{m,p}_{n_2,N_2}\to\mathcal{Z}^{m,p}_{n,N},
\]
where $n:=n_1+n_2$ and $N:=\min(n_1+N_2,n_2+N_1)$ is bounded. 
In particular, $\mathcal{Z}^{m,p}_N$ is a Banach algebra for any $N\ge 0$.
\end{itemize}
\end{proposition}

For a given $R>0$ and $m>2/p$ define the auxiliary space 
\begin{align}\label{eq:Z_R-space}
\mathcal{Z}_N^{m,p}(B_R^c):=\Big\{\sum_{0\le k+l\le N}\frac{a_{kl}}{z^k\bar{z}^l} + 
f \; \Bigl\vert \; f \in W_{\gamma_N}^{m,p}(B_R^c)\,\,\text{\rm and}\,\,a_{kl}\in\mathbb{C} \Big\},
\end{align}
where $B^c_R := \big\{z\in\C\big| \, |z| > R\big\}$ and $W_{\gamma_N}^{m,p}(B_R^c)$
denotes the weighted Sobolev space,
\[
W^{m,p}_\delta(B_R^c):=\big\{f\in H^{m,p}_{loc}(B_R^c,\C)\,\big|\,
|z|^{\delta+|\alpha|}\partial^\alpha f\in L^p\,\,\text{\rm for}\,\,|\alpha|\le m\big\},\quad
\partial^\alpha\equiv\partial_z^{\alpha_1}\partial_{\bar z}^{\alpha_2},
\]
equipped with the norm
$\|u\|_{W^{m,p}_\delta(B_R^c)}:=\sum_{|\alpha|\le m}\big\||z|^{\delta+|\alpha|}\partial^\alpha f\big\|_{L^p(B_R^c)}$.
We equip $\mathcal{Z}_N^{m,p}(B_R^c)$ with the norm
$\norm{u}_{\mathcal{Z}^{m,p}_N(B_R^c)}:=\sum_{0\le k+l\le N} \abnorm{a_{kl}} + 
\norm{f}_{W^{m,p}_{\gamma_N}(B_R^c)}$. (If $m=0$ we set $L^p_\delta\equiv W^{0,p}_\delta$.)
In addition to this norm we will also consider the equivalent norm on $\mathcal{Z}_N^{m,p}(B_R^c)$
(see Appendix \ref{appendix:technical} for detail)
\begin{align}\label{eq:ZR-norm}
|u|_{\mathcal{Z}^{m,p}_N(B_R^c)} := \sum_{0\le k+l \le N} \frac{|a_{kl}|}{R^{k+l}} + 
\norm{f}_{W^{m,p}_{\gamma_N}(B_R^c)}.
\end{align}
For $R>0$ denote $B_R:=\big\{z\in\C\,\big|\,|z|<R\big\}$.
The proof of the following lemma is straightforward.

\begin{lemma}\label{lem:break}
Assume that $R>0$ and $m>2/p$. Then, for $u\in H^{m,p}_{loc}(\C,\C)$ we have that 
$u \in \mathcal{Z}^{m,p}_N$ if and only if $u \in H^{m,p}(B_{R+1})$ and $u\in\mathcal{Z}^{m,p}_N(B^c_R)$. 
The norms $\|\cdot\|_{\mathcal{Z}^{m,p}_N}$ and
$\|\cdot\|_{H^{m,p}(B_{R+1})}+|\cdot|_{\mathcal{Z}^{m,p}_N(B^c_R)}$ on $\mathcal{Z}^{m,p}_N$
are equivalent.\footnote{For simplicity of notation we use the same symbol for a function $u$ and its restriction to an open
subset of its domain of definition.}
\end{lemma}

For $m>2/p$ the space $\mathcal{Z}^{m,p}_N(B_R^c)$ is a Banach algebra with respect to pointwise multiplication of
complex-valued functions. Moreover, this space has the following uniform in $R\ge 1$ multiplicative property
that is proved in Appendix \ref{appendix:technical}.

\begin{lemma}\label{coro:strict_banach}
For $m>2/p$ there exists a constant $C>0$ independent of the choice of $R\ge 1$ such that for 
any $u_1,u_2 \in \mathcal{Z}^{m,p}_N(B_R^c)$ we have
$\abnorm{u_1 \cdot u_2}_{\mathcal{Z}^{m,p}_N(B_R^c)} \le C \; 
\abnorm{u_1}_{\mathcal{Z}^{m,p}_N(B_R^c)} \abnorm{u_2}_{\mathcal{Z}^{m,p}_N(B_R^c)}$.
\end{lemma}

The notion of a Banach algebra we adopt in this paper does {\em not} require that the constant $C$ above is equal to one. 
We will also need the following lemma.

\begin{lemma}\label{lem:division}
Assume that $v\in\mathcal{Z}^{m,p}_N$ with $m>2/p$ has a non-vanishing leading asymptotic term $a_{00}\ne 0$ and 
that $v(z,\bar{z})\ne 0$ for any $z\in\C$. 
Then $1/v\in\mathcal{Z}^{m,p}_N$ and there exists an open neighborhood $U$ of 
zero in $\mathcal{Z}^{m,p}_N$ such that the map
\[
u\mapsto 1/(v+u),\quad U\to\mathcal{Z}^{m,p}_N,
\] 
is analytic. 
\end{lemma}

\begin{proof}[Proof of Lemma \ref{lem:division}]
We will assume without loss of generality that the leading asymptotic term $a_{00}$ of $v\in\mathcal{Z}^{m,p}_N$ is 
equal to one. Then, by \eqref{eq:ZR-norm} (cf. Lemma \ref{coro:norms} $(i)$) we can choose $R\ge 2$ large enough so that 
$|v-1|_{\mathcal{Z}^{m,p}_N(B_R^c)}<\min\big(1/2,1/(2C)\big)$ 
where $C>0$ is the constant appearing in Lemma \ref{coro:strict_banach}. 
Then we have,
\[
1/v=1/\big(1-(1-v)\big)=\sum_{k\ge 0}(1-v)^k,
\]
where the series converges in $\mathcal{Z}^{m,p}_N(B_R^c)$. The series also converges pointwisely by Lemma \ref{coro:norms} $(ii)$. 
In particular, we see that $1/v\in\mathcal{Z}^{m,p}_N(B_R^c)$. The condition that
$v\in H^{m,p}(B_{R+1})$ with $m>d/p$ and the fact that $v(z,\bar{z})\ne 0$ for any $z\in\C$ then imply that 
$1/v\in H^{m,p}(B_{R+1})$. Hence, $1/v\in\mathcal{Z}^{m,p}_N$ by Lemma \ref{lem:break}.
Since the space $\mathcal{Z}^{m,p}_N$ is a Banach algebra, it satisfies an analog of Lemma \ref{coro:strict_banach}
with an appropriately chosen constant $C>0$. By this Banach algebra property of $\mathcal{Z}^{m,p}_N$ and the continuity of the embedding
$\mathcal{Z}^{m,p}_N\subseteq L^\infty$ (see Lemma \ref{coro:norms} $(iii)$ in Appendix \ref{appendix:technical}) we can find an open 
neighborhood $U$ of zero in $\mathcal{Z}^{m,p}_N$ such that $\|u/v\|_{\mathcal{Z}^{m,p}_N}<1/(2C)$ and 
$\|u/v\|_{L^\infty}<1/2$ for any $u\in U$. Then, we have that
\[
1/(v+u)=(1/v)\sum_{k\ge 0}(-1)^k(u/v)^k
\]
where the series converges in $\mathcal{Z}^{m,p}_N$ uniformly in $U$. This completes the proof of the lemma. 
\end{proof}

For $m>1+(2/p)$ consider the set of maps $\R^2\to\R^2$,
\begin{equation}\label{eq:Z-group}
\mathcal{ZD}_N^{m,p}:=\Big\{\varphi\in\mathop{\rm Diff_+^1}(\R^2)\;\Big|\;
\varphi=\idmap+u,\;u\in\mathcal{Z}_N^{m,p}\Big\},
\end{equation} 
where $\mathop{\rm Diff_+^1}(\R^2)$ denotes the group of $C^1$-smooth orientation preserving 
diffeomorphisms of $\R^2$. As described in the Introduction, we identify $\R^2$ with the complex plane $\C$ and 
by following the standard notation in complex analysis write $\varphi(z,\bar{z})=z+u(z,\bar{z})$ for $z\in\C$. 
It follows easily from Hadamard-Levy's theorem (see e.g. \cite[Supplement 2.5D]{AMRBook}) and the decay properties of 
the elements of the Banach algebra $\mathcal{Z}_N^{m,p}$ that the set of maps $\mathcal{ZD}_N^{m,p}$ in \eqref{eq:Z-group} 
can be identified with the open set $\mathcal{O}$ in $\mathcal{Z}_N^{m,p}$ of those $u\in\mathcal{Z}_N^{m,p}$ such that there exists 
$\varepsilon>0$ so that uniformly on $\R^2$,
\begin{equation}\label{eq:open_condition}
\det(I+du)>\varepsilon,
\end{equation} 
where $I$ is the identity $2\times 2$ matrix and $du$ denotes the Jacobian matrix of the map $u : \R^2\to\R^2$. 
In fact, for any given $C^1$-diffeomorphism $\varphi\in\mathcal{ZD}_N^{m,p}$, $\varphi=z+u$, $u\in\mathcal{Z}_N^{m,p}$, 
we obtain from Proposition \ref{prop:properties_Z-spaces} that $du\in\mathcal{Z}_{1,N+1}^{m-1,p}$, and hence,
the entries of $du$ are continuous functions of order $O(1/r)$ as $r\to\infty$ (cf. Proposition \ref{prop:main_technical} and 
Remark \ref{rem:R^d} in Appendix \ref{appendix:technical}). This, together with the fact that $\varphi$ is an orientation preserving
$C^1$-diffeomorphism, then implies that there exists $\varepsilon>0$ such that \eqref{eq:open_condition} holds.
By the continuity of the differentiation $d : \mathcal{Z}_N^{m,p}\to\mathcal{Z}_{1,N+1}^{m-1,p}$, the Banach algebra property
of $\mathcal{Z}_{1,N+1}^{m-1,p}$, and the continuity of the embedding $\mathcal{Z}_{1,N+1}^{m-1,p}\subseteq L^\infty$
(Lemma \ref{coro:norms} $(iii)$) one also concludes that \eqref{eq:open_condition} holds locally uniformly in $u\in\mathcal{Z}_N^{m,p}$. 
This implies that $\mathcal{O}$ is an open set in $\mathcal{Z}_N^{m,p}$ and that $u\in\mathcal{O}$. 
Conversely, if $\varphi=z+u$ with $u\in\mathcal{O}$ we conclude from \eqref{eq:open_condition} and Hadamard-Levy's theorem that 
$\varphi=z+u$ is a $C^1$-diffeomorphism of the plane $\R^2$ and hence $\varphi\in\mathcal{ZD}_N^{m,p}$. This implies that 
$\mathcal{ZD}_N^{m,p}$ is a Banach manifold modeled on $\mathcal{Z}_N^{m,p}$ where $\mathcal{O}$ is the only chart and 
$u\in\mathcal{O}\subseteq\mathcal{Z}_N^{m,p}$ is the coordinate of $\varphi\in\mathcal{Z}_N^{m,p}$.
Note that by Remark \ref{rem:real_components} the set $\mathcal{ZD}_N^{m,p}$ is a closed submanifold in 
the group of asymptotic diffeomorphisms
\begin{equation*}
\mathcal{AD}_N^{m,p}:=\Big\{\varphi\in\mathop{\rm Diff_+^1}(\R^2)\;\Big|\;
\varphi =\idmap+u,\;u\in\mathcal{A}_N^{m,p}\Big\}
\end{equation*} 
considered in \cite[Section 2]{McOTo3}. Recall that the elements of $\mathcal{A}_N^{m,p}$ are vector fields on $\R^2$ of the form
\eqref{eq:asymptotic_expansion_logs} with vanishing coefficients in front of the log terms and that
the Banach manifold structure on $\mathcal{AD}_N^{m,p}$ is introduced in a fashion similar to the one for $\mathcal{ZD}_N^{m,p}$.

Consider the auxiliary space

\begin{equation}\label{Z_bar-space}
\widetilde{\mathcal{Z}}_N^{m,p}:=\Big\{\frac{\chi}{z^2} 
\sum_{0\le k+l\le N-2}\frac{a_{kl}}{z^k\bar{z}^l} + f\; 
\Bigl\vert \; f \in W_{\gamma_N}^{m,p} \,\,\text{\rm and}\,\,a_{kl}\in\mathbb{C}\Big\}
\end{equation}
where $\chi\equiv\chi(|z|)$ and we set that $\widetilde{\mathcal{Z}}_1^{m,p}\equiv W^{m,p}_{\gamma_1}$.
Note that $\widetilde{\mathcal{Z}}_N^{m,p}$ is a closed ideal in the Banach algebra $\mathcal{Z}_N^{m,p}$.
We have 

\begin{lemma}\label{lem:composition}
Assume that $m>1+(2/p)$. Then for any $\varphi\in\mathcal{ZD}^{m,p}_N$ and for any $u\in\mathcal{Z}^{m,p}_N$ we have that
$u\circ\varphi\in\mathcal{Z}^{m,p}_N$.
Similarly, for any $\varphi\in\mathcal{ZD}^{m,p}_N$ and for any $u\in\widetilde{\mathcal{Z}}^{m,p}_N$ we have that
$u\circ\varphi\in\widetilde{\mathcal{Z}}^{m,p}_N$.
\end{lemma}

\begin{proof}[Proof of Lemma \ref{lem:composition}]
Let $\varphi\in\mathcal{ZD}^{m,p}_N$ where $\varphi = z + v$ with $v\in\mathcal{Z}^{m,p}_N$.
First, assume that $u = \chi/z$. Then we have
\begin{equation}\label{eq:expansion1}
u \circ \varphi = \chi\circ\varphi\cdot\frac{1}{z+v} = \chi\circ\varphi\cdot\frac{1}{z}\cdot
\frac{1}{\left(1+\frac{v}{z}\right)}.
\end{equation}
By taking $R>0$ large enough we can ensure that $\chi\circ\varphi\equiv 1$ on $B_R^c$
and $|v/z|_{\mathcal{Z}^{m,p}_N(B_R^c)}<\max\big(1/2,1/(2C)\big)$ (see Lemma \ref{coro:norms}).
Then, by arguing as in the proof of Lemma \ref{lem:division}, we conclude that the term $1/\left(1+\frac{v}{z}\right)$ 
in \eqref{eq:expansion1} belongs to $\mathcal{Z}^{m,p}_N(B_R^c)$. Hence, by \eqref{eq:expansion1}, 
$u\circ\varphi\in\mathcal{Z}^{m,p}_N(B_R^c)$. Since $\varphi$ is a $C^1$-diffeomorphism of the complex plane $\C=\R^2$,
the expression $\varphi=z+v$ does not vanish on the support of $\chi\circ\varphi$. This shows that 
$u\circ\varphi=\chi\circ\varphi/(z+v)$ belongs to $H^{m,p}(B_{R+1})$ since $\chi\circ\varphi=1-(1-\chi)\circ\varphi\in H^{m,p}_{loc}$
by Corollary 6.1 (b) in \cite{McOTo2}. Hence, $u\circ\varphi\in\mathcal{Z}^{m,p}_N$ by Lemma \ref{lem:break}. 
By arguing in the same way one also treats the case when $u=\chi/\bar{z}$. Since $\mathcal{Z}^{m,p}_N$ is an algebra,
the statement of the lemma holds for $u=\chi/z^k$ and $u=\chi/\bar{z}^k$ with $k\ge 1$.
Finally, if $u$ belongs to the remainder space $W^{m,p}_{\gamma_N}$ then $u\circ \varphi \in W^{m,p}_{\gamma_N}$ 
by Corollary 6.1 (b) in \cite{McOTo2} and the fact that $m>1+(2/p)$. This completes the proof of the lemma for the space 
$\mathcal{Z}^{m,p}_N$. The case when $u$ belongs to the auxiliary space $\widetilde{\mathcal{Z}}^{m,p}_N$ follows easily 
by the same arguments.
\end{proof}

\begin{remark}\label{rem:improved_composition}
Formula \eqref{eq:expansion1} implies that for $u=\chi/z$ and for any $\varphi\in\mathcal{ZD}^{m,p}_N$ we have that
\begin{equation}\label{eq:improved_composition}
u\circ\varphi=\frac{\chi}{z}\big(1+f\big),\quad f\in\mathcal{Z}^{m,p}_{1,N+1},
\end{equation}
and a similar formula for $u=\chi/\bar{z}$. This, together with Proposition \ref{prop:properties_Z-spaces} and 
Lemma 6.5 in \cite{McOTo2}, then shows that the following stronger statement holds: for any $\varphi\in\mathcal{ZD}^{m,p}_N$ and 
for any $u\in\mathcal{Z}^{m,p}_{n,N+1}$, $0\le n\le N+2$ we have that $u\circ\varphi\in\mathcal{Z}^{m,p}_{n,N+1}$.
\end{remark}

As a corollary we obtain the following proposition.

\begin{proposition}\label{prop:composition}
For $m>1+(2/p)$ the set $\mathcal{ZD}^{m,p}_N$ is closed under the composition of diffeomorphisms.
In addition, the composition map
$(\varphi,\psi)\mapsto\varphi\circ\psi$, $\mathcal{ZD}^{m,p}_N\times\mathcal{ZD}^{m,p}_N\to\mathcal{ZD}^{m,p}_N$,
is continuous and the associated map
\[
(\varphi,\psi)\mapsto\varphi\circ\psi,\quad\mathcal{ZD}^{m+1,p}_N\times\mathcal{ZD}^{m,p}_N\to\mathcal{ZD}^{m,p}_N,
\]
is $C^1$-smooth.
\end{proposition}

\begin{remark}\label{rem:composition}
Similarly, for $m>1+(2/p)$ the composition map 
$(u,\psi)\mapsto u\circ\psi$, $\mathcal{Z}^{m,p}_N\times \mathcal{ZD}^{m,p}_N\to\mathcal{Z}^{m,p}_N$,
is continuous and the map
\[
(u,\psi)\mapsto u\circ\psi,\quad\mathcal{Z}^{m+1,p}_N\times \mathcal{ZD}^{m,p}_N\to\mathcal{Z}^{m,p}_N,
\]
is $C^1$-smooth.
\end{remark}

\begin{proof}[Proof of Proposition \ref{prop:composition}]
Take $\varphi,\psi\in\mathcal{ZD}^{m,p}_N$ where $\varphi=z+v$ and $\psi=z+w$ with 
$v,w\in \mathcal{Z}^{m,p}_N$.
Then, $\varphi\circ\psi=\psi+v\circ\psi = z+w+v\circ\psi=z+u$ where $u := w + v\circ \psi$.
By Lemma \ref{lem:composition}, $v\circ \psi\in{\mathcal{Z}}^{m,p}_N$. Hence, $u\in{\mathcal{Z}}^{m,p}_N$.
Since, the composition of $C^1$-diffeomorphisms is again a $C^1$-diffeomorphism we then conclude that 
$\varphi\circ\psi\in\mathcal{ZD}^{m,p}_N$.
The last statement of the proposition then follows from the analogous result for the asymptotic group $\mathcal{AD}^{m,p}_N$
(see \cite[Proposition 5.1]{McOTo2}) and the fact that ${\mathcal{ZD}}^{m,p}_N$ is a closed submanifold in 
$\mathcal{AD}^{m,p}_N$.
\end{proof}

The following three lemmas are needed for the proof that $\mathcal{ZD}^{m,p}_N$ is closed under the inversion of the diffeomorphisms.

\begin{lemma}\label{propInvNghd} 
Assume that $m>2+(2/p)$. Then there exists an open neighborhood $U$ of the identity $\idmap$ in $\mathcal{ZD}^{m,p}_N$
such that for any $\varphi\in U$ we have that $\varphi^{-1}\in\mathcal{ZD}^{m,p}_N$.
\end{lemma}

\begin{proof}[Proof of Lemma \ref{propInvNghd}]
This lemma follows directly from the regularity statement in Proposition \ref{prop:composition}, the inverse function theorem,
and a simple bootstrapping argument.
In fact, consider the map $F : \mathcal{ZD}^{m,p}_N\times\mathcal{ZD}^{m-1,p}_N\to\mathcal{ZD}^{m-1,p}_N$,
$F(\varphi,\psi):=\varphi\circ\psi$. Since $m>2+(2/p)$, we obtain from Proposition \ref{prop:composition} that
$F$ is a $C^1$-map. For any $\psi\in\mathcal{ZD}^{m-1,p}_N$ we have that $F(\idmap,\psi)=\psi$. This implies that
$(D_2F)(\idmap,\idmap)=\idmap_{\mathcal{Z}^{m-1,p}_N}$ where 
$(D_2F)(\idmap,\idmap)\in\mathcal{L}\big(\mathcal{Z}^{m-1,p}_N,\mathcal{Z}^{m-1,p}_N\big)$ denotes the partial derivative of 
$F$ with respect to the second argument at the point $(\idmap,\idmap)\in\mathcal{ZD}^{m,p}_N\times\mathcal{ZD}^{m-1,p}_N$.
Since $F(\idmap,\idmap)=\idmap$, we then conclude from the inverse function theorem that there exists an open neighborhood
$U$ of identity $\idmap$ in $\mathcal{ZD}^{m,p}_N$ and a $C^1$-map $I : U\to\mathcal{ZD}^{m-1,p}_N$ such that
for any $\varphi\in\mathcal{ZD}^{m,p}_N$ we have that $F(\varphi,I(\varphi))=\idmap$ and $I(\idmap)=\idmap$.
In particular, we see that $I(\varphi)=\varphi^{-1}\in\mathcal{ZD}^{m-1,p}_N$ for any $\varphi\in U\subseteq\mathcal{ZD}^{m,p}_N$.
A simple bootstrapping argument then shows that $\varphi^{-1}\in\mathcal{ZD}^{m,p}_N$.
In fact, by taking the pointwise derivative in the equality $\varphi^{-1}\circ\varphi=\idmap$ we easily obtain that
\begin{equation}\label{eq:relation1}
d(\varphi^{-1})=(d\varphi)^{-1}\circ\varphi^{-1}.
\end{equation}
Recall that $(d\varphi)^{-1}=\mathop{\rm Adj}(d\varphi)/\det(d\varphi)$ where $\mathop{\rm Adj}(d\varphi)$ is the transpose of 
the cofactor matrix of $d\varphi$,
\begin{equation}\label{eq:jacobian_relation}
d\varphi=
\begin{pmatrix}
\varphi_z&\varphi_{\bar{z}}\\
\overline{\varphi}_z&\overline{\varphi}_{\bar{z}}
\end{pmatrix}
=
\begin{pmatrix}
1+w_z&w_{\bar{z}}\\
\overline{w}_z&1+\overline{w}_{\bar{z}}
\end{pmatrix}
\end{equation}
and $\varphi=z+w$ with $w\in\mathcal{Z}^{m,p}_N$.
Since $\varphi\in\mathcal{ZD}^{m,p}_N$, we obtain from the Banach algebra property of $\mathcal{Z}^{m-1,p}_{N+1}$ that 
the matrix elements of $\mathop{\rm Adj}(d\varphi)$ belong to $\mathcal{Z}^{m-1,p}_{N+1}$. For simplicity of notation we will
write $\mathop{\rm Adj}(d\varphi)\in\mathcal{Z}^{m-1,p}_{N+1}$.
On the other side, we conclude from \eqref{eq:jacobian_relation} that $\det(d\varphi)=1+f$ where $f\in\mathcal{Z}^{m-1,p}_{N+1}$ has a  
vanishing leading asymptotic term. By Lemma \ref{lem:division} we then conclude that $1/\det(d\varphi)\in\mathcal{Z}^{m-1,p}_{N+1}$. 
This and the Banach algebra property then imply that $(d\varphi)^{-1}\in\mathcal{Z}^{m-1,p}_{N+1}$. Hence, by \eqref{eq:relation1} and 
Remark \ref{rem:improved_composition}, we conclude that $d(\varphi^{-1})\in\mathcal{Z}^{m-1,p}_{N+1}$.  
This, together with the fact that $\varphi^{-1}\in\mathcal{ZD}^{m-1,p}_N$, then implies that $\varphi^{-1}\in\mathcal{ZD}^{m,p}_N$.
\end{proof}

Although the set $\mathcal{ZD}^{m,p}_N$ is not path connected we have the following weaker version of path 
connectedness that is sufficient for our purposes.

\begin{lemma}\label{propPthConnCmptSup}
For $m>1+(2/p)$ and any $\varphi\in\mathcal{ZD}^{m,p}_N$ there exists a continuous path 
$\gamma : [0,1]\to\mathcal{ZD}^{m,p}_N$ such that $\gamma(0)=\idmap+f$ where $f$ has compact support and 
$\gamma(1)=\varphi$.
\end{lemma}

\begin{proof}[Proof of Lemma \ref{propPthConnCmptSup}]
The proof of this lemma follows the lines of the proof of Lemma 7.2 in \cite{McOTo2}.
In fact, take $\varphi\in\mathcal{ZD}^{m,p}_N$ such that $\varphi=z+u$ and 
$u=\sum_{0\le k+l\le N}\frac{a_{kl}}{z^k\bar{z}^l}+g$ with $g\in W^{m,p}_{\gamma_N}$.
Then, consider the deformation
\begin{equation}\label{eq:deformation}
\gamma(s)=z+\Big[u+(s-1)\,\Big(\sum_{0\le k+l\le N}\frac{a_{kl}}{z^k\bar{z}^l}+g\Big)\chi_R\Big],\quad s\in[0,1],
\end{equation}
of the diffeomorphism $\varphi\in\mathcal{ZD}^{m,p}_N$ where 
$\chi_R(x,y):=\chi\big(|z|/R\big)$ for $(x,y)\in\R^2$ and $R\ge 1$, and $\chi$ is the cut-off function appearing 
in the definition of $\mathcal{Z}^{m,p}_N$.
It is clear that $\gamma(1)=\varphi$ and $\gamma(0)=\idmap+f$ where $f\in H^{m,p}$ has compact support.
Hence, we will conclude the proof of the lemma if we show that $\gamma(s)$ is an orientation preserving 
$C^1$-diffeomorphism of $\R^2$ for any $s\in[0,1]$.  Since the Jacobian matrix of the vector field 
$(s-1)\Big(\sum_{k+l\le N}\frac{a_{kl}}{z^k\bar{z}^l}+g\Big)\chi_R\in\mathcal{Z}^{m,p}_N$ 
is of order $O(1/R)$ as $R\to\infty$ uniformly in $s\in[0,1]$, we see that by choosing $R\ge 1$ sufficiently large 
we obtain that there exists $\varepsilon>0$ such that $\det[d\gamma(s)]>\varepsilon$ uniformly on $\R^2$ 
and $s\in[0,1]$. Then, by Hadamard-Levy's theorem, $\gamma(s)$ is an orientation preserving diffeomorphism of $\R^2$
for any $s\in[0,1]$. This completes the proof of the lemma.
\end{proof}

We will also need the following lemma.

\begin{lemma}\label{propTrnsOpn}
For $m>2+(2/p)$ and any given $\varphi\in\mathcal{ZD}^{m,p}_N$ there exists an open neighborhood $U$ of 
the identity $\idmap$ in $\mathcal{ZD}^{m-1,p}_N$ such that the left translation,
\[
\psi\mapsto L_\varphi(\psi):=\varphi\circ\psi,\quad U\to L_\varphi(U)\subseteq\mathcal{ZD}^{m-1,p}_N,
\]
is a $C^1$-diffeomorphism where $L_\varphi(U)$ is an open neighborhood of $\varphi$ in $\mathcal{ZD}^{m-1,p}_N$.
\end{lemma}

This lemma follows directly from the regularity statement in Proposition \ref{prop:composition} and the inverse function theorem.

\medskip

Now we are ready to prove that $\mathcal{ZD}^{m,p}_N$ with $m>3+(2/p)$ is closed under the inversion of diffeomorphisms.
More specifically, we prove the following:

\begin{proposition}\label{prop:inverse}
For $m>3+(2/p)$ and for any $\varphi\in\mathcal{ZD}^{m,p}_N$ we have that $\varphi^{-1}\in\mathcal{ZD}^{m,p}_N$. In addition,
the inversion map $\varphi\mapsto\varphi^{-1}$, $\mathcal{ZD}^{m,p}_N\rightarrow\mathcal{ZD}^{m,p}_N$, is continuous and
the associated map $\varphi\mapsto\varphi^{-1}$, $\mathcal{ZD}^{m,p}_N\rightarrow\mathcal{ZD}^{m-1,p}_N$, 
is $C^1$-smooth.
\end{proposition}

\begin{proof}[Proof of Proposition \ref{prop:inverse}]
Take $\varphi\in\mathcal{ZD}^{m,p}_N$. By Lemma \ref{propPthConnCmptSup}, there exists a continuous path 
$\gamma: [0,1]\to\mathcal{ZD}^{m,p}_N$ that connects $\varphi$ with $\psi_0\in\mathcal{ZD}^{m,p}_N$ where
$\psi_0=\idmap+f$, $f$ has compact support, $\gamma(0)=\psi_0$, and $\gamma(1)=\varphi$.
In particular, we see that $f\in H^{m,p}(\C,\C)$, and hence $\psi_0\in\mathcal{D}^{m,p}(\R^2)$ where 
$\mathcal{D}^{m,p}(\R^2)$ is the group of Sobolev type diffeomorphisms of $\R^2$ considered in \cite{IKT}. 
Hence, $\psi_0^{-1}\in\mathcal{D}^{m,p}(\R^2)$. Since $f$ has compact support, we obtain that
$\psi_0^{-1}=\idmap+g$ where $g$ has compact support. This implies that $\psi_0^{-1}\in\mathcal{ZD}^{m,p}_N$ with vanishing
asymptotic part.
By Lemma \ref{propInvNghd} and Lemma \ref{propTrnsOpn}, for any given $t\in[0,1]$ we can find an open neighborhood $U_t$ of 
the identity in $\mathcal{ZD}^{m-1,p}_N$ with the property that $\psi^{-1}\in\mathcal{ZD}^{m-1,p}_N$ for any
$\psi\in U_t$ and an open neighborhood $V_t$ of $\gamma(t)$ in $\mathcal{ZD}^{m-1,p}_N$ such
that $V_t=L_{\gamma(t)}(U_t)$. By the compactness of the image of  $\gamma$ in $\mathcal{ZD}^{m-1,p}_N$ we can cover it
with finitely many such open neighborhoods in $\mathcal{ZD}^{m-1,p}_N$,
\[
V_{t_k}=L_{\gamma(t_k)}(U_{t_k}),\quad 0\le k\le\ell,
\] 
where $t_0=0$ and $t_\ell=1$. We will assume without loss of generality that
$V_{t_k}\cap V_{t_{k+1}}\ne\emptyset$.
Note that any element of $V_0\equiv V_{t_0}$ is of the form $\psi_0\circ\psi$ where $\psi\in U_0$ with $\psi^{-1}\in\mathcal{ZD}^{m-1,p}_N$.
Since $\psi_0^{-1}\in\mathcal{ZD}^{m-1,p}_N$ we then conclude that $(\psi_0\circ\psi)^{-1}=\psi^{-1}\circ\psi_0^{-1}$
belongs to $\in\mathcal{ZD}^{m-1,p}_N$ by Proposition \ref{prop:composition}. Hence, the open neighborhood $V_0$ in 
$\mathcal{ZD}^{m-1,p}_N$ consists of diffeomorphisms the inverse diffeomorphisms of which belong to $\mathcal{ZD}^{m-1,p}_N$.
Now, consider the neighborhood $V_{t_1}$. Since, $V_0\cap V_{t_1}\ne\emptyset$ we can find $\psi_*\in V_0\cap V_{t_1}$.
Then, $\psi_*^{-1}\in\mathcal{ZD}^{m-1,p}_N$ and $\psi_*=\gamma(t_1)\circ\psi$ where $\psi^{-1}\in\mathcal{ZD}^{m-1,p}_N$.
Hence, $\gamma(t_1)^{-1}=\psi\circ\psi_*^{-1}$ belongs to $\mathcal{ZD}^{m-1,p}_N$.
Since any element in $V_{t_1}$ is of the form $\gamma(t_1)\circ\psi$ where $\psi^{-1}\in\mathcal{ZD}^{m-1,p}_N$ 
and $\gamma(t_1)^{-1}\in\mathcal{ZD}^{m-1,p}_N$ we conclude that $V_{t_1}$ consists of elements invertible in 
$\mathcal{ZD}^{m-1,p}_N$. By continuing this argument inductively, we conclude that $V_1\equiv V_{t_\ell}$ consists of 
elements invertible in $\mathcal{ZD}^{m-1,p}_N$. In particular, we see that
$\varphi^{-1}\in\mathcal{ZD}^{m-1,p}_N$. 
The bootstrapping argument from the proof of Lemma \ref{propInvNghd} now implies that $\varphi^{-1}\in\mathcal{ZD}^{m,p}_N$.
The last statement of the proposition follows from the analogous results for the asymptotic group $\mathcal{AD}^{m,p}_N$
(see \cite[Proposition 5.2]{McOTo2}, \cite{Montgomery}) and the fact that ${\mathcal{ZD}}^{m,p}_N$ is a closed submanifold in 
$\mathcal{AD}^{m,p}_N$.
\end{proof} 

By combining Proposition \ref{prop:composition} with Proposition \ref{prop:inverse} we obtain the following important

\begin{theorem}\label{th:ZD}
For $m>3+(2/p)$ the set $\mathcal{ZD}^{m,p}_N$ is a topological group under composition of diffeomorphisms.
In addition, the composition map $(\psi, \varphi)\mapsto\psi\circ\varphi$, 
$\mathcal{ZD}^{m,p}_N \times \mathcal{ZD}^{m-1,p}_N\to\mathcal{ZD}^{m-1,p}_N$ and 
the inverse map $\varphi \mapsto \varphi^{-1}$, $\mathcal{ZD}^{m,p}_N \rightarrow \mathcal{ZD}^{m-1,p}_N$ are $C^1$-smooth.
\end{theorem}

\medskip\medskip

{\em The Asymptotic Homomorphism of $\mathcal{ZD}^{m,p}_N$:} The asymptotic group $\mathcal{ZD}^{m,p}_N$ has
a natural homomorphism onto a finite dimensional Lie group. Here we define and study the main properties of this homomorphism. 
Assume that $m>3+(2/p)$ and $N\ge 0$. Then, any $\varphi\in\mathcal{ZD}^{m,p}_N$ can be written in the form
\[
\varphi(x,y)=z+\sum_{0\le k+l\le N}\frac{a_{kl}}{z^k\bar{z}^l}+f,\quad f\in W^{m,p}_{\gamma_N},
\]
where $f(x,y)=o\big(1/\z^N\big)$ by Proposition \ref{prop:main_technical} and Remark \ref{rem:R^d}.
Denote by $\hrho$ the map that assigns to an asymptotic diffeomorphism $\varphi\in\mathcal{ZD}^{m,p}_N$ its
asymptotic part $\hrho(\varphi):=\sum_{0\le k+l\le N}\frac{a_{kl}}{z^k\bar{z}^l}$, that we identify with a
point in $\C^M=\R^{2M}$ represented by the real and the imaginary part of the asymptotic coefficients
$a_{kl}$, $0\le k+l\le N$. Here  $M\equiv M_N$ denotes the total number of the asymptotic coefficients. 
In this way we obtain the real analytic map
\begin{equation}\label{eq:hrho}
\hrho : \mathcal{ZD}^{m,p}_N\to\R^{2M}.
\end{equation}
Denote by $\hA\equiv\hA_N$ the image of \eqref{eq:hrho} in $\R^{2M}$. Obviously, $\hA$ is an open set in $\R^{2M}$.
Let us introduce the following multiplication operation on the elements of $\hA$: Take $a,b\in\hA$. Then, for
$\varphi,\psi\in\mathcal{ZD}^{m,p}_N$ such that $a=\hrho(\varphi)$ and $b=\hrho(\psi)$, we set
\begin{equation}\label{eq:*-product}
a\star b:=\hrho(\varphi\circ\psi).
\end{equation}
It is clear from the arguments in the proof of Lemma \ref{lem:composition} that the definition above
is independent of the choice of $\varphi,\psi\in\mathcal{ZD}^{m,p}_N$ such that $a=\hrho(\varphi)$ and $b=\hrho(\psi)$ and that
\[
\hA\times\hA\to\hA,\quad(a,b)\mapsto a\star b,
\]
is a polynomial map. Since the composition of diffeomorphisms is associative, we see from \eqref{eq:*-product}
that the $\star$-product is associative (but not commutative). The element of $\hA$ corresponding to $0\in\R^{2M}$
is the identity element $e\in\hA$. Since $\mathcal{ZD}^{m,p}_N$ is a group, one also concludes from \eqref{eq:*-product}
that the elements of $\hA$ are invertible and for $a=\hrho(\varphi)$ we have that $a^{-1}=\hrho(\varphi^{-1})$.
This shows that $\hA\equiv\hA_N$ is a group. Note that for $N=0$ we have that $\big(\hA_0,\star\big)=\big(\R^2,+\big)$.
In fact, we have the following proposition.

\begin{proposition}\label{prop:hA-group}
For any $N\ge 0$ we have:
\begin{itemize}
\item[(i)] $\big(\hA_N,\star\big)$ is a group modeled on $\R^{2M}$ with real analytic product and inverse operations.
\item[(ii)] For any $m>3+(2/p)$ the map $\hrho : \mathcal{ZD}^{m,p}_N\to\hA_N$ is a homomorphism of real analytic groups.
\end{itemize}
\end{proposition}

\begin{remark}
In geometrical terms, the group $\hA$ can be identified with a group of germs of real analytic diffeomorphisms on the Riemann sphere
that are defined in an open neighborhood of $\infty$ and have $\infty$ as a fixed point.
\end{remark}

\begin{proof}[Proof of Proposition \ref{prop:hA-group}]
It remains to prove item $(i)$ only. As described above, we identify the elements of $\hA$ with points in $\R^{2M}$.
Denote $F(a,b):=a\star b$. Then, $F : \hA\times\hA\to\hA$ is a polynomial map 
that is defined in an open neighborhood of zero 
(that corresponds to the identity element $e$).\footnote{Note that $F$ is a well defined polynomial map 
$\R^{2M}\times\R^{2M}\to\R^{2M}$.}
Since $e\star e=e$ and $a\star e=a$, we obtain that for any $a\in\R^{2M}$,
\[
F(0,0)=0\quad\text{and}\quad F(a,0)=a.
\]
In particular, the partial derivative of $F$ with respect to the first argument $\big(D_1F\big)(0,0)$ is the identity $2M\times 2M$ matrix. 
Hence, by the inverse function theorem, there exists an open neighborhood $U(e)$ of the zero in $\R^{2M}$ and a unique 
{\em real analytic} map $I_U : U(e)\to\hA$, $b\mapsto I_U(b)$, such that $F\big(I_U(b),b\big)=0$ for any $b\in U(e)$. 
This implies that for any $b\in U(e)$ we have that
\[
I_U(b)=b^{-1},
\]
and hence, the inverse operation in $\hA$ is real analytic in the open neighborhood $U(e)$. 
Further, we argue as in the proof of Proposition \ref{prop:inverse}: Take an arbitrary element $a_\bullet\in\hA$.
Then, since $a_\bullet$ is invertible, we obtain from the inverse function theorem that 
the right and the left translation $a\stackrel{R_{a_\bullet}}{\longmapsto} a\star a_\bullet$ and
$a\stackrel{L_{a_\bullet}}{\longmapsto} a_\bullet\star a$ are real analytic diffeomorphisms of $\hA$ so that
$R_{a_\bullet}^{-1}=R_{a_\bullet^{-1}}$ and $L_{a_\bullet}^{-1}=L_{a_\bullet^{-1}}$.
Now, consider the open neighborhood 
\[
U(a_\bullet):=R_{a_\bullet}\big(U(e)\big)=\big\{a\star a_\bullet\,\big|\,a\in U(e)\big\}
\] 
of $a_\bullet$ in $\hA$. Note that for any $c\in U(a_\bullet)$ we have that $c=a\star a_\bullet$ and hence
\[
c^{-1}=a_\bullet^{-1}\star a^{-1}=L_{a_\bullet^{-1}}\big(I_U(a)\big)=
\big(L_{a_\bullet^{-1}}\circ I_U\circ R_{a_\bullet^{-1}}\big)(c).
\]
Since the maps $L_{a_\bullet^{-1}} : \hA\to\hA$, $R_{a_\bullet^{-1}} : \hA\to\hA$, and
$I_U : U(e)\to\hA$ appearing on the right hand side of the formula above are real analytic, 
we conclude that the inverse operation on $\hA$ is real analytic when restricted to $U(a_\bullet)$.
This completes the proof of the proposition.
\end{proof}

Now, take a curve $\varphi\in C^1\big((-\varepsilon,\epsilon),\mathcal{ZD}^{m,p}_N\big)$ with $\varepsilon>0$ such that
$\dt\varphi|_{t=0}=u$. Then, by \eqref{eq:*-product}, for any given $\psi\in\mathcal{ZD}^{m,p}_N$ we have that
\begin{equation}\label{eq:hrho-representation}
\hrho\big(\varphi(t)\circ\psi\big)=\hrho\big(\varphi(t)\big)\star\hrho(\psi),\quad t\in(-\varepsilon,\varepsilon).
\end{equation}
Note that the right translation $R_\psi : \mathcal{ZD}^{m,p}_N\to\mathcal{ZD}^{m,p}_N$ on $\mathcal{ZD}^{m,p}_N$ is 
an affine linear, and hence a real analytic map. By taking the $t$ derivative at zero in \eqref{eq:hrho-representation},
we then obtain that for any $u\in\mathcal{Z}^{m,p}_N$,
\begin{equation}\label{eq:hrho_*-representation}
\hrho_*(u\circ\psi)=d_eR_{\hrho(\psi)}\big(\hrho_*(u)\big),
\end{equation}
where $\hrho_* : \mathcal{Z}^{m,p}_N\to\R^{2M}$ denotes the linear map assigning to $u\in\mathcal{Z}^{m,p}_N$
its asymptotic part and $d_eR_a : \R^{2M}\to\R^{2M}$ with $a\in\hA$ is the differential of the right translation in $\hA$ 
at the identity element $e\in\hA$ .

\section{The Cauchy operator in asymptotic spaces}\label{sec:cauchy_operator}
In this section we study the properties of the Cauchy operator $\partial_z$ acting on the scale of asymptotic spaces $\mathcal{Z}^{m,p}_N$.
To this end, we first study the properties of this operator on the scale of weighted Sobolev spaces $W^{m,p}_\delta$ with $\delta\in\R$.
In particular, we show that as in the case of the Laplace operator (cf. \cite{McOwen1}), the Cauchy operator is an injective or
surjective Fredholm operator for all but a discrete values of the weight $\delta\in\R$. In the Fredholm case we describe explicitly the kernel
and the (closed) image of $\partial_z$. Note that the fundamental solution of the Cauchy operator
$\partial_z$ is $\K:=1/\pi\bar{z}$.
The following estimate will play an important role in our analysis.  
 
\begin{lemma}\label{lem:kernel}
Let $\Psi$ be a $C^\infty$-smooth complex valued function such that
$\Psi(z,\bar{z})=1/\bar{z}$ for $|z|\ge 2$ and $\big|\Psi(z,\bar{z})\big|\le 1$ for $z\in\C$. For $z,w\in\C$ and any given 
$l\in\Z_{\ge 0}$ define,
\begin{equation*}
K(\bar{z},\bar{w}):=\frac{1}{\bar{z}-\bar{w}},\quad
\widetilde{K}_l(z,\bar{z};w,\bar{w}):=\Psi(z,\bar{z})\sum_{k=0}^{l}\bar{w}^ k \Psi(z,\bar{z})^k.
\end{equation*}
Then, there exists a constant $C\equiv C(l)>0$ such that for any $z,w\in\C$,
\begin{equation}\label{eq:kernel_estimate}
\big|K(\bar{z},\bar{w})-\widetilde{K}_l(z,\bar{z};w,\bar{w})\big|\le C(l)\,\frac{\w^{l+1}}{|z-w|\z^{l+1}}.
\end{equation}
\end{lemma}

\begin{proof}[Proof of Lemma \ref{lem:kernel}]
We will prove \eqref{eq:kernel_estimate} with $\z^{l+1}$ and $\w^{l+1}$ replaced by $1+|z|^{l+1}$  and $1+|w|^{l+1}$
respectively. Assume that $\big|\frac{w}{z}\big|\ge\frac{1}{2}$. If $|z| \geq 2$ then
\begin{align}
\abnorm{\left(K - \widetilde{K}_l\right)(\bar{z}-\bar{w})(1+|z|^{l+1})}&= 
(1+|z|^{l+1})\Big|1-(\bar{z}-\bar{w})\Psi(z,\bar{z})\sum_{k=0}^{l} \bar{w}^k\Psi(z,\bar{z})^k\Big|\label{eqEst1}\\ 
&=(1+|z|^{l+1})\abnorm{1-\Big(1 - \frac{\bar{w}}{\bar{z}}\Big)\sum_{k=0}^{l} \frac{\bar{w}^k}{\bar{z}^k}}\nonumber\\
&=(1+|z|^{l+1})\abnorm{\frac{\bar{w}}{\bar{z}}}^{l+1}\le\Big(1+\frac{1}{|z|^{l+1}}\Big)(1+|w|^{l+1})\nonumber\\
&\le 2\big(1+|w|^{l+1}\big)\nonumber
\end{align}
which proves \eqref{eq:kernel_estimate} in the considered case.
If $|z|\le 2$ then the right hand side of \eqref{eqEst1} is estimated by
$\big(1+2^{l+1}\big)\big(1+3|w|\sum_{k=0}^l |w^k|\big)$, since 
$\big|\Psi(z,\bar{z})\big|\le 1$ and $|\bar{z}-\bar{w}|\le|z|+|w|\le 3|w|$.
This, together with the estimate
\[
|w|\,\sum_{k=0}^l|w|^k=|w|\,\frac{1+|w|^{l+1}}{1+|w|}\le 1+|w|^{l+1},
\]
then implies that the right hand side of \eqref{eqEst1} is bounded by $4\big(1+2^{l+1}\big)\big(1+|w|^{l+1}\big)$.
This proves \eqref{eq:kernel_estimate} in the case when $\big|\frac{w}{z}\big|\ge\frac{1}{2}$.

Now, assume that $\big|\frac{w}{z}\big|\le\frac{1}{2}$. Then, 
$K(\bar{z},\bar{w})=\frac{1}{\bar{z}} \sum_{k=0}^{\infty} (\frac{\bar{w}}{\bar{z}})^k$. 
If $|z|\ge 2$ we have
\begin{equation*}
\begin{split}
\abnorm{K(\bar{z},\bar{w}) - \widetilde{K}_l(z,\bar{z};w,\bar{w})}& = \abnorm{\frac{1}{\bar{z}}
\sum_{k=l+1}^{\infty}\Big(\frac{\bar{w}}{\bar{z}}\Big)^k}\le\frac{2}{|z|}\frac{|w|^{l+1}}{|z|^{l+1}}\\ &
\le\frac{3}{2} \frac{2}{|z-w|}\frac{|w|^{l+1}}{|z|^{l+1}} 
\le 3\frac{1+|w|^{l+1}}{|z-w|(1+|z|^{l+1})}
\end{split}
\end{equation*}
where we used that
\begin{equation}\label{eq:two_estimates}
|z-w|\le|z|+|w|\le\frac{3}{2}|z|\quad\text{\rm and}\quad\frac{|w|^{l+1}}{|z|^{l+1}}\le\frac{1+|w|^{l+1}}{1+|z|^{l+1}}.
\end{equation}
The second inequality in \eqref{eq:two_estimates} follows from the estimate
\[
|w|^{l+1}\big(1+|z|^{l+1}\big)-|z|^{l+1}\big(1+|w|^{l+1}\big)=|w|^{l+1}-|z|^{l+1}=\Big(\frac{1}{2^{l+1}}-1\Big)|z|^{l+1}\le 0.
\]
If $|z| \le 2$ we obtain from $\big|\Psi(z,\bar{z})\big|\le 1$ and $\big|\frac{w}{z}\big|\le\frac{1}{2}$ that
\begin{eqnarray*}
\Big|K(\bar{z},\bar{w}) - \widetilde{K}_l(z,\bar{z};w,\bar{w})\Big|&=& 
\abnorm{\frac{1}{\bar{z}}\sum_{k=l+1}^{\infty}\Big(\frac{\bar{w}}{\bar{z}}\Big)^k + 
\frac{1}{\bar{z}}\sum_{k=0}^{l}\Big(\frac{\bar{w}}{\bar{z}}\Big)^k\big(1-\bar{z}^{k+1}\,\Psi(z,\bar{z})^{k+1}\big)}\\
&\le&\frac{1}{|z|}\sum_{k=l+1}^{\infty}\Big|\frac{w}{z}\Big|^k+
\frac{1}{|z|}\sum_{k=0}^l\Big(\frac{1}{2}\Big)^k\big(1+2^{k+1}\big)\\
&\le& \frac{2}{|z|}\frac{|w|^{l+1}}{|z|^{l+1}}+\frac{1}{|z|}(4+2l)\\
&\le&C(l)\,\frac{(1+|w|^{l+1})}{|z-w|(1+|z|^{l+1})}
\end{eqnarray*}
where, in order to conclude the final estimate, we used \eqref{eq:two_estimates} and the fact that 
\begin{equation*}
1\le\big(1+2^{l+1}\big)\,\frac{1+|w|^{l+1}}{1+|z|^{l+1}}
\end{equation*}
for any $|z|\le 2$ and for any $w\in\C$. This completes the proof of the lemma.
\end{proof}

We will use Lemma \ref{lem:kernel} to construct a bounded map that inverts the Cauchy operator acting on weighed Sobolev and 
asymptotic spaces. For $\frac{1}{p}+\frac{1}{q}=1$ consider the non-degenerate pairing
\begin{equation}\label{eq:pairing}
(\cdot,\cdot) : L^p_{\delta } \times L^{q}_{-\delta}\to\C,\quad(u,v):= \int_{\R^2} uv\,dx\,dy=-\frac{1}{2i}\int_{\C}uv\,dz\wedge d\bar{z}
\end{equation}
where $L^p_\delta\equiv W^{0,p}_\delta$ consists of maps $u : \C\to\C$ such that $|u|\z^\delta\in L^p(\C,\R)$.
Note that by H\"older's inequality $|(u,v)|\le\|u\|_{L^p_\delta}\|v\|_{L^q_{-\delta}}$, and hence the bilinear map in \eqref{eq:pairing}
is bounded. For any $u,v\in C^\infty_c$,
\[
\big(\partial_zu,v\big)=-\frac{1}{2i}\int_{\C}\big((uv)_z-u v_z\big)\,dz\wedge d\bar{z}=-\big(u,\partial_z v\big),
\]
where we used that by Stokes' theorem
\[
\int_{\C}(uv)_z\,dz\wedge d\bar{z}=\int_{|z|\le R}d\big(uv\,d\bar{z}\big)=\oint_{|z|=R}uv\,d\bar{z}=0
\] 
for $R>0$ taken sufficiently large. 

Denote by $X^*$ the dual of the Banach space $X$.
The Cauchy operator extends to a continuous operator $\partial_z : S'\to S'$ where $S'$ denotes Schwartz space of
tempered distributions acting on the Schwartz space  $S$ of maps $\C\to\C$. The continuity of the pairing \eqref{eq:pairing} then easily implies 
that for any $u\in L^p_\delta$ we have that $\partial_zu\in(W^{1, q}_{-\delta - 1})^*$ and the map 
$\partial_z : L^p_\delta\to(W^{1, q}_{-\delta - 1})^*$ is bounded. Moreover, for $u\in L^p_\delta$ with 
$\partial_z u\in L^p_{\delta+1}$ we have that for any $v\in L^q_{-\delta-1}$,
\begin{equation}\label{eq:weak_derivative}
\big(\partial_z v\big)(u)=-\big(v,\partial_z u\big).
\end{equation}

We will first prove the following simple lemma.

\begin{lemma}\label{lem:injectivity}
Assume that $\delta+(2/p)>0$. Then, $u\in L^p_\delta$ and $\partial_z u=0$ imply that $u=0$.
In particular, the (bounded) map $\partial_z : W^{1,p}_{\delta}\to L^p_{\delta+1}$ is injective.
\end{lemma}

\begin{remark}\label{rem:cauchy_kernel}
In fact, one can easily see by using the Fourier transform that the kernel of the Cauchy operator
$\partial_z : S'\to S'$ consists of all polynomials of $\bar{z}$ with complex coefficients, $\ker\partial_z=\C[\bar{z}]$.
Then, Lemma \ref{lem:injectivity} follows, since for $\delta+(2/p)>0$ we have that $\bar{z}^k\notin L^p_\delta$ for any
$k\ge 0$.
\end{remark}

\begin{proof}[Proof of Lemma \ref{lem:injectivity}]
Assume that $u\in L^p_\delta$ and $\partial_zu=0$. Take an arbitrary test function $\varphi\in C^\infty_c$.
Since $\delta+(2/p)>0$ we conclude that $(\delta+2)q>2$ where $(1/p)+(1/q)=1$. This, together with the estimate 
$\K*\varphi=O\big(1/\z\big)$, implies that $\K*\varphi\in L^q_{-\delta-1}$.
By using that $\K$ is the fundamental solution of $\partial_z$ we obtain from \eqref{eq:weak_derivative} that
\begin{equation*}
\big(\varphi,u\big)=\big(\partial_z(\K*\varphi)\big)(u)=-\big(\K*\varphi,\partial_z u\big) = 0
\end{equation*}
since $\partial_zu=0$. Since this holds for any $\varphi\in C^\infty_c$, we conclude that $u=0$.
\end{proof}

Further, we prove

\begin{lemma}\label{lem:almost_onto}
For any weight $\delta\in\R$ we have that $u\in L^p_{\delta}$ and $\partial_zu \in L^p_{\delta+1}$ imply that 
$u\in W^{1,p}_{\delta}$.
\end{lemma}

\begin{proof}[Proof of Lemma \ref{lem:almost_onto}]
By multiplying $u$ with the weight $\z^\alpha$ with $\alpha\in\R$ appropriately chosen, we reduce the lemma to the case when 
$0<\delta+(2/p)<1$.
Take $u\in L^p_{\delta}$ and $f:=\partial_zu\in L^p_{\delta+1}$ with $0<\delta+(2/p)<1$.
Let $(f_k)_{k\ge 1}$ be a sequence in $C^\infty_c$ such that $f_k\to f$ in $L^p_{\delta+1}$ as $k\to\infty$. 
Denote 
\begin{equation}\label{eq:u_k}
u_k:=\K*f_k\,.
\end{equation}
Since $\K\in S'$ is the fundamental solution of the Cauchy operator $\partial_z$, we have that for any $k\ge 1$,
\begin{equation}\label{eq:cauchy_equation}
\partial_z u_k=f_k\quad\text{\rm and}\quad\partial_{\bar{z}} u_k=(\partial_{\bar{z}}\K)*f_k
\end{equation}
in distributional sense in $S'$.
A direct computation shows that
\begin{equation}\label{eq:K_{z-bar}}
\partial_{\bar{z}}\K=-\text{p.v.}\,\frac{1}{\pi\bar{z}^2}\in S'
\end{equation}
where p.v. denotes the Cauchy principal value.
In fact, for any test function $\varphi \in C^{\infty}_c$ and $R>0$ sufficiently large we have 
\begin{align*}
\big(\partial_{\bar{z}}\K\big)(\varphi)&=-(\K, \partial_{\bar{z}}\varphi) = 
\frac{1}{2\pi i}\int_{\C} \frac{1}{\bar{z}}\,\varphi_{\bar{z}}\,dz\wedge d \bar{z} = 
\frac{1}{2\pi i}\lim_{r\to 0+}\int_{r\le|z|\le R}\frac{1}{\bar{z}}\,\varphi_{\bar{z}}\,dz\wedge d\bar{z}\\
&=\frac{1}{2\pi i}\lim_{r\to 0+}\int_{r\le|z|\le R} \partial_{\bar{z}}\Big(\frac{1}{\bar{z}}\,\varphi\Big)\,dz\wedge d\bar{z} 
+\frac{1}{2\pi i}\lim_{r\to 0+}\int_{r\le|z|\le R}\frac{1}{\bar{z}^2}\,\varphi\,dz\wedge d\bar{z} \\
&=\frac{1}{2 i}\,\text{p.v.}\int_{\C}\frac{1}{\pi\bar{z}^2}\,\varphi\,dz\wedge d\bar{z}.
\end{align*}
Here we used Stokes' theorem to obtain
\begin{align*}
\lim_{r\to 0+}\int_{r\le|z|\le R}\partial_{\bar{z}}\Big(\frac{1}{\bar{z}}\,\varphi\Big)\,dz\wedge d\bar{z}
&=-\lim_{r\to 0+}\int_{r\le|z|\le R}d\Big(\frac{1}{\bar{z}}\,\varphi\,dz\Big)=\lim_{r\to 0+}\oint_{|z|=r}\frac{1}{\bar{z}}\,\varphi\,dz\\
&=i \lim_{r\to 0+} \int_0^{2\pi} e^{2i\theta}\varphi(re^{i\theta})\,d\theta = 0.
\end{align*}
By Lemma 1 in \cite{McOwen1} (cf. Theorem B* in \cite{SteinWeiss} and Lemma 2.4 in \cite{Lockhart}) for $0<\delta+(2/p)<1$,
the convolution with the fundamental solution $\K$ of the Cauchy operator $\partial_z$ extends to a bounded linear map 
$L^p_{\delta+1}\rightarrow L^p_{\delta}$. Moreover, by Theorem 3 in \cite[Ch. II, \S 4]{SteinBook}, Theorem 1 in \cite{Stein}, 
and \eqref{eq:K_{z-bar}}, for $0<\delta+(2/p)<2$, the convolution with $\partial_{\bar{z}}\K$ extends to a bounded linear map 
$L^p_{\delta+1}\to L^p_{\delta+1}$.
(Note that the cancellation condition required in Theorem 3 in \cite[Ch. II,\S 4]{SteinBook} holds, since $1/\bar{z}^2=e^{2i\theta}/r^2$
and $\int_0^{2\pi}e^{2i\theta}\,d\theta=0$.) 
In particular, by passing to the limit in \eqref{eq:u_k} we obtain that
\[
{\tilde u}:=\K* f=\lim_{k\to\infty}\K*f_k\in L^p_\delta
\]
where the limit exists in $L^p_\delta$. Similarly, by passing to the limit in $S'$ in \eqref{eq:cauchy_equation} we obtain that
\[
\partial_{z}{\tilde u}=f\in L^p_{\delta+1}\quad\text{\rm and}\quad\partial_{\bar{z}}{\tilde u}=(\partial_{\bar{z}}\K)*f\in L^p_{\delta+1}.
\]
This implies that ${\tilde u}\in W^{1,p}_\delta$. Since $\partial_{z}{\tilde u}=\partial_{z}u=f$, we obtain from the first
statement of Lemma \ref{lem:injectivity} that $u={\tilde u}$. Hence, $u\in W^{1,p}_\delta$.
This completes the proof of the lemma.
\end{proof}

As a direct consequence of these results we obtain

\begin{proposition}\label{th:isomorphism W->L}
For $0<\delta+(2/p)<1$ the map $\partial_z : W^{1,p}_\delta\to L_{\delta+1}^p$ is an isomorphism of Banach spaces.
\end{proposition}

\begin{proof}[Proof of Proposition \ref{th:isomorphism W->L}]
For any $f \in L^p_{\delta+1}$ define $u:=\K*f$. Since for $0<\delta+(2/p)<1$ the convolution with $\K$ extends to 
a bounded map $L^p_{\delta+1}\to L^p_\delta$, we obtain that $u\in L^p_\delta$. A simple continuity argument and the fact
that $\K$ is the fundamental solution of $\partial_z$ then implies that $\partial_zu=f\in L^p_{\delta+1}$.
Hence, by Lemma \ref{lem:almost_onto}, we obtain that $u\in W^{1,p}_\delta$. This, together with $\partial_zu = f$,
finally gives that $\partial_z : W^{1,p}_\delta\to L_{\delta+1}^p$ is onto. 
The injectivity of this map follows from Lemma \ref{lem:injectivity}. The statement of the proposition then follows from 
the open mapping theorem.
\end{proof}

We also have 

\begin{proposition}\label{th:W->L}
For $l+1<\delta+(2/p)< l+2$ and $l\in\Z_{\ge 0}$ the map $\partial_z : W^{1,p}_\delta\to L_{\delta + 1}^p$ is injective 
with closed  in $L^p_{\delta+1}$ image $R_l:=\big\{f\in L^p_{\delta+1}\,\big|\,\big(f,\bar{z}^k\big)=0\,\,\,\text{\rm for}\,\,\,0\le k\le l\big\}$
of (complex) co-dimension $l+1$.
\end{proposition}

\begin{proof}[Proof of Proposition \ref{th:W->L}]
For a given $l\ge 0$ and $z,w\in\C$ denote 
\begin{equation*}
F_l(z,\bar{z};w,\bar{w}):=K(\bar{z},\bar{w})-\widetilde{K}_l(z,\bar{z};w,\bar{w}).
\end{equation*}
Note that by Lemma \ref{lem:kernel},
\begin{equation}\label{eq:F_l-estimate}
\big|F_l(z,\bar{z};w,\bar{w})\big|\le C(l)\,\frac{\w^{l+1}}{|z-w|\z^{l+1}}.
\end{equation}
Then the map,
\begin{equation}\label{eq:F-map}
\mathcal{F}_l : C^\infty_c\to C^\infty,\quad 
\mathcal{F}_l\big(u\big)(z,\bar{z}):=-\frac{1}{2\pi i}\int_\C F_l(z,\bar{z};w,\bar{w})u(w,\bar{w})\,dw\wedge d\bar{w},
\end{equation}
associated to the kernel $F_l/\pi$ extends to a bounded linear map $L^p_{\delta+1} \rightarrow L^p_{\delta}$. 
In order to see this, note that \eqref{eq:F-map} can be decomposed as
\begin{equation}\label{eq:decomposition}
\mathcal{F}_l=(\mathcal{M}_{l+1})^{-1}\circ\mathcal{F}\circ\mathcal{M}_{l+1}
\end{equation}
where for any given weight $\mu\in\R$,
\[
\mathcal{M}_{l+1} : L^p_\mu\to L^p_{\mu-(l+1)},\quad f\mapsto\z^{l+1} f,
\]
is a linear isomorphism and $\mathcal{F} : C^\infty_c\to C^\infty$ is of the form \eqref{eq:F-map} with kernel
$F(z,\bar{z};w,\bar{w})$ satisfying, by the estimate \eqref{eq:F_l-estimate}, the inequality
\begin{equation*}
\big|F(z,\bar{z};w,\bar{w})\big|\le\frac{C}{|z-w|}
\end{equation*}
with a constant $C>0$ depending only on $l\ge 0$.
This, together with the estimate
\[
0<\big[\delta-(l+1)\big]+(2/p)<1,
\]
then implies that the map $\mathcal{F} : C^\infty_c\to C^\infty$ extends to a bounded linear map 
$L^p_{[\delta-(l+1)]+1}\to L^p_{\delta-(l+1)}$,
as discussed in the proof of Lemma \ref{lem:almost_onto}. Hence, in view of the decomposition \eqref{eq:decomposition}, the map
\eqref{eq:F-map} extends to a bounded linear map
\begin{equation}\label{eq:F-regularity}
\mathcal{F}_l : L^p_{\delta+1}\to L^p_\delta.
\end{equation}
Then, for any $f\in L^p_{\delta+1}$,
\begin{equation}\label{eq:representation}
\K*f=\widetilde{\K}_l(f)+\mathcal{F}_l(f),
\end{equation}
where $\widetilde{\K}_l : L^p_{\delta+1}\to C^\infty$ is the integral operator with kernel $\widetilde{K}_l/\pi$,
\begin{equation}\label{eq:K-tilde}
\widetilde{\K}_l(f):=\frac{\Psi}{\pi}\sum_{k=0}^l\big(f,\bar{z}^k\big)\Psi^k.
\end{equation}

\begin{remark}\label{rem:asymptotic_operator}
In view of the definition of $\Psi$ we see that $\widetilde{\K}_l(f)\in C^\infty$ and that for $|z|\ge 2$,
\begin{equation}\label{eq:asymptotic_operator}
\widetilde{\K}_l(f)=\frac{1}{\pi}\sum_{k=1}^{l+1}\big(f,\bar{z}^{k-1}\big)\,\frac{1}{\bar{z}^k},
\end{equation}
for any $f\in L^p_{\delta+1}$. This, together with Lemma \ref{lem:break}, then implies that 
$\widetilde{\K}(f)\in\mathcal{Z}^{m,p}_{l+1}$ for any regularity exponent $m\ge 0$. In addition, 
the linear map $\widetilde{\K_l} : L^p_{\delta+1}\to\mathcal{Z}^{m,p}_{l+1}$ is bounded for any $m\ge 0$.
\end{remark}

Note that for $l+1<\delta+(2/p)< l+2$ we have $(\delta+1-l)q>2$ which implies that
\begin{equation}\label{eq:z^k}
\bar{z}^k\in L^q_{-(\delta+1)+(l-k)}\subseteq L^q_{-(\delta+1)}
\end{equation}
for any $0\le k\le l$.
In particular, we see that for any $u\in L^p_{\delta+1}$ and for any $0\le k\le l$ the pairing
$\big(f,\bar{z}^k\big)$ appearing in \eqref{eq:K-tilde} is well-defined and continuous considered as a function of 
its first argument $f\in L^p_{\delta+1}$.
Now, take $f\in L^p_{\delta+1}\cap R_l$. By \eqref{eq:K-tilde} and the fact that $f\in R_l$ we obtain that
$\widetilde{\K}(f)=0$. This, together with \eqref{eq:representation} and \eqref{eq:F-regularity}, then implies that
\[
\K*f=\mathcal{F}_l(f)\in L^p_\delta.
\]
Since $\K$ is the fundamental solution of $\partial_z$ we also have that $\partial_z\big(\K*f\big)=f\in L^p_{\delta+1}$.
Then, by Lemma \ref{lem:almost_onto} we obtain that $\K*f\in W^{1,p}_\delta$. This proves that $R_l$ is contained in
the image of $\partial_z : W^{1,p}_\delta\to L_{\delta + 1}^p$.
On the other side, it follows from \eqref{eq:z^k} that for $u\in W^{1,p}_{\delta}$ and for any $0\le k\le l$,
\begin{equation*}
\big(\partial_zu,\bar{z}^k\big)=-\big(u,\partial_z (\bar{z}^k)\big) = 0.
\end{equation*}
Hence, $R_l$ is the image of $\partial_z : W^{1,p}_\delta\to L_{\delta + 1}^p$.
The injectivity of this map follows from Lemma \ref{lem:injectivity}. This completes the proof of the proposition.
\end{proof}




The last two theorems can be generalized for an arbitrary regularity exponent $m\ge 0$.

\begin{theorem}\label{th:cauchy_m>=0} 
Assume that $m\in\Z_{\ge 0}$. Then
\begin{itemize}
\item[(i)] For $0<\delta+(2/p)<1$ the map $\partial_z : W_{\delta}^{m+1,p}\to W_{\delta +1}^{m,p}$ is an isomorphism
of Banach spaces.
\item[(ii)] For  $l+1<\delta+(2/p)<l+2$, $l\in\Z_{\ge 0}$, and any regularity exponent $m\ge 0$ the map 
$\partial_z : W_{\delta}^{m+1,p}\to W_{\delta +1}^{m,p}$ 
is injective with closed in $W^{m,p}_{\delta+1}$ image 
$R_l:=\big\{ f\in W^{m,p}_{\delta+1}\,\big|\,\big(f,\bar{z}^k\big) = 0\,\,\,\text{\rm for}\,\,\,0\le k\le l\big\}$
of (complex) co-dimension $l+1$.
\end{itemize}
\end{theorem}

\begin{remark}\label{rem:cauchy_m>=0}
Theorem \ref{th:cauchy_m>=0} also holds with $\partial_z$ replaced by $\partial_{\bar{z}}$ and $R_l$ replaced by
\[
\overline{R_l}:=\big\{ f\in W^{m,p}_{\delta+1}\,\big|\,\big(f,z^k\big) = 0\,\,\,\text{\rm for}\,\,\,0\le k\le l\big\}.
\]
\end{remark}

\begin{proof}[Proof of Theorem \ref{th:cauchy_m>=0}]
We will prove the theorem by induction in the regularity exponent $m\ge 0$.
For $m=0$ the statement follows from Proposition \ref{th:isomorphism W->L}. 
Now, take $m\ge 1$ and assume that the statement holds for $m$ replaced by $m-1\ge 0$.
Take $f\in W^{m,p}_{\delta +1}\subseteq W^{m-1,p}_{\delta+1}\subseteq L^p_{\delta+1}$. 
By the induction hypothesis, there exists $u\in W^{m,p}_\delta$ such that $\partial_z u=f$.
Then for any multi-index $\alpha\in\Z_{\ge 0}^2$ with $|\alpha|\le m$ we have that 
$\partial^\alpha f\in L^p_{\delta+|\alpha|+1}$.
Hence, $\partial^\alpha u\in L^p_{\delta+|\alpha|}$ and
\[
\partial_z\big(\partial^\alpha u\big)=\partial^\alpha\big(\partial_z u\big)=\partial_z^\alpha f\in L^p_{\delta+|\alpha|+1}.
\]
Next, we apply Lemma \ref{lem:almost_onto} to conclude that $\partial^\alpha u\in W^{1,p}_{\delta+|\alpha|}$ 
for any $|\alpha|\le m$. This implies that $u\in W^{m+1,p}_\delta$.
Hence, $\partial_z : W_{\delta}^{m+1,p}\to W_{\delta +1}^{m,p}$ is onto.
The injectivity of this map follows from Lemma \ref{lem:injectivity}.
The first statement of the theorem then follows from the open mapping theorem. 
The same arguments, together with Proposition \ref{th:W->L}, prove the second statement of theorem.
\end{proof}
 
Recall that for $m>2/p$ and $N\ge 0$ the weight $\gamma_N$ of the remainder space $W^{m,p}_{\gamma_N}$ of 
the asymptotic space $\mathcal{Z}^{m,p}_{\gamma_N}$ (cf. \eqref{eq:Z-space}) are taken of the form $\gamma_N=N+\gamma_0$
where $0<\gamma_0+(2/p)<1$. In particular, we see that
\begin{equation}\label{eq:inequality_gamma_N}
N<\gamma_N+(2/p)<N+1.
\end{equation}
Then, by Theorem \ref{th:cauchy_m>=0} above, the map 
$\partial_z : W_{\gamma_N}^{m+1,p}\to W^{m,p}_{\gamma_N+1}$ is injective with closed 
in $W^{m,p}_{\gamma_N+1}$ image
\[
\partial_z\big(W_{\gamma_N}^{m+1,p}\big)=
\big\{f\in W^{m,p}_{\gamma_N+1}\,\big|\,\big(f,\bar{z}^k\big) = 0\,\,\,\text{\rm for}\,\,\,0\le k\le N-1\big\}
\] 
where we assume that for $N=0$ the set on the right hand side is equal to $W^{m,p}_{\gamma_N+1}$, and hence,
the map is an isomorphism.
In fact, our method allows to construct a right inverse of the map 
$\partial_z : W_{\gamma_N}^{m+1,p}\to W^{m,p}_{\gamma_N+1}$
that is defined on the whole of $W^{m,p}_{\gamma_N+1}$ and takes values in the asymptotic space $\mathcal{Z}^{m+1,p}_N$. 
More specifically, we have the following important proposition.

\begin{proposition}\label{prop:asymptotic_operator}
For any $m, N\ge 0$, $m+1>2/p$, there exists a bounded map $\mathcal{L} : W^{m,p}_{\gamma_N+1}\to\mathcal{Z}^{m+1,p}_{1,N}$ 
such that for any $f\in W^{m,p}_{\gamma_N+1}$ we have that $\partial_z\big(\mathcal{L}(f)\big)= f$ and
\begin{equation*}
\mathcal{L}(f)=\frac{\chi}{\pi}\sum_{k=1}^N\big(f,\bar{z}^{k-1}\big)\,\frac{1}{\bar{z}^k}+\mathcal{R}(f)
\end{equation*}
where $\mathcal{R} : W^{m,p}_{\gamma_N+1}\to W^{m+1,p}_{\gamma_N}$ is a bounded linear map and
for $N=0$ we omit the summation term in the formula.
\end{proposition}

\begin{proof}[Proof of Proposition \ref{prop:asymptotic_operator}]
The case when $N=0$ is trivial since the map $\partial_z : W_{\gamma_N}^{m+1,p}\to W^{m,p}_{\gamma_N+1}$
is an isomorphism. For $N\ge 1$ consider the bounded linear map $\mathcal{P} : W^{m,p}_{\gamma_N+1}\to W^{m,p}_{\gamma_N+1}$,
\begin{equation}\label{eq:P}
\mathcal{P}(f):=f-\frac{1}{\pi}\sum_{k=1}^N\big(f,\bar{z}^{k-1}\big)\,\frac{\chi_z}{\bar{z}^k},
\end{equation}
where $\chi_z\equiv\partial_z\chi\in C^\infty_c$. Note that for any $k\in\Z$ we have by Stokes' theorem that
\begin{equation}\label{eq:orthogonality}
\big(\chi_z,1/\bar{z}^k\big)=-\frac{1}{2i}\int_{|z|\le R}d\Big(\frac{\chi}{\bar{z}^k}\,d\bar{z}\Big)\\
=-\frac{1}{2i}\oint_{|z|=R}\frac{d\bar{z}}{\bar{z}^k}=\pi\delta_{k1}
\end{equation}
where $R>0$ is chosen sufficiently large. This, together with \eqref{eq:P}, then implies that 
$\big(\mathcal{P}(f),\bar{z}^k\big)=0$ for any $0\le k\le N-1$.
Hence, by Theorem \ref{th:cauchy_m>=0},
\[
\mathcal{P}\big(W^{m,p}_{\gamma_N+1}\big)=\partial_z\big(W_{\gamma_N}^{m+1,p}\big).
\]
Since the map $\partial_z : W_{\gamma_N}^{m+1,p}\to W^{m,p}_{\gamma_N+1}$ is injective
with closed image, we obtain from the open mapping theorem that there is a bounded linear map
$\imath : \partial_z\big(W_{\gamma_N}^{m+1,p}\big)\to W_{\gamma_N}^{m+1,p}$
such that $\partial_z\big(\imath(f)\big)=f$ for any $f\in\partial_z\big(W_{\gamma_N}^{m+1,p}\big)$.
In particular, we have that $\partial_z\big(\big(\imath\circ\mathcal{P}\big)(f)\big)=\mathcal{P}(f)$ for any 
$f\in W_{\gamma_N+1}^{m,p}$. This, together with \eqref{eq:P} then gives that
\[
\partial_z\left(\big(\imath\circ\mathcal{P}\big)(f)+\frac{\chi}{\pi}\sum_{k=1}^N\big(f,\bar{z}^{k-1}\big)\,\frac{1}{\bar{z}^k}\right)=f
\]
for any $f\in W_{\gamma_N+1}^{m,p}$. By setting $\mathcal{R}(f):=\big(\imath\circ\mathcal{P}\big)(f)$ and
\[
\mathcal{L}(f):=\frac{\chi}{\pi}\sum_{k=1}^N\big(f,\bar{z}^{k-1}\big)\,\frac{1}{\bar{z}^k}+\big(\imath\circ\mathcal{P}\big)(f),\quad
f\in W_{\gamma_N+1}^{m,p},
\]
we conclude the proof of the proposition.
\end{proof}

\medskip

Now, we are ready to study the action of the Cauchy operator $\partial_z$ on the scale of asymptotic
spaces $\mathcal{Z}_N^{m,p}$. To this end, recall that
\[
\widetilde{\mathcal{Z}}_N^{m,p}\equiv \Big\{\frac{\chi}{z^2} 
\sum_{0\le k+l\le N-2}\frac{a_{kl}}{z^k\bar{z}^l}+f\,\Big|\,
f\in W_{\gamma_N}^{m,p}\,\,\text{\rm and}\,\,a_{kl}\in\C\Big\}
\]
and 
\begin{equation}
\mathcal{Z}_{n,N}^{m,p}\equiv\Big\{\chi\!\!\!\sum_{n\le k+l\le N}\frac{a_{kl}}{z^k\bar{z}^l}+f
\,\Big|\,f \in W_{\gamma_N}^{m,p}\,\,\text{\rm and}\,\,a_{kl}\in\C\Big\}
\end{equation}
where $0\le n\le N+1$ and we set that $\mathcal{Z}_{N+1,N}^{m,p}\equiv W^{m,p}_{\gamma_N}$ and 
$\widetilde{\mathcal{Z}}_1^{m,p}\equiv W^{m,p}_{\gamma_1}$. 
Then, we have

\begin{theorem}\label{th:Z->Z-tilde}
For any $m>2/p$ and  $N\ge 0$ we have that 
$\partial_z : \mathcal{Z}^{m+1,p}_{1,N}\to\widetilde{\mathcal{Z}}^{m,p}_{N+1}$ is 
an isomorphism of Banach spaces.
\end{theorem}

\begin{remark}\label{rem:Z->Z-tilde}
Similar statement holds for the Cauchy operator $\partial_{\bar{z}}$ if we replace $\widetilde{\mathcal{Z}}^{m,p}_{N+1}$
by the space 
\[
\overline{\Big(\widetilde{\mathcal{Z}}^{m,p}_{N+1}\Big)}=
\Big\{\frac{\chi}{\bar{z}^2}\sum_{0\le k+l\le N-1}\frac{a_{kl}}{z^k\bar{z}^l}+f\,\Big|\,
f\in W_{\gamma_N+1}^{m,p}\,\,\text{\rm and}\,\,a_{kl}\in\C\Big\}.
\]
Moreover, one easily sees from Theorem \ref{th:Z->Z-tilde} and the equality $\partial_{\bar{z}}=\tau\circ\partial_z\circ\tau$
where $\tau : u\mapsto\bar{u}$ is the operation of taking the complex conjugate of a function that 
$\partial_{\bar{z}}^{-1}=\tau\circ\partial_z^{-1}\circ\tau$, and similarly 
$\partial_z^{-1}=\tau\circ\partial_{\bar{z}}^{-1}\circ\tau$. Finally, note that $\partial_z^{-1}$ and $\partial_{\bar{z}}^{-1}$ 
commute, since the Cauchy operators $\partial_z$ and $\partial_{\bar{z}}$ commute. 
\end{remark}

\begin{proof}[Proof of Theorem \ref{th:Z->Z-tilde}]
Let us first show that the map is well defined. For any $k\ge 1$ and $l\ge 0$ we have
\begin{equation}\label{eq:derivative1}
\partial_z\Big(\frac{\chi}{z^k\bar{z}^l}\Big)=-k\,\frac{\chi}{z^2}\,\frac{1}{z^{k-1}\bar{z}^l}+g
\end{equation}
where $g:=\chi_z/z^k\bar{z}^l\in C^\infty_c\subseteq W^{m,p}_{\gamma_N+1}$. 
Similarly, $\partial_z\big(\chi/\bar{z}^l\big)=\chi_z/\bar{z}^l\in C^\infty_c\subseteq W^{m,p}_{\gamma_N+1}$.
This, together with the fact that $\partial_z : W^{m+1,p}_{\gamma_N}\to W^{m,p}_{\gamma_N+1}$, shows
that $\partial_z\big(\mathcal{Z}^{m+1,p}_{1,N}\big)\subseteq\widetilde{\mathcal{Z}}^{m,p}_{N+1}$. Hence, the map 
\begin{equation}\label{eq:Z->Z-tilde}
\partial_z : \mathcal{Z}^{m+1,p}_{1,N}\to\widetilde{\mathcal{Z}}^{m,p}_{N+1}
\end{equation}
is well defined. Further, note that $\mathcal{Z}^{m+1,p}_{1,N}\subseteq W^{1,p}_{\gamma_0}$ where $\gamma_0+(2/p)>0$.
Hence, the map \eqref{eq:Z->Z-tilde} is injective by Lemma \ref{lem:injectivity}. Let us now prove that the map is onto.
Take $u\in\widetilde{\mathcal{Z}}^{m,p}_{N+1}$,
\[
u=\frac{\chi}{z^2}\sum_{0\le k+l\le N-1}\frac{a_{kl}}{z^k\bar{z}^l}+f,\quad f\in W_{\gamma_N+1}^{m,p}.
\]
Then, we obtain from \eqref{eq:derivative1} that
\begin{equation*}
u=\partial_z\left(\sum_{0\le k+l\le N-1}\Big(-\frac{a_{kl}}{k+1}\Big)\frac{\chi}{z^{k+1}\bar{z}^l}\right)+
\big(f-\tilde{g}\big)=\partial_z\left(\sum_{0\le k+l\le N-1}\Big(-\frac{a_{kl}}{k+1}\Big)\frac{\chi}{z^{k+1}\bar{z}^l}+
\mathcal{L}\big(f-\tilde{g}\big)\right)
\end{equation*}
where $\tilde{g}\in C^\infty_c\subseteq W^{m,p}_{\gamma_N+1}$ and
$\mathcal{L} : W^{m,p}_{\gamma_N+1}\to\mathcal{Z}^{m+1,p}_{1,N}$ is the map in Proposition \ref{prop:asymptotic_operator}. 
Note that $\sum_{0\le k+l\le N-1}\Big(-\frac{a_{kl}}{k+1}\Big)\frac{\chi}{z^{k+1}\bar{z}^l}\in\mathcal{Z}^{m+1,p}_{1,N}$.
This shows that \eqref{eq:Z->Z-tilde} is a bijective map onto its image. The statement of the theorem then follows from the open mapping theorem.
\end{proof}

\medskip

The following result can be considered as a (global in time) existence and uniqueness theorem for spatial asymptotic (ODE) solutions of
an asymptotic vector field of class $\mathcal{Z}^{m,p}_N$ on $\R^2$. In particular, the result implies that the group 
$\mathcal{ZD}^{m,p}_N$ has a well defined Lie group exponential map 
$\mathop{\rm Exp}_{LG} : \mathcal{ZD}^{m,p}_N\to\mathcal{ZD}^{m,p}_N$.

\begin{lemma}\label{lem:ode}
Let $u\in C\left( [0,T],\mathcal{Z}^{m,p}_N \right)$ for some $T>0$ and $m>2+(2/p)$. 
Then there exists a unique solution $\varphi\in C^1\left([0,T],\mathcal{ZD}^{m,p}_N\right)$ of the equation
\begin{equation}\label{equtflow}
\dt{\varphi}=u\circ\varphi,\quad\varphi\big|_{t =0}=\idmap. 
\end{equation}
\end{lemma}

\begin{proof}[Proof of Lemma \ref{lem:ode}]
For any $t\in[0,T]$ and $\varphi\in\mathcal{ZD}^{m-1,p}_N$ denote $F(t,\varphi): = u(t)\circ\varphi$.
By Theorem \ref{th:ZD},  the map $F : [0,T]\times\mathcal{ZD}^{m-1,p}_N\to\mathcal{Z}^{m-1,p}_N$ and 
its partial derivative with respect to the second argument,
\begin{equation*}
D_2F: [0,T]\times \mathcal{ZD}^{m-1,p}_N\to\mathcal{L}(\mathcal{Z}^{m-1,p}_N,\mathcal{Z}^{m-1,p}_N) ,
\end{equation*}
are continuous. In particular, $F$ is locally Lipschitz and, by the existence and uniqueness theorems for solutions of 
ordinary differential equations (ODEs) in Banach spaces (e.g. see \cite{LangDiffGeo}), we obtain that for any $t_0\in[0,T]$
there exists $\varepsilon_0 >0$ such that there is a unique solution  
$\varphi\in C^1\big([t_0-\varepsilon_0,t_0+\varepsilon_0],\mathcal{ZD}^{m-1,p}_N\big)$ 
of \eqref{equtflow}. Take an arbitrary $\psi_0\in\mathcal{ZD}^{m-1,p}_N$. Since the right translation 
$R_{\psi_0} : \mathcal{ZD}^{m-1,p}_N\to\mathcal{ZD}^{m-1,p}_N$, $R_{\psi_0}(\varphi):=\varphi\circ\psi_0$,
is a $C^\infty$-smooth map\footnote{This follows from Proposition \ref{prop:composition} and the fact that the right translation
is an affine linear map.}, we obtain from the uniqueness theorem for solutions of ODEs in Banach spaces that the curve 
$\psi:=\varphi\circ\psi_0\in C^1\big([t_0-\varepsilon_0,t_0+\varepsilon_0],\mathcal{ZD}^{m-1,p}_N\big)$ 
is the unique solution of
\begin{equation*}
\dt{\psi}=u\circ\psi,\quad\psi\big |_{t=t_0}=\psi_0,
\end{equation*}
on the interval $[t_0-\varepsilon_0,t_0+\varepsilon_0]$.
This implies that equation \eqref{equtflow} has a unique solution
\begin{equation}\label{eq:varphi}
\varphi\in C^1\big([0,T],\mathcal{ZD}^{m-1,p}_N \big).
\end{equation}
It remains to show that $\varphi(t)\in\mathcal{ZD}^{m,p}_N$. By applying the pointwise derivative to equation 
(\ref{equtflow}) we obtain that
\begin{equation*}
(d\varphi)^\bolddot=du\circ\varphi\cdot d\varphi.
\end{equation*}
Consider the linear system $\dt{Y}=A(t)\,Y$, $Y\big|_{t=0}=I$, where $A(t):=d(u(t))\circ\varphi$. 
Since $A : [0,T]\to\mathcal{Z}^{m-1,p}_{N+1}$ is continuous, this linear system has a unique solution 
$Y\in C\big([0,T],\mathcal{Z}^{m-1,p}_{N+1})$. This implies that $d\varphi\in C\big([0,T],\mathcal{Z}^{m-1,p}_{N+1})$.
Since
\[
d\varphi=
\begin{pmatrix}
1+w_z&w_{\bar{z}}\\
\overline{w}_z&1+\overline{w}_{\bar{z}}
\end{pmatrix}
\]
where $\varphi=z+w$ and $w\in\mathcal{Z}^{m-1,p}_N$, we conclude that
$w_z\in C\big([0,T],\mathcal{Z}^{m-1,p}_{N+1}\big)$, and hence
\begin{equation}\label{eq:w_z}
\partial_zw\equiv w_z\in C\big([0,T],\widetilde{\mathcal{Z}}^{m-1,p}_{N+1}\big).
\end{equation}
Finally, by using that by Theorem \ref{th:Z->Z-tilde} the map 
$\partial_z : \mathcal{Z}^{m,p}_{1,N}\to\widetilde{\mathcal{Z}}^{m-1,p}_{N+1}$
is a linear isomorphism of Banach spaces, we obtain from \eqref{eq:varphi} and \eqref{eq:w_z}
that $w\in C\big([0,T],\mathcal{Z}^{m,p}_N\big)$. Hence, $\varphi\in C\big([0,T],\mathcal{ZD}^{m,p}_N\big)$.
This completes the proof of the Lemma.
\end{proof}

\medskip

The following two propositions are used in the proof of Theorem \ref{th:main_introduction} that was stated
in the Introduction.

\begin{lemma}\label{lem:cauchy-conjugate}
For $m>3+(2/p)$ the map 
\begin{equation}\label{eq:F}
F : (\varphi,v)\mapsto\big(R_\varphi\circ\partial_z\circ R_{\varphi^{-1}}\big)(v),\quad
\mathcal{ZD}^{m,p}_N\times\mathcal{Z}^{m,p}_N\to\widetilde{\mathcal{Z}}^{m-1,p}_{N+1},
\end{equation}
is real analytic (see Remark \ref{rem:real_analyticity}).
\end{lemma}

\begin{remark}\label{rem:real_analyticity}
Note that the asymptotic space $\mathcal{Z}^{m,p}_N$ is a complex Banach space
and $\mathcal{ZD}^{m,p}_N$ is a complex manifold modeled on $\mathcal{Z}^{m,p}_N$.
The map \eqref{eq:F} is {\em not} real analytic with respect to the complex structure on 
$\mathcal{Z}^{m,p}_N$ and $\mathcal{ZD}^{m,p}_N$ but it becomes a real analytic map 
if we ignore the complex structure on $\mathcal{Z}^{m,p}_N$ and $\mathcal{ZD}^{m,p}_N$
and consider them as real Banach manifolds. In what follows, real analyticity is understood in this sense.
\end{remark}

\begin{proof}[Proof of Lemma \ref{lem:cauchy-conjugate}]
Take $\varphi\in\mathcal{ZD}^{m,p}_N$ such that $\varphi=z+u$ with $u\in\mathcal{Z}^{m,p}_N$ and let $v\in\mathcal{Z}^{m,p}_N$. 
The fact that 
\begin{equation}\label{eq:well-defined}
\big(R_\varphi\circ\partial_z\circ R_{\varphi^{-1}}\big)(v)\in\widetilde{\mathcal{Z}}^{m-1,p}_{N+1}
\end{equation}
follows from Theorem \ref{th:ZD}, Theorem \ref{th:Z->Z-tilde} and the second statement of Lemma \ref{lem:composition}.
Then, $\varphi$ and $v$ are $C^1$-smooth maps and by the chain rule 
$\partial_z\big(R_{\varphi^{-1}}(v)\big)=v_z\circ\psi\cdot\psi_z+v_{\bar{z}}\circ\psi\cdot \overline{\psi}_z$
where we set $\psi:=\varphi^{-1}$. This implies that
\begin{equation*}
R_{\varphi}\Big(\partial_z\big(R_{\varphi^{-1}}(v)\big)\Big)=v_z\cdot\psi_z\circ\varphi+v_{\bar{z}}\cdot\overline{\psi}_z\circ\varphi.
\end{equation*}
By taking the derivative in the equality $\psi\circ\varphi=\idmap$, we obtain that
\[
\begin{pmatrix}
\varphi_z&\overline{\varphi}_z\\
\varphi_{\bar{z}}&\overline{\varphi}_{\bar{z}}
\end{pmatrix}
\begin{pmatrix}
\psi_z\circ\varphi\\
\psi_{\bar{z}}\circ\varphi\\
\end{pmatrix}
=
\begin{pmatrix}
1\\
0
\end{pmatrix}.
\]
This implies that $\psi_z\circ\varphi=\overline{\varphi}_{\bar{z}}/{\rm D}$ and
$\overline{\psi}_z\circ\varphi=-\overline{\varphi}_z/{\rm D}$
where 
\begin{equation}\label{eq:D}
{\rm D}\equiv{\rm D}(\varphi):=\varphi_z\bar{\varphi}_{\bar{z}}-\varphi_{\bar{z}}\bar{\varphi}_z
=1+\big[(u_z+\bar{u}_{\bar{z}})+(u_z\bar{u}_{\bar{z}}+u_{\bar{z}}\bar{u}_z)\big]
\end{equation}
is the determinant of the Jacobian matrix of the diffeomorphism $\varphi$.
This implies that
\begin{equation}\label{eq:conjugation}
\big(R_{\varphi}\circ\partial_z\circ R_{\varphi^{-1}}\big)(v)=
\frac{v_z\overline{\varphi}_{\bar{z}}-v_{\bar{z}}\overline{\varphi}_z}{{\rm D}(\varphi)}
=\frac{v_z+v_z\bar{u}_{\bar{z}}-v_{\bar{z}}\bar{u}_z}{{\rm D}(\varphi)}.
\end{equation}
Note that ${\rm D}$ does not vanish on $\C$ since $\varphi$ is a diffeomorphism.
Moreover, by \eqref{eq:D}, the determinant ${\rm D}\in\mathcal{Z}^{m-1,p}_{N+1}$ has 
constant asymptotic term equal to one. This and the Banach algebra property in $\mathcal{Z}^{m-1,p}_{N+1}$ allow us
to conclude from Lemma \ref{lem:division} that the map 
\begin{equation}\label{eq:denominator}
\varphi\mapsto 1/{\rm D}(\varphi),\quad\mathcal{ZD}^{m,p}_N\to\mathcal{Z}^{m-1,p}_{N+1},
\end{equation}
is real analytic. By the Banach algebra property in $\mathcal{Z}^{m-1,p}_{N+1}$ the map
\begin{equation}\label{eq:numerator}
(\varphi,v)\mapsto v_z+v_z\bar{u}_{\bar{z}}-v_{\bar{z}}\bar{u}_z,\quad
\mathcal{ZD}^{m,p}_N\times\mathcal{Z}^{m,p}_N\to\mathcal{Z}^{m-1,p}_{N+1},
\end{equation}
is real analytic. Then, by combining this with \eqref{eq:conjugation}, \eqref{eq:denominator}, \eqref{eq:numerator}, and 
\eqref{eq:well-defined}, we conclude the proof of the proposition.
\end{proof}

As a corollary we obtain

\begin{proposition}\label{prop:cauchy-conjugate}
For $m>3+(2/p)$ the map
\begin{equation}\label{eq:congugate_antiderivative}
(\varphi,v)\mapsto\big(R_\varphi\circ\partial_z^{-1}\circ R_{\varphi^{-1}}\big)(v),\quad
\mathcal{ZD}^{m,p}_N\times\widetilde{\mathcal{Z}}^{m-1,p}_{N+1}\to\mathcal{Z}^{m,p}_{1,N},
\end{equation}
is real analytic. Here $\partial_z^{-1}$ denotes the inverse of the Cauchy operator given be Theorem \ref{th:Z->Z-tilde}.
\end{proposition}

\begin{proof}[Proof of Proposition \ref{prop:cauchy-conjugate}]
First, note that by Theorem \ref{th:ZD}, Theorem \ref{th:Z->Z-tilde} and Lemma \ref{lem:composition} the map 
$v\mapsto\big(R_\varphi\circ\partial_z\circ R_{\varphi^{-1}}\big)(v)$,
$\mathcal{Z}^{m,p}_{1,N}\to\widetilde{\mathcal{Z}}^{m-1,p}_{N+1}$, is a bounded linear map. 
By Theorem \ref{th:Z->Z-tilde}, this map is a linear isomorphism with inverse
\begin{equation}\label{eq:the_inverse}
\big(R_\varphi\circ\partial_z\circ R_{\varphi^{-1}}\big)^{-1}=R_\varphi\circ\partial_z^{-1}\circ R_{\varphi^{-1}}.
\end{equation}
In particular, 
$R_\varphi\circ\partial_z\circ R_{\varphi^{-1}}\in GL\big(\mathcal{Z}^{m,p}_{1,N},\widetilde{\mathcal{Z}}^{m-1,p}_{N+1}\big)$
where $GL(X,Y)$ denotes the Banach manifold of linear isomorphisms between two Banach spaces $X$ and $Y$.
On the other side, since the map \eqref{eq:F} in Lemma \ref{lem:cauchy-conjugate} is real analytic, we obtain that
its first partial derivative $G\equiv D_2F$ with respect to the second argument
\begin{equation}\label{eq:G}
G : \mathcal{ZD}^{m,p}_N\to GL\big(\mathcal{Z}^{m,p}_{1,N},\widetilde{\mathcal{Z}}^{m-1,p}_{N+1}\big),\quad
\varphi\mapsto R_\varphi\circ\partial_z\circ R_{\varphi^{-1}},
\end{equation}
is also real analytic. Note in addition that the map
\[
\imath : GL\big(\mathcal{Z}^{m,p}_{1,N},\widetilde{\mathcal{Z}}^{m-1,p}_{N+1}\big)\to 
GL\big(\widetilde{\mathcal{Z}}^{m-1,p}_{N+1},\mathcal{Z}^{m,p}_{1,N}\big),\quad L\mapsto L^{-1},
\]
is analytic. It then follows from \eqref{eq:the_inverse} and \eqref{eq:G} that the map 
\[
\imath\circ G : \varphi\mapsto R_\varphi\circ\partial_z^{-1}\circ R_{\varphi^{-1}},\quad
\mathcal{ZD}^{m,p}_N\to GL\big(\widetilde{\mathcal{Z}}^{m-1,p}_{N+1},\mathcal{Z}^{m,p}_{1,N}\big),
\]
is real analytic as it is a composition of real analytic maps. This implies that the map \eqref{eq:congugate_antiderivative} is
also real analytic.
\end{proof}

We will see in Section \ref{sec:euler_equation} that the non-linear map
\[
Q : u\mapsto Q(u):=(u_z)^2+u_{\bar{z}}\bar{u}_z,\quad\mathcal{Z}^{m,p}_N\to\widetilde{\mathcal{Z}}^{m-1,p}_{N+1},
\]
plays an important role in studying the well-posedness of the Euler equation.
The fact that $Q(u)\in\widetilde{\mathcal{Z}}^{m-1,p}_{N+1}$ for $u\in\mathcal{Z}^{m,p}_N$ follows 
easily from the fact that $\widetilde{\mathcal{Z}}^{m-1,p}_{N+1}$ is an ideal in ${\mathcal{Z}}^{m-1,p}_{N+1}$ that is closed
under complex conjugation. By arguing as in the proof of Lemma \ref{lem:cauchy-conjugate} we obtain the following

\begin{proposition}\label{prop:Q-conjugate}
For $m>3+(2/p)$ the map 
\[
(\varphi, v)\mapsto\big(R_{\varphi}\circ Q\circ R_{\varphi^{-1}}\big)(v),\quad 
\mathcal{ZD}^{m,p}_N\times\mathcal{Z}^{m,p}_N\to\widetilde{\mathcal{Z}}^{m-1,p}_{N+1},
\] 
is real analytic.
\end{proposition}

\begin{proof}[Proof of Proposition \ref{prop:Q-conjugate}]
Take $\varphi\in\mathcal{ZD}^{m,p}_N$ and $v\in\mathcal{Z}^{m,p}_N$.
By a direct differentiation we obtain that $\big(R_{\varphi}\circ Q\circ R_{\varphi^{-1}}\big)(v)$ is a sum of two terms:
$\left(\big(R_{\varphi}\circ\partial_z\circ R_{\varphi^{-1}}\big)(v)\right)^2$ and
\begin{equation*}
R_{\varphi}\left(\left(\partial_{\bar{z}}(v\circ\varphi^{-1})\right)\cdot\left(\partial_z (\bar{v}\circ\varphi^{-1})\right)\right)=  
\overline{\left(\big(R_{\varphi}\circ\partial_z\circ R_{\varphi^{-1}}\big)(\bar{v}) \right)}\cdot
\left(\big((R_{\varphi}\circ\partial_z\circ R_{\varphi^{-1}}\big)(\bar{v})\right).
\end{equation*}
The statement of the proposition then follows from Lemma \ref{lem:cauchy-conjugate} and the Banach algebra property in
$\mathcal{Z}^{m-1,p}_{N+1}$.
\end{proof}

\section{Local existence and uniqueness of solutions in $\mathcal{Z}^{m,p}_N$}\label{sec:euler_equation}
In this section we prove a local version of Theorem \ref{th:main_introduction} stated in the Introduction.
As shown in Appendix \ref{appendix:complex_form}, the 2d Euler's equation \eqref{eq:euler} 
represented in complex form is
\begin{equation}\label{EQeulerCauchyop}
\left\{
\begin{array}{l}
u_t+(u u_z+{\bar u} u_{\bar z})=-2\p_{\bar z},\quad\divop u\equiv u_z+{\bar u}_{\bar z}=0,\\
u|_{t=0}=u_0,
\end{array}
\right.
\end{equation}
where $\p: \R^2 \rightarrow \R$ is the scalar pressure. 
In order to obtain a {\em unique} solution of \eqref{EQeulerCauchyop} in $\mathcal{Z}^{m,p}_N$ we will require that 
$\p\in\mathcal{Z}^{m+1,p}_{1,N}$. We will first prove the following proposition.

\begin{proposition}\label{PropEqEulerDynSys}
Assume that $m>2+(2/p)$ and 
$u\in C\big([0,T],\mathcal{Z}^{m,p}_N\big)\cap C^1\big([0,T],\mathcal{Z}^{m-1,p}_N\big)$. 
Then $u$ satisfies (\ref{EQeulerCauchyop}) for some real valued $\p\equiv\p(t)\in\mathcal{Z}^{m+1,p}_{1,N}$, $t\in[0,T]$, 
if and only if $u$ satisfies
\begin{equation}\label{EQEulerwoDiv}
\dt{u}+(uu_z+\bar{u}u_{\bar{z}})=\partial_z^{-1}\big((u_z)^2+|u_{\bar{z}}|^2\big) 
\end{equation}
with divergence free initial data $u_0=u|_{t=0}$, $\divop u_0=0$. 
Here $\partial_z^{-1} : \widetilde{\mathcal{Z}}^{m-1,p}_{N+3}\to\mathcal{Z}^{m,p}_{1,N+2}$
denotes the operator constructed in Theorem \ref{th:Z->Z-tilde}.
\end{proposition}

\begin{remark}
Note that for $u\in\mathcal{Z}^{m,p}_N$ with $m>2+(2/p)$ we have that 
$u_z,\bar{u}_z\in\widetilde{\mathcal{Z}}^{m-1,p}_{N+1}\subseteq\mathcal{Z}^{m-1,p}_{2,N+1}$ and 
$u_{\bar{z}}\in\mathcal{Z}^{m-1,p}_{2,N+1}$.
Since $\widetilde{\mathcal{Z}}^{m-1,p}_{N+1}$ is an ideal in the Banach algebra $\mathcal{Z}^{m-1,p}_{N+1}$ 
(cf. Proposition \ref{prop:properties_W-spaces} and Proposition \ref{prop:properties_Z-spaces}) we conclude that
\[
Q(u)\equiv u_z u_z + \bar{u}_zu_{\bar{z}}\in\widetilde{\mathcal{Z}}^{m-1,p}_{N+1}.
\]
This shows that the term $\partial_z^{-1}\big((u_z)^2+|u_{\bar{z}}|^2\big)$ on the right hand side of 
\eqref{EQEulerwoDiv} is well defined and belongs to $\mathcal{Z}^{m,p}_{1,N}$.
\end{remark}

\begin{proof}[Proof of Proposition \ref{PropEqEulerDynSys}]
Assume $u\in C\big([0,T],\mathcal{Z}^{m,p}_N\big)\cap C^1\big([0,T],\mathcal{Z}^{m-1,p}_N\big)$ satisfies equation 
\eqref{EQeulerCauchyop} for some {\em real} valued $\p\in\mathcal{Z}^{m+1,p}_{1,N}$. 
Then, by applying the divergence operator to the both hand sides of \eqref{EQeulerCauchyop} and then using that
$\divop\dt{u}=(\divop u)^\bolddot=0$ we conclude that $\divop(uu_z+\bar{u}u_{\bar{z}})=-2\divop(\p_{\bar{z}})$.
We have
\begin{equation*}
\divop(\p_{\bar{z}})=\partial_z(\p_{\bar{z}})+\partial_{\bar{z}}\overline{(\p_{\bar{z}})}= 
2\partial_z \partial_{\bar{z}}\p.
\end{equation*}
The computation in \eqref{eq:Q} implies that for $m>2+(2/p)$ we have
\begin{equation}\label{eq:Q-divergence}
\divop(uu_z+\bar{u}u_{\bar{z}})=2\big((u_z)^2+|u_{\bar{z}}|^2\big)=2Q(u).
\end{equation}
Hence,
\begin{equation}\label{eq:Delta p=Q}
-2\partial_z(\p_{\bar{z}})=(u_z)^2+|u_{\bar{z}}|^2.
\end{equation}
By using the condition that $\p\in\mathcal{Z}^{m+1,p}_{1,N}$ we conclude that 
$\p_{\bar{z}}\in\mathcal{Z}^{m,p}_{2,N+1}\subseteq\mathcal{Z}^{m,p}_{1,N+1}$ and hence
by Theorem \ref{th:Z->Z-tilde} we obtain
\[
-2\p_{\bar{z}}=\partial_z^{-1}\big((u_z)^2+|u_{\bar{z}}|^2\big).
\]
This, together with \eqref{EQeulerCauchyop}, then proves the direct statement of the proposition.
Let us now prove the converse statement. 
Assume that $u\in C\big([0,T],\mathcal{Z}^{m,p}_N\big)\cap C^1\big([0,T],\mathcal{Z}^{m-1,p}_N\big)$ satisfies
equation \eqref{EQEulerwoDiv} with divergence free initial data $\divop u_0=0$.
By applying the divergence operator to the both hand sides of \eqref{EQEulerwoDiv} we obtain, after a simplification, that
\begin{equation}\label{eq:t-derivative}
(\divop{u})^\bolddot+u\,\partial_z(\divop{u})+\bar{u}\,\partial_{\bar{z}}(\divop{u})=0.
\end{equation}
Let $\varphi\in C^1\big([0,T],\mathcal{ZD}^{m,p}_N\big)$ be the solution of the equation
$\dt\varphi=u\circ\varphi$, $\varphi|_{t=0}=\idmap$ that exists by Lemma \ref{lem:ode}.
Since $\divop u\in C\big([0,T],\mathcal{Z}^{m-1,p}_N\big)\cap C^1\big([0,T],\mathcal{Z}^{m-2,p}_N\big)$ and 
$\varphi\in C^1\big([0,T],\mathcal{ZD}^{m,p}_N\big)$ with $m>2+(2/p)$ we conclude that 
$\divop u\in C^1\big([0,T]\times\C,\R\big)$ and $\varphi\in C^1(\C,\C)$. This, together with \eqref{eq:t-derivative},
implies that for any given $(x,y)\in\R^2$ and for any $t\in[0,T]$,
\begin{equation}\label{eq:pointwise}
\partial_t\big((\divop u)\circ\varphi\big)=0,
\end{equation}
where $\partial_t$ denotes the pointwise partial derivative in the direction of the variable $t$.
Hence, $\divop u(t)=\divop u_0=0$ for any $t\in[0,T]$.
By Lemma \ref{LemmaPressureTermAsymp} below there exists $\p\in C\big([0,T],\mathcal{Z}^{m+1,p}_{1,N}\big)$
such that $-2\p_{\bar{z}}=\partial_z^{-1}\big((u_z)^2+|u_{\bar{z}}|^2\big)$ for any $t\in[0,T]$. 
This completes the proof of the proposition.
\end{proof}

In the proof of Proposition \ref{PropEqEulerDynSys} we used the following lemma.

\begin{lemma}\label{LemmaPressureTermAsymp}
Assume that $m>2+(2/p)$. Then, for any $u\in\mathcal{Z}^{m,p}_N$ divergence free there exists 
a real valued $\p\in\mathcal{Z}^{m+1,p}_{1,N}$ such that $-2\p_{z\bar{z}} =(u_z)^2 + |u_{\bar{z}}|^2$.
Moreover, we have that $-2\p_{\bar{z}}=\partial_z^{-1}\big((u_z)^2+|u_{\bar{z}}|^2\big)$
and $-2\p=\partial_{\bar{z}}^{-1}\partial_z^{-1}\big((u_z)^2+|u_{\bar{z}}|^2\big)$.
\end{lemma}

\begin{proof}[Proof of Lemma \ref{LemmaPressureTermAsymp}]
Take $u\in\mathcal{Z}^{m,p}_N$ such that $\divop u=u_z+\bar{u}_{\bar{z}}=0$.
Then we have
\begin{equation}\label{eq:Q-formula}
Q(u)\equiv(u_z)^2+u_{\bar{z}}\bar{u}_z=
-u_z\overline{(u_z)}+\bar{u}_z\overline{(\bar{u}_z)}.
\end{equation}
We will first consider the case when $N\ge 1$. 
Since $u_z,\bar{u}_{z}\in\widetilde{\mathcal{Z}}^{m-1,p}_{N+1}$ we obtain from 
Proposition \ref{prop:properties_W-spaces}, Proposition \ref{prop:properties_Z-spaces} and \eqref{eq:Q-formula} that $
Q(u)\in\mathcal{Z}^{m-1,p}_{4,N+3}$ and
\begin{equation}\label{eq:QinZ}
Q(u)=\frac{\chi}{z^2\bar{z}^2}\sum_{0\le k+l\le N-1}\frac{a_{kl}}{z^k\bar{z}^l}+f,\quad f\in W^{m-1,p}_{\gamma_N+3}.
\end{equation}
In particular, we see that $Q(u)\in W^{m-1,p}_{\gamma_3}$. Then, by Proposition \ref{prop:asymptotic_operator},
$\partial_z^{-1}Q(u)\in\mathcal{Z}^{m,p}_{1,2}$ and
\[
\partial_z^{-1}Q(u)=\frac{1}{\pi}\Big[\big(Q(u),1\big)\,\frac{\chi}{\bar{z}}+\big(Q(u),\bar{z}\big)\,\frac{\chi}{\bar{z}^2}\Big]+g,\quad
g\in W^{m,p}_{\gamma_2}.
\]
Since $\big(Q(u),1\big)=\big(Q(u),\bar{z}\big)=0$ by Lemma \ref{lem:vanishing_momenta} below, we conclude that
\begin{equation}\label{eq:integral of Q'}
\partial_z^{-1}Q(u)\in W^{m,p}_{\gamma_2}.
\end{equation}
On the other side, by Theorem \ref{th:Z->Z-tilde} and the fact that  $Q(u)\in\mathcal{Z}^{m-1,p}_{4,N+3}$ we have that
$\partial_z^{-1}Q(u)\in\mathcal{Z}^{m,p}_{1,N+2}$. Moreover, it follows from \eqref{eq:QinZ} and 
Proposition \ref{prop:asymptotic_operator} (cf. the proof of Theorem \ref{th:Z->Z-tilde}) that
\begin{equation*}
\partial_z^{-1}Q(u)=\chi\,\frac{b_1}{\bar{z}}+\chi\,\frac{b_2}{\bar{z}^2}
+\left(\frac{\chi}{z\bar{z}^2}\sum_{0\le k+l\le N-1}\frac{b_{kl}}{z^k\bar{z}^l}
+\frac{\chi}{\bar{z}^3}\sum_{3\le k\le N+2}\frac{b_k}{\bar{z}^{k-3}}\right)
+h,\quad h\in W^{m,p}_{\gamma_N+2}
\end{equation*}
where $b_k:=\frac{1}{\pi}\big(f,\bar{z}^{k-1}\big)$.
By comparing this with \eqref{eq:integral of Q'}, we obtain that $b_1=b_2=0$, and hence
\begin{align}\label{eq:integral of Q}
\partial_z^{-1}Q(u)=\left(\frac{\chi}{z\bar{z}^2}\sum_{0\le k+l\le N-1}\frac{b_{kl}}{z^k\bar{z}^l}
+\frac{\chi}{\bar{z}^3}\sum_{3\le k\le N+2}\frac{b_k}{\bar{z}^{k-3}}\right)
+h,\quad h\in W^{m,p}_{\gamma_N+2}.
\end{align}
In particular, this implies that
\[
\partial_z^{-1}Q(u)\in\overline{\Big(\widetilde{\mathcal{Z}}^{m,p}_{N+2}\Big)}.
\]
Hence, by Theorem \ref{th:Z->Z-tilde} and Remark \ref{rem:Z->Z-tilde}, we obtain that
\[
\p:=-\frac{1}{2}\,\partial_{\bar{z}}^{-1}\partial_z^{-1}Q(u)\in\mathcal{Z}^{m+1,p}_{1,N+1}\subseteq\mathcal{Z}^{m+1,p}_{1,N}.
\]
Let us now consider the case when $N=0$. Then, $u\in\mathcal{Z}^{m,p}_0$ and $u=a_{00}+f$, $f\in W^{m,p}_{\gamma_0}$.
This implies that $u_z,\bar{u}_z\in W^{m-1,p}_{\gamma_1}$ and hence by Proposition \ref{prop:properties_W-spaces} and 
\eqref{eq:Q-formula} we obtain that
\[
Q(u)\in W^{m-1,p}_{2\gamma_1+(2/p)}\subseteq W^{m-1,p}_{\gamma_2}
\]
where we used that $1<\gamma_1+(2/p)<2$. This, together with Proposition \ref{prop:asymptotic_operator} and 
Lemma \ref{lem:vanishing_momenta} below, implies that
\[
\partial_z^{-1}Q(u)\in W^{m,p}_{\gamma_1}.
\]
Finally, by applying Theorem \ref{th:cauchy_m>=0} (i) we obtain that
\[
\p:=-\frac{1}{2}\,\partial_{\bar{z}}^{-1}\partial_z^{-1}Q(u)\in W^{m+1,p}_{\gamma_0}\equiv\mathcal{Z}^{m+1,p}_{1,0}.
\]
Let us now prove that $\p$ is real valued. Recall that $\tau : u\mapsto\bar{u}$ is the operation of taking the complex
conjugate of a function. By Remark \ref{rem:Z->Z-tilde} we then have
\begin{eqnarray*}
-2\overline{\p}&=&\tau\circ\big(\partial_{\bar{z}}^{-1}\partial_z^{-1}Q(u)\big)=
\big(\tau\circ\partial_{\bar{z}}^{-1}\circ\tau\big)\circ\big(\tau\circ\partial_z^{-1}\tau\big)\circ\tau Q(u)\\
&=&\partial_{\bar{z}}^{-1}\partial_z^{-1}Q(u)=-2\p
\end{eqnarray*}
where we used that $Q(u)$ is real valued by \eqref{eq:Q-formula} and that 
$\partial_z^{-1}=\tau\circ\partial_{\bar{z}}^{-1}\circ\tau$ and similarly
$\partial_{\bar{z}}^{-1}=\tau\circ\partial_z^{-1}\circ\tau$.
This completes the proof of the proposition.
\end{proof}



\begin{lemma}\label{lem:vanishing_momenta}
Assume that $m>1+(2/p)$ and $N\ge 0$. If $u\in\mathcal{Z}^{m,p}_N$ is divergence free then 
$\big(Q(u),1\big)=0$ and $\big(Q(u),\bar{z}\big)=0$.
\end{lemma}

\begin{proof}[Proof of Lemma \ref{lem:vanishing_momenta}]
First, assume that $m>2+(2/p)$ and $N\ge 1$. Since $Q(u)$ does not involve the leading asymptotic term of 
$u\in\mathcal{Z}^{m,p}_N$ we will assume without loss of generality that $u\in\mathcal{Z}^{m,p}_{1,N}$.
Then, the lemma follows easily from the Stokes' theorem and the relation \eqref{eq:Q-divergence}. 
In fact, with $f:=u u_z+\bar{u}u_{\bar{z}}$ we obtain from \eqref{eq:Q-divergence} that
\begin{equation*}
2 \big(Q(u),1\big)=
-\frac{1}{2i}\int_{\C}\divop f\,dz\wedge d\bar{z}
=-\mathop{\rm Im}\left(\int_{\C}f_z\,dz\wedge d\bar{z}\right)
=-\mathop{\rm Im}\left(\lim_{R\to\infty}\oint_{|z|=R}f\,d\bar{z}\right)=0
\end{equation*}
where we used that $f=O\big(1/|z|^3\big)$ and $Q(u)=O\big(1/|z|^4\big)$ as $|z|\to\infty$.
Similarly,
\begin{equation}\label{eq:z-bar momenta}
2 \big(Q(u),\bar{z}\big)=
-\frac{1}{2i}\int_{\C}\bar{z}\divop f\,dz\wedge d\bar{z}
=-\frac{1}{2i}\int_{\C}\bar{z}(f_z+\bar{f}_{\bar{z}})\,dz\wedge d\bar{z}
=\frac{1}{2i}\int_{\C}\bar{f}\,dz\wedge d\bar{z}
\end{equation}
where we used the Stokes' theorem to conclude that
\[
\int_{\C}\bar{z}f_z\,dz\wedge d\bar{z}=\int_{\C}\big(\bar{z}f\big)_z\,dz\wedge d\bar{z}=0
\]
and
\[
\int_{\C}\bar{z}\bar{f}_{\bar{z}}\,dz\wedge d\bar{z}=
\int_{\C}\big((\bar{z}\bar{f})_{\bar{z}}-\bar{f}\big)\,dz\wedge d\bar{z}=-\int_{\C}\bar{f}\,dz\wedge d\bar{z}.
\]
On the other side, since $\divop u=u_z+\bar{u}_{\bar{z}}=0$, we have
\[
f=u u_z+\bar{u}u_{\bar{z}}=\frac{1}{2}(u^2)_z+(u\bar{u})_{\bar{z}}-u\bar{u}_{\bar{z}}
=\frac{1}{2}(u^2)_z+(u\bar{u})_{\bar{z}}+uu_z=(u^2)_z+(u\bar{u})_{\bar{z}}.
\]
Since $u^2,u\bar{u}=O\big(1/|z|^2\big)$ as $|z|\to\infty$ we then conclude again from Stokes' theorem that
$\int_{\C}\bar{f}\,dz\wedge d\bar{z}=0$.
This, together with \eqref{eq:z-bar momenta}, then implies that $\big(Q(u),\bar{z}\big)=0$.
Finally, the case when $1+(2/p)<m\le 2+(2/p)$ and $N=0$ follows by continuity since 
$\mathcal{Z}^{m+1,p}_{N+1}$ is densely embedded in $\mathcal{Z}^{m,p}_N$ and since the map
$Q : \mathcal{Z}^{m,p}_N\to\mathcal{Z}^{m-1,p}_{4,N+3}$ and the inclusions
$\mathcal{Z}^{m-1,p}_{4,N+3}\subseteq W^{m-1,p}_{\gamma_3}\subseteq L^1_1$ are bounded.
\end{proof}

\begin{proposition}\label{PropBijecDynSysDivFreeEulerEq}
Let $m>3+(2/p)$ and  $u_0\in\mathcal{Z}^{m,p}_N$. There is a bijection between solutions of the dynamical system 
\begin{equation}\label{EQDynSys}
\left\{
\begin{array}{l}
(\dt{\varphi},\dt{v})=\Big(v,\big(R_{\varphi}\circ \partial_z^{-1} \circ Q \circ R_{\varphi^{-1}}\big)(v)\Big)\\
(\varphi,v)|_{t=0}=(\idmap, u_0)
\end{array}
\right.
\end{equation}
on $\mathcal{ZD}^{m,p}_N\times\mathcal{Z}^{m,p}_N$ and solutions of \eqref{EQEulerwoDiv}. More specifically,
$(\varphi,v)\in C^1\big([0,T],\mathcal{ZD}^{m,p}_N\times\mathcal{Z}^{m,p}_N\big)$ is a solution of \eqref{EQDynSys}
if and only if 
$u:=v\circ\varphi^{-1}\in C\big([0,T],\mathcal{Z}^{m,p}_N\big)\cap C^1\big([0,T],\mathcal{Z}^{m-1,p}_N\big)$ 
is a solution of \eqref{EQEulerwoDiv}.
\end{proposition}

\begin{remark}\label{rem:smooth_vector_field}
Since $\big(R_{\varphi}\circ\partial_z^{-1}\circ Q\circ R_{\varphi^{-1}}\big)(v)=
\big(R_{\varphi}\circ\partial_z^{-1}\circ R_{\varphi^{-1}}\big)\circ
\big(R_{\varphi}\circ Q\circ R_{\varphi^{-1}}\big)(v)$ for $(\varphi,v)\in\mathcal{ZD}^{m,p}_N\times\mathcal{Z}^{m,p}_N$,
we obtain from Proposition \ref{prop:cauchy-conjugate} and Proposition \ref{prop:Q-conjugate}
that the right hand side of \eqref{EQDynSys} is well defined real analytic vector field on 
$\mathcal{ZD}^{m,p}_N\times\mathcal{Z}^{m,p}_N$  (cf. \cite{EM}).
\end{remark}

\begin{proof}[Proof of Proposition \ref{PropBijecDynSysDivFreeEulerEq}]
First, assume that 
$(\varphi,v)\in C^1\big([0,T],\mathcal{ZD}^{m,p}_N\times\mathcal{Z}^{m,p}_N\big)$ 
is a solution of the dynamical system (\ref{EQDynSys}). Then, by Theorem \ref{th:ZD},
$u:=v\circ\varphi^{-1}\in C\big([0,T],\mathcal{Z}^{m,p}_N\big)\cap C^1\big([0,T],\mathcal{Z}^{m-1,p}_N\big)$. 
We have
\begin{equation*}
\dt{v}=(u\circ\varphi)^\bolddot=\dt{u}\circ\varphi+u_z\circ\varphi\cdot\dt{\varphi}+
u_{\bar{z}}\circ\varphi\cdot\dt{\overline{\varphi}}=\big(\dt{u}+u_z u+u_{\bar{z}}\bar{u}\big)\circ\varphi. 
\end{equation*}
This gives
\[
\dt{u}+u_z u+u_{\bar{z}}\bar{u}=R_{\varphi^{-1}}(\dt{v})=  
R_{\varphi^{-1}}\Big(\big(R_{\varphi}\circ\partial_z^{-1}\circ Q\circ R_{\varphi^{-1}}\big)(v)\Big)=\partial_z^{-1}Q(u).
\]
Hence, $u$ solves equation \eqref{EQEulerwoDiv}.
Conversely, assume that $u\in C\big([0,T],\mathcal{Z}^{m,p}_N\big)\cap C^1\big([0,T],\mathcal{Z}^{m-1,p}_N\big)$ 
solves equation \eqref{EQEulerwoDiv}. Let $\varphi\in C^1([0,T],\mathcal{ZD}^{m,p}_N)$ be the solution of
$\dt\varphi=u\circ\varphi$, $\varphi|_{t=0}=\idmap$, given by Proposition \ref{lem:ode}.
By the Sobolev embedding $\mathcal{Z}^{m-1,p}_N \subseteq C^1$, we obtain that
$u,\varphi\in C^1\big([0,T]\times\R^2,\R^2\big)$. Then, a pointwise differentiation in $t$ of 
$v:=u\circ\varphi$ gives
\begin{equation*}
\dt{v}=\big(\dt{u}+u_z u+u_{\bar{z}} \bar{u}\big)\circ\varphi=R_{\varphi}\big(\partial_z^{-1}Q (u)\big) =  
\big(R_{\varphi}\circ\partial_z^{-1}\circ Q\circ R_{\varphi^{-1}}\big)(v).
\end{equation*}
By a pointwise integration we then obtain
\begin{equation}\label{eq:integral_relation}
v(t;z,\bar{z})=u_0(z,\bar{z})+\int_0^t \mathcal{E}_2(\varphi, v)\big|_{(s;z,\bar{z})}\,ds
\end{equation}
where 
\[
\mathcal{E}_2(\varphi,v):=\big(R_{\varphi}\circ\partial_z^{-1}\circ Q\circ R_{\varphi^{-1}}\big)(v).
\]
By Remark \ref{rem:smooth_vector_field} the map
\[
\mathcal{E}_2 : (\varphi,v)\mapsto \big(R_{\varphi}\circ\partial_z^{-1}\circ Q\circ R_{\varphi^{-1}}\big)(v),\quad
\mathcal{ZD}^{m,p}\times\mathcal{Z}^{m,p}_N\to \mathcal{Z}^{m,p}_N,
\]
is real analytic. Hence the curve $s\mapsto\mathcal{E}_2(\varphi(s),v(s))$,  $[0,T]\to\mathcal{Z}^{m,p}_{N+1}$,
is continuous and the integral in \eqref{eq:integral_relation} converges in $\mathcal{Z}^{m,p}_N$. 
This shows that $v\in C^1\big([0,T],\mathcal{Z}^{m,p}_N\big)$ satisfies $\dt{v}=\mathcal{E}_2(\varphi,v)$, where
$\varphi\in C^1\big([0,T],\mathcal{ZD}^{m,p}_N\big)$ satisfies $\dt\varphi=u\circ\varphi=v$, $\varphi|_{t=0}=\idmap$. 
This shows that the curve $(\varphi,v)\in C^1\big([0,T],\mathcal{ZD}^{m,p}_N\times\mathcal{Z}^{m,p}_N\big)$
satisfies \eqref{EQDynSys}.
\end{proof}

For any $\rho>0$ denote by $B_{\mathcal{Z}^{m,p}_N}(\rho)$ the ball of radius $\rho$ in 
$\mathcal{Z}^{m,p}_N$ centered at zero. Finally, we prove

\begin{theorem}\label{th:main}
Assume that $m>3+(2/p)$ where $1<p<\infty$ and $N\ge 0$. Then for any $\rho>0$ there exists $T>0$ such that for any 
$u_0\in B_{\mathcal{Z}^{m,p}_N}(\rho)$ the 2d Euler equation has a unique solution 
$u\in C\big([0,T], \mathcal{Z}^{m,p}_N\big)\cap C^1\big([0,T],\mathcal{Z}^{m-1,p}_N\big)$ such that 
$\p\equiv\p(t)\in\mathcal{Z}^{m+1,p}_{1,N}$ for any $t\in[0,T]$. This solution depends continuously on the initial data 
$u_0\in B_{\mathcal{Z}^{m,p}_N}(\rho)$ in the sense that the data-to-solution map
$u_0\mapsto u$, $B_{\mathcal{Z}^{m,p}_N}(\rho)\to C\big([0,T],\mathcal{Z}^{m,p}_N\big)\cap C^1\big([0,T],\mathcal{Z}^{m-1,p}_N\big)$,
is continuous. In addition, the coefficients $a_{kl} : [0,T]\to\C$, $0\le k+l\le N$, in the asymptotic expansion of the solution
\begin{equation*}
u(t)=\chi\sum_{0\le k+l\le N}\frac{a_{kl}(t)}{z^k\bar{z}^l}+f(t),\quad f(t)\in W^{m,p}_{\gamma_N},
\end{equation*}
are {\em holomorphic} functions of $t$ in an open neighborhood of $[0,\infty)$ in $\C$.
\end{theorem}

\begin{proof}[Proof of Theorem \ref{th:main}]
By Proposition \ref{PropEqEulerDynSys} and Proposition \ref{PropBijecDynSysDivFreeEulerEq} there is a bijection between solutions of 
the Euler equation and the real analytic dynamical system \eqref{EQDynSys}. 
By the existence and uniqueness theorem for solutions of an ODE in a Banach space (\cite{LangDiffGeo}) there exist $\rho_0>0$ and 
$T_0 > 0$ such that for any $u_0 \in B_{\mathcal{Z}^{m,p}_N}(\rho_0)$ there exists a unique solution 
$(\varphi, v)\in C^1\big([0,T_0], \mathcal{ZD}^{m,p}_N\times\mathcal{Z}^{m,p}_N\big)$, and therefore by 
Proposition \ref{PropBijecDynSysDivFreeEulerEq} 
$u=v\circ\varphi^{-1}\in C\big([0,T_0],\mathcal{Z}^{m,p}_N\big)\cap C^1\big([0,T_0],\mathcal{Z}^{m-1,p}_N\big)$. 
Since for any $c>0$ the curve $u_c(t) := c u(ct)$, $u_c : [0,T_0/c]\to\mathcal{Z}^{m,p}_N$, 
is a solution of \eqref{EQEulerwoDiv} (and \eqref{EQeulerCauchyop}) with initial data $u_c(0)=c u_0\in B_{\mathcal{Z}^{m,p}_N}(c\rho_0)$, 
we prove the first two statements of the theorem by taking $c=\rho/\rho_0$.

In order to prove the last statement of the theorem, we argue as follows: In view of Remark \ref{rem:smooth_vector_field}, the solution 
$t\mapsto\big(\varphi(t),\dt\varphi(t)\big)$, $[0,T]\to\mathcal{ZD}^{m,p}_N\times\mathcal{ZD}^{m,p}_N$, of the dynamical system
\eqref{EQDynSys} is a real analytic curve. Let $\hrho : \mathcal{ZD}^{m,p}_N\to\hA$ be the homomorphism of the
asymptotic group $\mathcal{ZD}^{m,p}_N$ that assigns to the elements of $\mathcal{ZD}^{m,p}_N$ their asymptotic part --
see Proposition \ref{prop:hA-group}. Denote by $\mathop{\rm Asymp}(w)$ the asymptotic part of an element $w\in\mathcal{Z}^{m,p}_N$.
Then we have from \eqref{eq:hrho_*-representation} that
\begin{align}\label{eq:asymptotics_of_u(t)}
\mathop{\rm Asymp}\big(u(t)\big)&=\mathop{\rm Asymp}\big(\dt\varphi(t)\circ\varphi(t)^{-1}\big)=
d_eR_{\hrho(\varphi(t)^{-1})}\big(\mathop{\rm Asymp}\big(\dt\varphi(t)\big)\big)\nonumber\\
&=d_eR_{\hrho(\varphi(t))^{-1}}\big(\mathop{\rm Asymp}\big(\dt\varphi(t)\big)\big)
\end{align}
where $d_eR_a : \R^{2M}\to\R^{2M}$ for $a\in\hA$ denotes the differential of the right translation 
$R_a : \hA\to\hA$ at the identity element $e\in\hA$ and where the inverse of $\hrho\big(\varphi(t)\big)\in\hA$ is taken in the Lie group $\hA$.
Now the real analyticity of the asymptotic coefficients of the solution $u(t)$ follows from Proposition \ref{prop:hA-group}
and formula \eqref{eq:asymptotics_of_u(t)}.
\end{proof}

In the remainder of this section we prove Proposition \ref{prop:integrals} and Proposition \ref{prop:evolution_coefficients}
stated in the Introduction.  We will prove these propositions for the local in time solutions of the 2d Euler equation 
given by Theorem \ref{th:main}. The general statements then follow from the global in time existence proved in 
Section \ref{sec:global_solutions} (see Theorem \ref{th:main_global}).

\begin{proof}[Proof of Proposition \ref{prop:integrals}]
The proposition follows by comparing the asymptotic terms in \eqref{EQeulerCauchyop}.
In fact, since $\p\in\mathcal{Z}^{m+1,p}_{1,N}$, we obtain that 
$\p_{\bar{z}}\in\overline{\big(\widetilde{\mathcal{Z}}^{m-1,p}_{N}\big)}$.
On the other side, since $u$ is divergence free,
$u_zu+u_{\bar{z}}\bar{u}=-\bar{u}_{\bar{z}}u+u_{\bar{z}}\bar{u}\in\overline{\big(\widetilde{\mathcal{Z}}^{m-1,p}_{N}\big)}$.
Then, by \eqref{EQeulerCauchyop}, 
\[
\dt{u}=-2\p_{\bar{z}}-(u_zu+u_{\bar{z}}\bar{u})\in\overline{\big(\widetilde{\mathcal{Z}}^{m-1,p}_{N}\big)}
\]
where $\widetilde{\mathcal{Z}}^{m-1,p}_{N}\subseteq\mathcal{Z}^{m-1,p}_{2,N}$ for $N\ge 2$ and 
$\widetilde{\mathcal{Z}}^{m-1,p}_{1}\equiv W^{m-1,p}_{\gamma_1}$.
This proves the first statement of the proposition. The second statement follows from \eqref{eq:integral of Q} by comparing 
the asymptotic coefficients in the equation \eqref{EQEulerwoDiv}.
\end{proof}

\begin{proof}[Proof of Proposition \ref{prop:evolution_coefficients}]
Item (i) follows directly from Proposition \ref{prop:integrals}. 
Item (iii) follows from Lemma \ref{lem:a_kl-derivatives} below (see the proof of Corollary \ref{coro:asymptotics_pop-up} below).
Let us prove (ii). By Theorem \ref{th:main} for any $u_0\in\mathcal{Z}^{m,p}_{n,N}$, $3<n\le N+1$, there exists
a unique solution $u\in C\big([0,T], \mathcal{Z}^{m,p}_N\big)\cap C^1\big([0,T],\mathcal{Z}^{m-1,p}_N\big)$
of the 2d Euler equation \eqref{EQeulerCauchyop} (and \eqref{EQEulerwoDiv}) such that $\p\in\mathcal{Z}^{m+1,p}_{1,N}$.
Since $u$ is divergence free, we have that $\divop u=u_z+\bar{u}_{\bar{z}}=0$ and by formula \eqref{eq:curl} in 
Appendix \ref{appendix:complex_form},
\begin{equation}\label{eq:complex_vorticity}
\omega:=i \curlop u\equiv\frac{1}{2}\big(u_z-\bar{u}_{\bar{z}}\big)=u_z.
\end{equation}
By applying the Cauchy operator $\partial_z$ to \eqref{EQEulerwoDiv} we obtain that
$\dt\omega=-\big(uu_z+\bar{u}u_{\bar{z}}\big)_z+\big((u_z)^2+u_{\bar{z}}\bar{u}_z\big)=
-\big(u\,\omega_z+\bar{u}\,\omega_{\bar{z}}\big)$ 
that is the 2d Euler equation in vorticity form,
\begin{equation}\label{eq:euler_in_vorticity_form}
\dt\omega+u\,\omega_z+\bar{u\,}\omega_{\bar{z}}=0.
\end{equation}
This and the arguments used to prove the analogous formula \eqref{eq:pointwise} then imply that 
$\partial_t\big(\omega(t)\circ\varphi(t)\big)=0$ for any $t\in[0,T]$.
In particular, for any $t\in[0,T]$ we have that 
\begin{equation}\label{eq:preservation_of_vorticity_in_Z}
\omega(t)=\omega_0\circ\psi(t)
\end{equation}
where $\psi(t):=\varphi(t)^{-1}$, $\varphi\in C^1\big([0,T],\mathcal{ZD}^{m,p}_N\big)$, and 
$\omega_0=\partial_zu_0\in\widetilde{\mathcal{Z}}^{m-1,p}_{N+1}\cap\mathcal{Z}^{m-1,p}_{n+1,N+1}$,
\begin{equation}\label{eq:specific_form}
\omega_0=\frac{\chi}{z^2}\sum_{n-1\le k+l\le N-1}\frac{b_{kl}^0}{z^k\bar{z}^l}+f_0,\quad f_0\in W^{m-1,p}_{\gamma_{N+1}}.
\end{equation}
It follows from \eqref{eq:improved_composition}, \eqref{eq:specific_form}, and the fact that
$\mathcal{ZD}^{m,p}_N$ is a group that
\[
\partial_zu(t)=\omega_0\circ\psi(t)=
\frac{\chi}{z^2}\sum_{n-1\le k+l\le N-1}\frac{b_{kl}(t)}{z^k\bar{z}^l}+f(t),\quad f(t)\in W^{m-1,p}_{\gamma_{N+1}},
\]
for any $t\in[0,T]$. This, together with Proposition \ref{prop:asymptotic_operator} and Lemma \ref{lem:vanishing_momenta}, then implies that
\begin{equation}\label{eq:new_asymptotics}
u(t)={\chi}\sum_{3\le k\le n-1}\frac{a_{0k}(t)}{\bar{z}^k}+g(t),\quad g(t)\in\mathcal{Z}^{m,p}_{n,N}
\end{equation}
where
\begin{equation*}
a_{0k}(t):=\frac{1}{\pi}\big(\omega_0\circ\varphi(t)^{-1},\bar{z}^{k-1}\big)
=\frac{1}{\pi}\int_{\R^2}\omega_0\,\big(\overline{\varphi}(t)\big)^{k-1}\det\big(d\varphi(t)\big)\,dx dy
=\frac{1}{\pi}\int_{\R^2}\omega_0\,\big(\overline{\varphi}(t)\big)^{k-1}\,dx dy,
\end{equation*}
since $\varphi(t)$ is a volume preserving diffeomorphism.
The last statement follows easily from the fact that $\dt\varphi(t)=u(t)\circ\varphi(t)$ since $u(t)$ is divergence free.
This completes the proof of the proposition.
\end{proof}

Finally, we prove the following generalization of Example 2 in \cite[Appendix B]{McOTo3}.

\begin{lemma}\label{lem:a_kl-derivatives}
Assume that $u_0\in \mathcal{Z}^{m,p}_{n,N}$ where $3<n\le N+1$ and $m>3+(2/p)$.
Then, the solution $u\in C\big([0,\infty),\mathcal{Z}^{m,p}_N\big)\cap C^1\big([0,\infty),\mathcal{Z}^{m-1,p}_N\big)$
of the 2d Euler equation given by Theorem \ref{th:main_introduction} is of the form \eqref{eq:new_asymptotics}
where
\begin{equation}\label{eq:a_kl-derivatives}
{\dt a}_{k0}(0)=\frac{(k-1)(k-2)}{2\pi}\int_{\R^2}{\bar u}_0^2\,\bar{z}^{k-3}\,dx dy,\quad 3\le k\le n.
\end{equation}
In addition, for any $\epsilon>0$ and $k\ge 3$ there exists a divergence free $u_0\in C^\infty_c$ with support in 
the annulus $\epsilon\le|z|\le 2\epsilon$ in $\C$ such that $\dt a_{k0}(0)\ne 0$.
\end{lemma}

\begin{proof}[Proof of Lemma \ref{lem:a_kl-derivatives}]
By Proposition \ref{prop:evolution_coefficients} the solution of the 2d Euler equation has the form \eqref{eq:new_asymptotics}
where $a_{k0}(t)=\frac{1}{\pi}\int_{\R^2}\omega_0\,\big(\overline{\varphi}(t)\big)^{k-1}\,dx dy$.
Since $u\in C\big([0,\infty),\mathcal{Z}^{m,p}_N\big)\cap C^1\big([0,\infty),\mathcal{Z}^{m-1,p}_N\big)$, we can differentiate
under the integral sign to obtain that for $3\le k\le n$,
\begin{align*}
{\dt a}_{k0}(0)&=\frac{k-1}{\pi}\int_{\R^2}\omega_0\bar{u}_0\,\bar{z}^{k-2}\,dx dy
=-\frac{k-1}{\pi}\int_{\R^2}\big(\partial_{\bar{z}}\bar{u}_0\big)\,\bar{u}_0\,\bar{z}^{k-2}\,dx dy\\
&=\frac{(k-1)(k-2)}{2\pi}\int_{\R^2}{\bar u}_0^2\,\bar{z}^{k-3}\,dx dy,
\end{align*}
where we used that $\dt\varphi=u(t)\circ\varphi(t)$, $\partial_z u_0=-\partial_{\bar{z}}\bar{u}_0$,
and the Stokes' theorem. Let us now prove the last statement of the lemma.
To this end, for any $k\ge 3$ we will construct a real valued Hamiltonian $H=H(x,y)$ with compact support on $\R^2$ so that
its Hamiltonian vector field $X_H\equiv u_0\,\partial_z+\bar{u}_0\,\partial_{\bar{z}}$ corresponding to the canonical form 
$dx\wedge dy$ on $\R^2$ satisfies 
\begin{equation}\label{eq:momenta}
\int_{\R^2}u_0^2\,z^{k-3}\,dx dy\ne 0.
\end{equation}
In fact, choose $a,b\in C^\infty_c(\R,\R)$ that are not identically equal to zero with support in $\epsilon\le\rho\le 2\epsilon$.
For any $\ell\ge 2$ consider the Hamiltonian
\begin{equation}\label{eq:Hamiltonian}
H\equiv H_\ell(z,\bar{z}):=\big(z^\ell+\bar{z}^\ell\big)a(z\bar{z})+b(z\bar{z}).
\end{equation}
Then, $X_H=2i (\partial_{\bar{z}}H)\,\partial_z-2i(\partial_z H)\,\partial_{\bar{z}}$, and hence, $u_0=2i(\partial_{\bar{z}}H)$.
Note that this vector field is automatically divergence free.
We have,
\[
(\partial_{\bar{z}}H_\ell)=a'z^{\ell+1}+\big(\ell a+r^2a'\big)\bar{z}^{\ell-1}+b' z
\]
where $a\equiv a(z\bar{z})$, $a'\equiv a'(z\bar{z})$, similarly for $b$ and $b'$, and $(\cdot)'$ denotes the first derivative
with respect to the variable $\rho$. This implies that for $\ell\ge 2$,
\[
(\partial_{\bar{z}}H_\ell)^2=\big(\ell a+r^2a'\big)^2\bar{z}^{2(\ell-1)}
+2b'\big(\ell a+r^2a'\big)r^2\bar{z}^{\ell-2}
+...\,,
\]
where $...$ stands for terms of the form $c(r) z^n$ where $n\ge 1$ and $c$ is an expression
depending only on the radius $r\equiv|z|$. Then, by passing to polar coordinates in the double integral, 
we obtain that for $\ell\ge 2$,
\begin{align}
\int_{\R^2}(\partial_{\bar{z}}H_\ell)^2\, z^{\ell-2}\,dx dy
&=2\pi\int_0^\infty b'(r^2)\big(\ell a(r^2)+r^2 a'(r^2)\big)r^{2\ell-1}\,dr
=\pi\int_0^\infty b'\big(\ell a+\rho a'\big)\rho^{\ell-1}\,d\rho\nonumber\\
&=\pi\int_0^\infty b'\big(a\rho^\ell\big)'\,d\rho,\label{eq:ell-momentum}
\end{align}
where $\rho=r^2$. By taking $b:=a\rho^\ell$ we then obtain from \eqref{eq:ell-momentum}
that $\int_{\R^2}(\partial_{\bar{z}}H_\ell)^2\, z^{\ell-2}\,dx dy>0$ for $\ell\ge 2$.
By combining this with \eqref{eq:momenta} we conclude the proof of the lemma.
\end{proof}

Let us now prove Corollary \ref{coro:asymptotics_pop-up}.

\begin{proof}[Proof of Corollary \ref{coro:asymptotics_pop-up}]
For any $k\ge 3$ and a divergence free vector field $w\in\accentset{\circ}{S}$ denote 
\[
B_k(w):=\int_{\R^2}w^2 z^{k-3}\,dx dy.
\]
Note that $B_k : \accentset{\circ}{S}\times\accentset{\circ}{S}\to\C$ is a bounded quadratic functional on $\C$.
This together with Lemma \ref{lem:a_kl-derivatives} implies that for any $k\ge 3$ the set 
\[
\mathcal{Z}_k:=\big\{w\in\accentset{\circ}{S}\,\big|\, B_k(w)=0\big\}
\]
is a closed nowhere dense set in $\accentset{\circ}{S}$. Since $\accentset{\circ}{S}$ is a closed set inside the 
complete metric space $S$, we conclude from the Baire category theorem that the set
\[
\mathcal{N}:=\accentset{\circ}{S}\setminus\big(\bigcup_{k\ge 3}\mathcal{Z}_k\big)
\]
is dense in $\accentset{\circ}{S}$. It now follows from \eqref{eq:a_kl-derivatives} and the definition of the set
$\mathcal{N}$ that for any $u_0\in\mathcal{N}$ and for any $k\ge 3$ we have that
\[
{\dt a}_{0k}(0)\ne 0.
\]
Since the functions $a_{0k} : \R\to\C$ extend to holomorphic functions, we then conclude that 
$a_{0k}(t)\ne 0$ for almost any $t>0$. This completes the proof of Corollary \ref{coro:asymptotics_pop-up}.
\end{proof}

\section{Global existence in $\mathcal{Z}^{m,p}_N$}\label{sec:global_solutions}
In this section we prove that the solutions of the 2d Euler equation given by Theorem \ref{th:main} exist for all time $t\ge 0$. 
To this end, we first recall that for $m\ge 0$ the little H\"older space $c^{m,\gamma}_b$ is the closure of 
$C^{\infty}_b$ in the H\"older spaces $C^{m,\gamma}_b\equiv C^{m,\gamma}_b(\R^2,\R^2)$. The space $C^{m,\gamma}_b(\R^2,\R)$ consists 
of functions in $f\in C^m_b(\R^2,\R)$ with finite H\"older semi-norms $[D^\alpha f]_\gamma$ for any multi-index $\alpha\in\Z_{\ge 0}^2$ 
such that $|\alpha|=m$. The H\"older semi-norm is defined as
\[
[f]_\gamma:=\sup_{{\rm x}\ne{\rm y}}\frac{\big|f({\rm x})-f({\rm y})\big|}{|{\rm x}-{\rm y}|^\gamma}
\]
and the norm in $C^{m,\gamma}_b(\R^2,\R)$ is $|f|_{m,\alpha}:=|f|_m+\max_{|\alpha|=m}[D^\alpha f]_\gamma$ where
$|f|_m$ is the norm in the space $C^m_b(\R^2,\R)$ of $C^m$-smooth functions on $\R^2$ with bounded derivatives of order $\le m$.
(For $\gamma=0$ we have that $C^m_b\equiv C^{m,0}_b=c^{m,0}_b$.)
For $m\ge 1$ consider the group of diffeomorphism of $\R^2$ of class $c^{m,\gamma}$,
\begin{align*}
\diff^{m,\gamma}(\R^2):=\big\{\varphi=\idmap+u\,
\big|\, u\in c^{m,\gamma}_b\,\text{\rm and }\,\exists\,\varepsilon>0\,\,\text{\rm such that}\,\det(I+du)>\varepsilon\big\}.
\end{align*}
By Theorem 3.1 in \cite{ST1} (cf. \cite{MY}) the set $\diff^{m,\gamma}(\R^2)$ is a topological group with additional regularity properties 
of the composition and the inversion of diffeomorphisms similar to the ones in Theorem \ref{th:ZD}.

\begin{remark}
Note that if in the definition of $\diff^{m,\gamma}(\R^2)$ with $0<\gamma<1$ above we replace 
the  little H\"older space $c^{m,\gamma}_b$ by the H\"older space $C^{m,\gamma}_b$ then the composition
of diffeomorphisms in the resulting group will {\em not} be continuous (see e.g. Remark 3.1 in \cite{ST1}).
This is the reason for the ill posedness of the Euler equation in the H\"older space $C^{m,\gamma}_b$ (see \cite{MY}).
\end{remark}

\medskip

The following lemma follows from Remark \ref{rem:improved_composition} (see formula \eqref{eq:improved_composition}), 
Proposition \ref{prop:properties_Z-spaces} (iii), and Lemma 6.5 in \cite{McOTo2}.

\begin{lemma}\label{lem:longer_expansion}
Assume that $m>1+(2/p)$ with $1<p<\infty$. Then, for any $u\in\mathcal{Z}^{m,p}_{N+2}$ and for any $\varphi\in\mathcal{ZD}^{m,p}_N$ we 
have that $u\circ\varphi\in\mathcal{Z}^{m,p}_{N+2}$. Moreover, the composition map $(u,\varphi)\mapsto u\circ\varphi$, 
$\mathcal{Z}^{m,p}_{N+2}\times\mathcal{ZD}^{m,p}_N\to\mathcal{Z}^{m,p}_{N+2}$, is continuous.
If, in addition, $u\in\widetilde{\mathcal{Z}}^{m,p}_{N+2}$ then $u\circ\varphi\in\widetilde{\mathcal{Z}}^{m,p}_{N+2}$.
\end{lemma}

\begin{remark}\label{rem:longer_expansion}
The lemma also holds with $\mathcal{Z}^{m,p}_{N+2}$ replaced by $\mathcal{Z}^{m,p}_{N+1}$ and
$\widetilde{\mathcal{Z}}^{m,p}_{N+2}$ replaced by $\widetilde{\mathcal{Z}}^{m,p}_{N+1}$.
\end{remark}

Recall that for $m\ge 0$, $1<p<\infty$, and any weight $\delta\in\R$,
\[
H^{m,p}_\delta(\R^2,\R):=\big\{f\in L^p_{loc}(\R^2,\R)\,\big|\,
\x^\delta D^\alpha f\in L^p\,\text{\rm for }\,|\alpha|\le m \big\},\quad
D^\alpha\equiv\partial_{x_1}^{\alpha_1}\partial_{x_2}^{\alpha_2}.
\]
Note that for $m>2/p$ and non-negative weights $\delta\ge 0$ the space $H^{m,p}_\delta(\R^2,\R)$ is a Banach algebra
with respect to pointwise multiplication of functions (see e.g. Proposition 2.1 in \cite{McOTo2}). 
We have the following lemma.

\begin{lemma}\label{lem:sobolev-holder}
For $0\le\gamma<1$ and any weight $\delta\in\R$ and $1<p<\infty$ we have that the composition map 
\[
(u,\varphi)\mapsto u\circ\varphi,\quad L^p_\delta\times\diff^{1,\gamma}(\R^2)\to L^p_{\delta},
\] 
is continuous. More generally, for any integer $m\ge 1$ and $0\le k\le m$ we have that
the composition map $(u,\varphi)\mapsto u\circ\varphi$, 
$H^{k,p}_\delta\times\diff^{m,\gamma}(\R^2)\to H^{k,p}_{\delta}$, 
is continuous.
\end{lemma}

\begin{remark}\label{rem:sobolev-holder}
Lemma \ref{lem:sobolev-holder} also holds when $p=\infty$.
\end{remark}

\begin{proof}[Proof of Lemma \ref{lem:sobolev-holder}]
First, note that it is enough to prove the lemma only for $\gamma=0$. The general statement then follows, since
$\diff^{m,\gamma}(\R^2)$ is continuously embedded in $\diff^m(\R^2)\equiv\diff^{m,0}(\R^2)$.
Since the second statement of the lemma follows from the first one and the product rule,
we concentrate our attention only on proving its first statement.
Take $u_1, u_2\in L^p_\delta\equiv L^p_\delta(\R^2,\R)$ and $\varphi\in\diff^1(\R^2)$. 
By changing the variables in the corresponding integral we obtain that
\begin{equation}\label{eq:locally_uniformly1}
\norm{u_2\circ\varphi-u_1\circ\varphi}_{L^p_{\delta}}=
\norm{(u_2-u_1)\circ \varphi}_{L^p_{\delta}}
\le C\,\norm{u_2-u_1}_{L^p_{\delta}}
\end{equation}
where $C\equiv C(\varphi)>0$ depends on $|\varphi^{-1}|_1$. Since $\varphi\in\diff^1(\R^2)$ is a topological group
(see e.g Theorem 3.1 in \cite{ST1}), the constant $C>0$ in \eqref{eq:locally_uniformly1} can be chosen {\em locally uniformly} in 
$\varphi\in\diff^1(\R^2)$.
Now, choose $u\in C^\infty_c$ and $\varphi_1,\varphi_2\in\diff^1(\R^2)$. Then, we obtain from the mean-value theorem that
for any ${\rm x}\in\R^2$,
\begin{equation*}
(u\circ\varphi_2)({\rm x})-(u\circ\varphi_1)({\rm x})=
\Big(\int_0^1(du)|_{\zeta_t({\rm x})}\,dt\Big)
\cdot\big(\varphi_2({\rm x})-\varphi_1({\rm x})\big)
\end{equation*}
where $\zeta_t(x):=\varphi_1({\rm x})+t(\varphi_2({\rm x})-\varphi_1({\rm x}))$. By applying first the Cauchy-Schwartz inequality 
and then Jensen's inequality, we obtain
\begin{equation*}
\big|(u\circ\varphi_2)({\rm x})-(u\circ\varphi_1)({\rm x})\big|^p
\le C\,\big|\varphi_2-\varphi_1\big|_1^p\int_0^1\Big|(du)|_{\zeta_t({\rm x})}\Big|_{\ell^p}^p\,dt
\end{equation*}
where $|\cdot|_{\ell^p}$ denotes the $\ell^p$-norm of the components of a matrix and the constant $C>0$ depends only on 
the choice of $1<p<\infty$. By multiplying this inequality by $\x^{p\delta}$ and then by integrating it over ${\rm x}\in\R^2$, we obtain that
\begin{equation}\label{eq:locally_uniformly2'}
\big\|u\circ\varphi_2-u\circ\varphi_1\big\|_{L^p_{\delta}}
\le C\,\Big(\int_0^1\big\|(du)\circ\zeta_t\big\|_{L^p_{\delta}}^p\,dt\Big)^{1/p}\big|\varphi_2-\varphi_1\big|_1
\end{equation}
with a constant $C>0$ that depends only on the choice of $1<p<\infty$.
It follows from \eqref{eq:locally_uniformly1} that for $\varphi_2$ chosen in an open neighborhood $U(\varphi_1)$ of $\varphi_1$ 
in $\diff^1(\R^2)$ we have that 
\[
\big\|(du)\circ\zeta_t\big\|_{L^p_{\delta}}\le C\,\|du\|_{L^p_{\delta}}.
\]
By combining this with \eqref{eq:locally_uniformly1} we obtain that for any $\varphi_1\in\diff^1(\R^2)$ there exists 
an open neighborhood $U(\varphi_1)$ of $\varphi_1$ in $\diff^1(\R^2)$ and a constant $C\equiv C(\varphi_1)>0$ such that
for any $\varphi_2\in U(\varphi_1)$ and for any $u\in C^\infty_c$ we have that
\begin{equation}\label{eq:locally_uniformly2}
\big\|u\circ\varphi_2-u\circ\varphi_1\big\|_{L^p_{\delta}}\le C\,\|u\|_{H^{1,p}_{\delta}}\big|\varphi_2-\varphi_1\big|_1.
\end{equation}
Finally, it follows from \eqref{eq:locally_uniformly1} and \eqref{eq:locally_uniformly2} that there exists an open neighborhood 
$U(\varphi_1)$ of $\varphi_1\in\diff^1(\R^2)$ and constants $C_1>0$ and $C_2>0$ such that for any 
$\varphi_2\in U(\varphi_1)$ and for any $u_1,u_2\in L^p_\delta$ and $\widetilde{u}\in C^\infty_c$ we have
\begin{eqnarray*}
\big\|u_2\circ\varphi_2-u_1\circ\varphi_1\big\|_{L^p_\delta}&\le&
\big\|u_2\circ\varphi_2-\widetilde{u}\circ\varphi_2\big\|_{L^p_\delta}
+\big\|\widetilde{u}\circ\varphi_2-\widetilde{u}\circ\varphi_1\big\|_{L^p_\delta}
+\big\|\widetilde{u}\circ\varphi_1-u_1\circ\varphi_1\big\|_{L^p_\delta}\\
&\le& C_1\|u_2-\widetilde{u}\|_{L^p_\delta}
+C_2\,\|\widetilde{u}\|_{H^{1,p}_\delta}\big|\varphi_2-\varphi_1\big|_1
+C_1\|\widetilde{u}-u_1\|_{L^p_\delta}.
\end{eqnarray*}
Now, take $\epsilon>0$ and then choose $u_2\in L^p_\delta$ and $\widetilde{u}\in C^\infty_c$, $\widetilde{u}\not\equiv 0$, 
inside the open ball in $L^p_\delta$ of radius $\rho>0$ centered at $u_1\in L^p_\delta$. The inequality above then implies that
\[
\big\|u_2\circ\varphi_2-u_1\circ\varphi_1\big\|_{L^p_\delta}
\le 3C_1\rho+C_2\|\widetilde{u}\|_{H^{1,p}_\delta}\big|\varphi_2-\varphi_1\big|_1.
\]
Finally, by choosing $0<\rho<\epsilon/(6C_1)$ and $\varphi_2\in U(\varphi_1)$ such that
$|\varphi_2-\varphi_1\big|_1<\epsilon/\big(2C_2\,\|\widetilde{u}\|_{H^{1,p}_\delta}\big)$
we conclude that $\big\|u_2\circ\varphi_2-u_1\circ\varphi_1\big\|_{L^p_\delta}<\epsilon$.
This completes the proof of the lemma.
\end{proof}

As a consequence of Lemma \ref{lem:longer_expansion} we obtain the following proposition.

\begin{proposition}\label{prop:longer_expansion}
Assume that $m>3+(2/p)$ with $1<p<\infty$. Assume in addition that for some $T>0$,
$u\in C\big([0,T],\mathcal{Z}^{m,p}_N\big)\cap C^1\big([0,T],\mathcal{Z}^{m-1,p}_N\big)$ 
be a solution of the 2d Euler equation \eqref{EQeulerCauchyop} with $u_0\in\mathcal{Z}^{m,p}_{N+1}$.
Then, $u\in C\big([0,T],\mathcal{Z}^{m,p}_{N+1}\big)\cap C^1\big([0,T],\mathcal{Z}^{m-1,p}_{N+1}\big)$.
\end{proposition}

\begin{proof}[Proof of Proposition \ref{prop:longer_expansion}]
Assume that $u\in C\big([0,T],\mathcal{Z}^{m,p}_N\big)\cap C^1\big([0,T],\mathcal{Z}^{m-1,p}_N\big)$ be a solution
of the 2d Euler equation \eqref{EQeulerCauchyop} such that $u_0\in\mathcal{Z}^{m,p}_{N+1}$.
Then, for any $t\in[0,T]$ we have that 
$\omega(t):=i \curlop u\equiv\frac{1}{2}(u_z-\bar{u}_{\bar{z}})=\partial_zu(t)$ and by the preservation of vorticity 
(see \eqref{eq:euler_in_vorticity_form}),
\begin{equation}\label{eq:preservation_of_vorticity}
\omega(t)=\omega_0\circ\psi(t),\quad\omega_0\in\widetilde{\mathcal{Z}}^{m-1,p}_{N+2},
\end{equation}
where $\psi(t):=\varphi(t)^{-1}$ and $\varphi\in C^1\big([0,T],\mathcal{ZD}^{m,p}_N\big)\big)$ is the Lagrangian coordinate
of the solution $u$. By Theorem \ref{th:ZD}, 
$\psi\in C\big([0,T],\mathcal{ZD}^{m,p}_N\big)\cap C^1\big([0,T],\mathcal{ZD}^{m-1,p}_N\big)$. 
This, together with \eqref{eq:preservation_of_vorticity} and Lemma \ref{lem:longer_expansion}, 
implies that for any $t\in[0,T]$ we have that
\[
\partial_zu(t)=\omega(t)=\omega_0\circ\psi(t)\in\widetilde{\mathcal{Z}}^{m-1,p}_{N+2}
\quad\text{\rm and}\quad 
\partial_zu\in C\big([0,T],\widetilde{\mathcal{Z}}^{m-1,p}_{N+2}\big).
\]
Hence, we conclude from Theorem \ref{th:Z->Z-tilde} that $u\in C\big([0,T],\mathcal{Z}^{m,p}_{N+1}\big)$.
This implies that the local solution of the 2d Euler equation in $\mathcal{Z}^{m,p}_{N+1}$ 
with initial data $u_0\in\mathcal{Z}^{m,p}_{N+1}$ given by Theorem \ref{th:main} 
has a finite $\mathcal{Z}^{m,p}_{N+1}$-norm on the intersection of its maximal interval of existence with $[0,T]$.
Thus, the maximal interval of existence of this solution contains $[0,T]$ and
$u\in C\big([0,T],\mathcal{Z}^{m,p}_{N+1}\big)\cap C^1\big([0,T],\mathcal{Z}^{m-1,p}_{N+1}\big)$.
This completes the proof of the proposition.
\end{proof}

\begin{remark}\label{rem:Z->holder}
By Proposition \ref{prop:main_technical} (iii) in Appendix \ref{appendix:technical} for $m>2/p$, $\delta+(2/p)>0$, 
and $0<\gamma<1$ chosen sufficiently small, we have the continuous embedding $W^{m,p}_\delta\subseteq C^{k,\gamma}_b$ 
for any integer $0\le k<m-2/p$. Since the elements of $W^{m,p}_\delta$ are approximated by functions in $C^\infty_c$,
we then conclude that $W^{m,p}_\delta$ is continuously embedded in the little H\"older space $c^{m,\gamma}_b$. 
The same arguments applied to the asymptotic coefficients and the definition of the norm
in the asymptotic space $\mathcal{Z}^{m,p}_N$ show that $\mathcal{Z}^{m,p}_N$ is continuously embedded in $c^{k,\gamma}_b$. 
Since $1<p<\infty$, we conclude that $0<2/p<2$, and hence for $m\ge 3$ we have the continuous embedding
$\mathcal{Z}^{m,p}_N\subseteq c^{m-2,\gamma}_b$. In what follows we will assume that $0<\gamma<1$ is chosen so that 
these embeddings hold.
\end{remark}

For $m\ge 1$ and $0<\gamma<1$ we say that $\varphi\in C^1\big([0,T],\diff^{m,\gamma}(\R^2)\big)$ with 
$T>0$ is a {\em solution of the 2d Euler equation in Lagrangian coordinates} on $\diff^{m,\gamma}(\R^2)$ if 
$\varphi\in C^2\big([0,T],\diff^{m,\gamma}(\R^2)\big)$ and it satisfies the following second order in time evolution equation 
(cf. \eqref{EQDynSys})
\begin{equation}\label{eq:ode_on_the_group}
\left\{
\begin{array}{l}
\dtt\varphi=\big(R_\varphi\circ\partial_z^{-1}\circ Q\circ R_{\varphi^{-1}}\big)(\dt\varphi)\\
\varphi|_{t=0}=\idmap,\,\dt\varphi|_{t=0}=u_0
\end{array}
\right. 
\end{equation}
where the Cauchy operator is considered as an operator on tempered  distributions $\partial_z : S'\to S'$ (see Remark \ref{rem:cauchy_kernel})
with an inverse $\partial_z^{-1}$ that is an extension of the operator $\partial_z^{-1} : L^p_{\gamma_0+1}\to W^{1,p}_{\gamma_0}$ 
given by Proposition \ref{th:isomorphism W->L}. 
If $\varphi\in C\big([0,T],\diff^{m,\gamma}(\R^2)\big)$ is a solution of the 2d Euler equation in Lagrangian coordinates then 
one sees from Theorem 3.1 in \cite{ST1} that 
\begin{equation}\label{eq:u-translated_holder}
u:=\dt\varphi\circ\varphi^{-1}\in C\big([0,T],c^{m,\gamma}_b\big)\cap C^1\big([0,T],c^{m-1,\gamma}_b\big). 
\end{equation}
By the arguments used in the proof of Proposition \ref{PropBijecDynSysDivFreeEulerEq} we then conclude that $u$ satisfies
equation \eqref{EQEulerwoDiv}. Then, the derivation of \eqref{eq:euler_in_vorticity_form} implies that $u$ and its vorticity 
$\omega\equiv\partial_z u$ satisfy the 2d Euler equation in vorticity form $\dt\omega+u\omega_z+\bar{u}\omega_{\bar{z}}=0$. 
As a consequence, we then conclude (see the derivation of \eqref{eq:preservation_of_vorticity_in_Z}) that
\[
\omega(t)=\omega_0\circ\varphi(t)^{-1},\quad t\in[0,T].
\]
We have the following proposition.

\begin{proposition}\label{prop:more_decay}
Assume that for $u_0\in\mathcal{Z}^{5,p}_0\subseteq c^{3,\gamma}_b$  with $1<p<\infty$ there exists a unique solution of 
the 2d Euler equation in Lagrangian coordinates $\varphi\in C^1\big([0,T],\diff^{3,\gamma}(\R^2)\big)$ for some $T>0$. Then, 
the solution of the 2d Euler equation given by Theorem \ref{th:main} extends to the whole interval $[0,T]$ and 
$u\in C\big([0,T],\mathcal{Z}^{5,p}_0\big)\cap C^1\big([0,T],\mathcal{Z}^{4,p}_0\big)$.
\end{proposition}

\begin{proof}[Proof of Proposition \ref{prop:more_decay}]
Take $u_0\in\mathcal{Z}^{5,p}_0$ so that $u_0=c+f_0$ where $c$ is a real constant and $f_0\in W^{5,p}_{\gamma_0}$.
Assume that there exists $T>0$ so that $\tvarphi\in C^1\big([0,T],\diff^{3,\gamma}(\R^2)\big)$
be the unique solution of the 2d Euler equation in Lagrangian coordinates (cf. \eqref{eq:ode_on_the_group}) and
assume that the statement of the proposition does {\em not} hold. Hence, there exists $0<\tau<T$ such that
$[0,\tau)$ is the maximal interval of existence of the solution
$u\in C\big([0,\tau),\mathcal{Z}^{5,p}_0\big)\cap C^1\big([0,\tau),\mathcal{Z}^{4,p}_0\big)$
given by Theorem \ref{th:main}. Then, it follows easily from Theorem \ref{th:main} that
\begin{equation}\label{eq:blow-up_in _norm}
\|u(t)\|_{\mathcal{Z}^{5,p}_0}\to\infty\quad\text{\rm as}\quad t\to\tau-0.
\end{equation}
Now, consider the curve 
\begin{equation}\label{eq:tomega}
\widetilde\omega(t):=\omega_0\circ\tpsi(t),\quad t\in[0,T],
\end{equation}
where $\tpsi(t):=\tvarphi(t)^{-1}$ and $\omega_0:=\partial_z u_0\in W^{4,p}_{\gamma_0+1}\subseteq H^{4,p}_{\gamma_0+1}$.
It follows from Lemma \ref{lem:sobolev-holder}, $\omega_0\in H^{3,p}_{\gamma_0+1}$, and the fact that
$\diff^{3,\gamma}(\R^2)$ is a topological group, that
\begin{equation}\label{eq:tomega_preliminary}
\tomega\in C\big([0,T],H^{3,p}_{\gamma_0+1}\big).
\end{equation}
The uniqueness of solutions and the conservation of the vorticity $\omega\equiv\partial_zu$ on $[0,\tau)$ 
(see \eqref{eq:preservation_of_vorticity_in_Z} in the proof of Proposition \ref{prop:evolution_coefficients})
imply that
\[
\omega(t)=\tomega(t),\quad t\in[0,\tau).
\]
This together with the fact that the leading asymptotic coefficient of the solution $u$ is independent of $t\in[0,\tau)$ 
(see Proposition \ref{prop:integrals}) then implies that
\begin{equation}\label{eq:on_[0,tau)}
u(t)=c+\partial_z^{-1}\tomega(t),\quad t\in[0,\tau),
\end{equation}
where $c$ is the leading asymptotic coefficient of the initial data $u_0\in\mathcal{Z}^{5,p}_0$.
For $t\in[0,T]$ consider the curve
\begin{equation}\label{eq:u-tilde}
\tu(t):=c+\tr(t)\quad\text{where}\quad\tr(t):=\partial_z^{-1}\tomega(t)
\end{equation}
and $\partial_z^{-1} : L^p_{\gamma_0+1}\to W^{1,p}_{\gamma_0}$ is given by Theorem \ref{th:cauchy_m>=0} $(i)$.
Note that by \eqref{eq:tomega_preliminary} we have that (cf. Lemma \ref{lem:commutator})
\[
\partial^\alpha\tomega\in C\big( [0,T], L^p_{\gamma_0+1}\big),\quad|\alpha|\le 3.
\]
Since $0<\gamma_0+(2/p)<1$, we obtain from Theorem \ref{th:cauchy_m>=0} $(i)$ and Remark \ref{rem:Z->Z-tilde} that
for any $|\alpha|\le 3$ we have that $\partial^{\alpha}\tr\in C\big([0,T],W^{1,p}_{\gamma_0}\big)$, $|\alpha|\le 3$.
This, together with \eqref{eq:u-tilde}, then implies that
\begin{equation}\label{eq:tu_step_one}
\tu-c\in C\big([0,T],W^{1,p}_{\gamma_0}\big)\quad\text{\rm and}\quad d\tu\in C\big([0,T],H^{3,p}_{\gamma_0+1}\big).
\end{equation}
Note that, by \eqref{eq:u-translated_holder}, $\tvarphi\in C^1\big([0,T],\diff^{3,\gamma}(\R^2)\big)$ satisfies the equation
\begin{equation}\label{eq:ode}
\dt\tvarphi=\tu\circ\tvarphi,\quad\tvarphi|_{t=0}=\idmap.
\end{equation}
This implies that $Y:=d\tvarphi=I+df$, $f\in c^{3,\gamma}_b$, satisfies the equation $\dt Y=A(t)Y$, $Y|_{t=0}=I$, where 
$A:=d\tu\circ\tvarphi\in C\big([0,T],H^{3,p}_{\gamma_0+1}\big)$ by Lemma \ref{lem:sobolev-holder}.
This equation can be re-written as 
\begin{equation}\label{eq:jacobian_evolution}
(df)^\bolddot=A(t)(df)+A(t),\quad df|_{t=0}=0,
\end{equation}
and considered, by the Banach algebra property in $H^{3,p}_{\gamma_0+1}$, as an ODE on $df$ in the space
of $2\times 2$ matrices with elements in $H^{3,p}_{\gamma_0+1}$. Hence, by the existence and uniqueness of solutions of ODEs
in Banach spaces, we obtain that
\begin{equation}\label{eq:tphi-smoother}
d\tvarphi-I\in C\big([0,T],H^{3,p}_{\gamma_0+1}\big).
\end{equation}
Since
\begin{equation}\label{eq:tpsi-formula}
d\tpsi=\big[\!\mathop{\rm Adj}(d\tvarphi)/\det(d\tvarphi)\big]\circ\tpsi,
\end{equation}
where $\mathop{\rm Adj}(d\tvarphi)$ is the transpose of the cofactor matrix of $d\tvarphi$,
we conclude from the Banach algebra property in $H^{3,p}_{\gamma_0+1}$, the arguments used to prove Lemma \ref{lem:division},
and Lemma \ref{lem:sobolev-holder}, that
\begin{equation}\label{eq:tpsi-smoother}
d\tpsi-I\in C\big([0,T],H^{3,p}_{\gamma_0+1}\big).
\end{equation}

\begin{remark}\label{rem:H-group}
In particular, \eqref{eq:tpsi-smoother} implies that $\partial_z\big(\tpsi-z\big)\in H^{3,p}_{\gamma_0+1}\subseteq L^p_{\gamma_0+1}$.
Hence, by Theorem \ref{th:cauchy_m>=0} $(i)$, we obtain that $\tpsi-z-c_0\in W^{1,p}_{\gamma_0}$ for some constant $c_0\in\C$, 
and therefore $\tpsi-z-c_0\in H^{4,p}_{\gamma_0}$. This and similar arguments for $\tvarphi$ involving \eqref{eq:tphi-smoother} imply that
\begin{equation}\label{eq:tpsi-smoother'}
\tvarphi, \tpsi\in C\big([0,T],\mathbb{A}\mathcal{D}^{4,p}_0\big)
\end{equation}
where $\mathbb{A}\mathcal{D}^{m,p}_N$ with $N\ge 0$ denotes the set of orientation preserving $C^1$-smooth diffeomorphisms 
associated to the weighted Sobolev space $H^{4,p}_{\gamma_N}$ and defined in a similar fashion as $\mathcal{ZD}^{m,p}_N$ 
in \eqref{eq:Z-group} (see Definition 5.1 in \cite[Section 5]{McOTo2} for more detail). By Corollary 6.1 in \cite{McOTo2}, the composition map
\begin{equation}\label{eq:H-composition}
(f,\nu)\mapsto f\circ\nu,\quad H^{4,p}_{\gamma_0+1}\times\mathbb{A}\mathcal{D}^{4,p}_0\to H^{4,p}_{\gamma_0+1},
\end{equation}
is continuous.
\end{remark}

This remark allows us to conclude from \eqref{eq:tomega} that
\begin{equation}\label{eq:tomega-smoother}
\tomega\in C\big([0,T],H^{4,p}_{\gamma_0+1}\big)
\end{equation}
(and hence, by \eqref{eq:u-tilde}, $\tu-c\in C\big([0,T],H^{5,p}_{\gamma_0}\big)$).
A simple approximation argument involving \eqref{eq:tpsi-smoother'}, \eqref{eq:tomega-smoother}, and the continuity of the
composition \eqref{eq:H-composition} then implies that we can differentiate equality \eqref{eq:tomega} four times and 
apply the product rule to obtain the following expressions for the differentials of $\tomega$ of higher order,
\begin{align}\label{eq:higher_order_differentials}
d\tomega&=\big[(d\omega_0)\circ\tpsi\big](d\tpsi),\nonumber\\
d^2\tomega&=\big[(d^2\omega_0)\circ\tpsi\big](d\tpsi,d\tpsi)+\big[(d\omega_0)\circ\tpsi\big](d^2\tpsi),\nonumber\\
d^3\tomega&=\big[(d^3\omega_0)\circ\tpsi\big](d\tpsi,d\tpsi,d\tpsi)+3\big[(d^2\omega_0)\circ\tpsi\big](d^2\tpsi,d\tpsi)
+\big[(d\omega_0)\circ\tpsi\big](d^3\tpsi),\\
d^4\tomega&=\big[(d^4\omega_0)\circ\tpsi\big](d\tpsi,d\tpsi,d\tpsi,d\tpsi)+6\big[(d^3\omega_0)\circ\tpsi\big](d^2\tpsi,d\tpsi,d\tpsi)
+4\big[(d^2\omega_0)\circ\tpsi\big](d^3\tpsi,d\tpsi)\nonumber\\
&+3\big[(d^2\omega_0)\circ\tpsi\big](d^2\tpsi,d^2\tpsi)+\big[(d\omega_0)\circ\tpsi\big](d^4\tpsi).\nonumber
\end{align}
These formulas together with the fact that $\omega_0\in W^{4,p}_{\gamma_0+1}$, $\tpsi\in C\big([0,T],\diff^{3,\gamma}(\R^2)\big)$, 
as well as Lemma \ref{lem:sobolev-holder}, formula \eqref{eq:tpsi-smoother}, \eqref{eq:tpsi-smoother'}, and 
Proposition \ref{prop:main_technical} (ii) with Remark \ref{rem:R^d}, imply that
\begin{equation}\label{eq:tomega_step_one}
\tomega\in C\big([0,T],W^{2,p}_{\gamma_0+1}\big)\quad\text{\rm and}\quad
d^3\tomega, d^4\tomega\in C\big([0,T], L^p_{\gamma_0+3}\big).
\end{equation}
In fact, since $d\omega_0\in W^{3,p}_{\gamma_0+2}\subseteq L^p_{\gamma_0+2}$ and
$\tpsi\in C\big([0,T],\diff^{3,\gamma}(\R^2)\big)$, we obtain from Lemma \ref{lem:sobolev-holder} that
$(d\omega_0)\circ\tpsi\in C\big([0,T], L^p_{\gamma_0+2}\big)$. By combining this with the fact that
$d\tpsi\in C\big([0,T], L^\infty\big)$, we conclude from the first equality in \eqref{eq:higher_order_differentials} that
\begin{equation}\label{eq:1-1}
d\tomega\in C\big([0,T], L^p_{\gamma_0+2}\big).
\end{equation}
Now, consider the second equality in \eqref{eq:higher_order_differentials}. 
Since $d^2\omega_0\in W^{2,p}_{\gamma_0+3}\subseteq L^p_{\gamma_0+3}$ and
$\tpsi\in C\big([0,T],\diff^{3,\gamma}(\R^2)\big)$, we conclude from Lemma \ref{lem:sobolev-holder} that
$(d^2\omega_0)\circ\tpsi\in C\big([0,T], L^p_{\gamma_0+3}\big)$. By combining this with the fact that
$d\tpsi\in C\big([0,T], L^\infty\big)$ we obtain that 
\begin{equation}\label{eq:2-1}
\big[(d^2\omega_0)\circ\tpsi\big](d\tpsi,d\tpsi)\in C\big([0,T], L^p_{\gamma_0+3}\big).
\end{equation}
Let us now consider the term $\big[(d\omega_0)\circ\tpsi\big](d^2\tpsi)$ in the second equality in \eqref{eq:higher_order_differentials}. 
By Proposition \ref{prop:main_technical} (ii) and Remark \ref{rem:R^d} we have that
$d\omega_0\in W^{3,p}_{\gamma_0+2}\subseteq L^\infty_{\gamma_0+2+(2/p)}$.
This and Remark \ref{rem:sobolev-holder} implies that 
$(d\omega_0)\circ\tpsi\in C\big([0,T],L^\infty_{\gamma_0+2+(2/p)}\big)$.
Since $d^2\tpsi\in C\big([0,T],L^\infty\big)$ and, by \eqref{eq:tpsi-smoother}, $d^2\tpsi\in C\big([0,T],L^p_{\gamma_0+1}\big)$, 
we then conclude from $0<\gamma_0+(2/p)<1$ that
\begin{equation}\label{eq:2-2}
\big[(d\omega_0)\circ\tpsi\big](d^2\tpsi)\in C\big([0,T], L^p_{\gamma_0+3}\big).
\end{equation}
By combining \eqref{eq:2-1} and \eqref{eq:2-2} we obtain from \eqref{eq:higher_order_differentials} that
\[
d^2\tomega\in C\big([0,T], L^p_{\gamma_0+3}\big).
\]
This, together with \eqref{eq:tomega_preliminary} and \eqref{eq:1-1}, then implies the first relation in \eqref{eq:tomega_step_one}.
Further, we argue as follows.
Since $d\omega_0\in W^{3,p}_{\gamma_0+2}\subseteq H^{3,p}_{\gamma_0+2}$ and
$\tpsi\in C\big([0,T],\diff^{3,\gamma}(\R^2)\big)$, we conclude from Lemma \ref{lem:sobolev-holder} that 
$(d\omega_0)\circ\tpsi\in C\big([0,T], H^{3,p}_{\gamma_0+2}\big)$
and $(d^2\omega_0)\circ\tpsi\in C\big([0,T], H^{2,p}_{\gamma_0+3}\big)$.
Then, by using that $d\big[(d\omega_0)\circ\tpsi\big]=\big[(d^2\omega_0)\circ\tpsi\big](d\tpsi)$
and $d\tpsi\in C\big([0,T], L^\infty\big)$, we obtain that 
$d\big[(d\omega_0)\circ\tpsi\big]\in C\big([0,T],L^p_{\gamma_0+3}\big)$.
Hence, $(d\omega_0)\circ\tpsi\in C\big([0,T], W^{1,p}_{\gamma_0+2}\big)$.
By arguing in the same way we obtain the more general relation
\begin{equation}\label{eq:omega_0-derivatives_transport_general}
(d^k\omega_0)\circ\tpsi\in C\big([0,T], W^{1,p}_{\gamma_0+k+1}\big),\quad 0\le k\le 3.
\end{equation}
This allows us to prove the second relation in \eqref{eq:tomega_step_one}.
In fact, the argument used to prove that $\big[(d\omega_0)\circ\tpsi\big](d^2\tpsi)\in C\big([0,T], L^p_{\gamma_0+3}\big)$
and formula \eqref{eq:tpsi-smoother} show that $\big[(d\omega_0)\circ\tpsi\big](d^4\tpsi)\in C\big([0,T], L^p_{\gamma_0+3}\big)$.
In addition, we conclude from \eqref{eq:omega_0-derivatives_transport_general} and $\tpsi\in C\big([0,T],\diff^{3,\gamma}(\R^2)\big)$,
that the other terms appearing on the right hand sides of the last two equalities in \eqref{eq:higher_order_differentials}
also belong to $C\big([0,T], L^p_{\gamma_0+3}\big)$. This completes the proof of \eqref{eq:tomega_step_one}.


Further, we argue as follows. Note that the condition \eqref{eq:tomega_step_one} is equivalent to
$\partial^\alpha\tomega\in C\big([0,T],W^{2,p}_{\gamma_0+1}\big)$, $0\le|\alpha|\le 2$.
By combining this with Theorem \ref{th:cauchy_m>=0} $(i)$, Remark \ref{rem:Z->Z-tilde}, 
and \eqref{eq:u-tilde} we obtain that
\begin{equation}\label{eq:tu_step_two}
\tu-c\in C\big([0,T],W^{3,p}_{\gamma_0}\big)\quad\text{and}\quad d^4\tu, d^5\tu\in C\big([0,T],L^p_{\gamma_0+3}\big).
\end{equation}
This, together with Remark \ref{rem:H-group}, allows us to differentiate $A\equiv(d\tu)\circ\tvarphi$ three times and by using the product rule
obtain the following formulas for the differentials of $A$ of higher order
\begin{align}\label{eq:higher_order_differentials_for_A}
dA&=\big[(d^2\tu)\circ\tvarphi\big](d\tvarphi),\nonumber\\
d^2A&=\big[(d^3\tu)\circ\tvarphi\big](d\tvarphi,d\tvarphi)+\big[(d^2\tu)\circ\tvarphi\big](d^2\tvarphi),\\
d^3A&=\big[(d^4\tu)\circ\tvarphi\big](d\tvarphi,d\tvarphi,d\tvarphi)+3\big[(d^3\tu)\circ\tvarphi\big](d^2\tvarphi,d\tvarphi)
+\big[(d^2\tu)\circ\tvarphi\big](d^3\tvarphi).\nonumber
\end{align}
It follows from \eqref{eq:higher_order_differentials_for_A}, \eqref{eq:tu_step_two}, \eqref{eq:tphi-smoother},
Proposition \ref{prop:main_technical} (ii), Proposition \ref{prop:W_R}, and Remark \ref{rem:R^d} that
\begin{equation}
A\in C\big([0,T],B^{3,p}_{\gamma_0+1}\big)
\end{equation}
where 
\[
B^{3,p}_{\gamma_0+1}:=
\big\{f\in H^{3,p}_{loc}(\C,\C)\,\big|\,f\in L^p_{\gamma_0+1},\,\partial^\alpha f\in L^p_{\gamma_0+2},\,\,1\le|\alpha|\le 3\big\}.
\]
The space $B^{3,p}_{\gamma_0+1}$ is a Banach algebra. This can be easily seen from the product rule, Proposition \ref{prop:W_R} and Remark \ref{rem:R^d},
and the fact that
$B^{3,p}_{\gamma_0+1}=\big\{f\in H^{3,p}_{loc}(\C,\C)\,\big|\,\partial^\alpha f\in W^{1,p}_{\gamma_0+1}\cap W^{2,p}_{\gamma_0},\,
0\le|\alpha|\le 1\big\}$.
This allows us to conclude from equation \eqref{eq:jacobian_evolution}, considered as an ODE in the Banach algebra $B^{3,p}_{\gamma_0+1}$,
that $d\tvarphi-I\in C\big([0,T],B^{3,p}_{\gamma_0+1}\big)$. By combining this with \eqref{eq:tpsi-formula} and then
using that\footnote{This follows easily from Lemma \ref{lem:sobolev-holder}.}
\[
(f,\psi)\mapsto f\circ\psi,\quad B^{3,p}_{\gamma_0+1}\times\diff^{3,\gamma}(\R^2)\to B^{3,p}_{\gamma_0+1},
\]
and the Banach algebra property of $B^{3,p}_{\gamma_0+1}$ (cf. the proof of Lemma \ref{lem:division}) we obtain that
\begin{equation*}
d\tpsi-I\in C\big([0,T],B^{3,p}_{\gamma_0+1}\big).
\end{equation*}
With this additional spatial decay of $d\tpsi-I$ and its derivatives we return to the system of equations
\eqref{eq:higher_order_differentials} to conclude that $\tomega\in C\big([0,T],W^{3,p}_{\gamma_0+1}\big)$, and hence
\begin{equation*}
\tu-c\in C\big([0,T],W^{4,p}_{\gamma_0}\big).
\end{equation*}
Now, we have enough regularity to conclude from \eqref{eq:ode} and Lemma \ref{lem:ode} that
\begin{equation}\label{eq:step_three}
\tvarphi,\tpsi\in C^1\big([0,T],\mathcal{ZD}^{4,p}_0).
\end{equation}
(Note that for any given $t\in(0,T]$ the diffeomorphism $\tpsi(t)$ is equal to $\eta(t)$ where $\eta\in C^1\big([0,t],\mathcal{ZD}^{4,p}_0\big)$
is the solution of the equation $\dt\eta(s)=-u(t-s)\circ\eta(s)$, $\eta|_{s=0}=\idmap$.)
By applying the bootstrapping arguments above one more time we now easily obtain from \eqref{eq:higher_order_differentials}
and \eqref{eq:step_three} that $\tomega\in C\big([0,T],W^{4,p}_{\gamma_0+1}\big)$, and hence $\tu-c\in C\big([0,T],W^{5,p}_{\gamma_0}\big)$.
By \eqref{eq:on_[0,tau)}, this contradicts the blow-up of the $\mathcal{Z}^{5,p}_0$-norm in \eqref{eq:blow-up_in _norm}.
Therefore, $\tau>T$. This completes the proof of the theorem.
\end{proof}

\begin{remark}\label{rem:more_decay}
By arguing as in the proof of Proposition \ref{prop:more_decay} we also see that if for $T>0$ and
$u_0\in\mathcal{Z}^{4,p}_0\subseteq c^{2,\gamma}_b$ with $p>2$ there exists a solution of the 2d Euler equation in Lagrangian coordinates 
$\varphi\in C^1\big([0,T],\diff^{2,\gamma}(\R^2)\big)$ then the solution of the 2d Euler equation given by Theorem \ref{th:main} extends 
to the whole interval $[0,T]$ and $u\in C\big([0,T],\mathcal{Z}^{4,p}_0\big)\cap C^1\big([0,T],\mathcal{Z}^{3,p}_0\big)$.
\end{remark}

Now, we are ready to prove the main result of this section.

\begin{theorem}\label{th:main_global}
For any $u_0\in\mathcal{Z}^{m,p}_N$ with $m>3+(2/p)$ the Euler equation (\ref{EQeulerCauchyop}) has a unique solution 
$u\in C\big([0,\infty),\mathcal{Z}^{m,p}_N\big)\cap C^1\big([0,\infty),\mathcal{Z}^{m-1,p}_N\big)$ such that the pressure  
$\p(t)$ lies in $\mathcal{Z}^{m+1,p}_N$ for any $t\ge 0$. This solution depends continuously on the initial data 
$u_0\in B_{\mathcal{Z}^{m,p}_N}(\rho)$ in the sense that for any $T>0$ the map
$u_0\mapsto u$, $\mathcal{Z}^{m,p}_N\to C\big([0,T],\mathcal{Z}^{m,p}_N\big) \cap C^1\big([0,T],\mathcal{Z}^{m-1,p}_N\big)$,
is continuous.
\end{theorem}
\begin{proof}[Proof of Theorem \ref{th:main_global}]
Assume that $m>3+(2/p)$ and take $u_0\in\mathcal{Z}^{m,p}_N$. First, consider the case when $m\ge 5$ and $1<p<\infty$.
Then $u_0\in\mathcal{Z}^{m,p}_N\subseteq\mathcal{Z}^{5,p}_0$. By Remark \ref{rem:Z->holder}, the space $\mathcal{Z}^{5,p}_0$ is continuously 
embedded into the little H\"older space $c^{3,\gamma}_b$. Then, it follows from the results in \cite[Theorem 2]{Serfati} and \cite{Serfati1,Serfati2,MY} that there 
exists a solution $\varphi\in C^1\big([0,\infty),\diff^{3,\gamma}(\R^2)\big)$ of 2d Euler equation in Lagrangian coordinates.

\begin{remark}\label{rem:solutions_in_the_little_holder_space}
It follows from the results in \cite[Theorem 2]{Serfati} (see also \cite[Theorem 2]{Serfati1},\cite{MY}) that for any $T>0$ and $m\ge 1$ the 
2d Euler equation has a unique solution $\varphi\in C^1\big([0,T],\Diff^{m,\gamma}(\R^2)\big)$ in Lagrangian coordinates on the group of 
orientation preserving diffeomorphisms of $\R^2$ of H\"older class $C^{m,\gamma}_b$. Since these solutions depend continuously on the initial data 
on $\Diff^{m,\gamma}(\R^2)\times C^{m,\gamma}_b$, one obtains, by approximating the initial data in the little H\"older space 
$c^{m,\gamma}_b$ by functions in $C^\infty_b$, that $\varphi(t)\in\diff^{m,\gamma}(\R^2)$ for any $u_0\in c^{m,\gamma}$ and 
for any $t\ge 0$. Hence $\varphi\in C^1\big([0,T],\diff^{m,\gamma}(\R^2)\big)$ is a solution of the 2d Euler 
equation in Lagrangian coordinates on $\diff^{m,\gamma}(\R^2)$.
\end{remark}

By Proposition \ref{prop:more_decay} we then conclude that the local solution of the 2d Euler equation in $\mathcal{Z}^{5,p}_0$ given by 
Theorem \ref{th:main} extends to a global solution
\[
u\in C\big([0,\infty), \mathcal{Z}^{5,p}_0\big)\cap C^1\big([0,\infty),\mathcal{Z}^{4,p}_0\big).
\] 
In the case when $m>5$ we have that $u_0\in\mathcal{Z}^{6,p}_0$ and  $\omega_0=\partial_z u\in\widetilde{\mathcal{Z}}^{5,p}_1$.
Then, by Lemma \ref{lem:ode} and Remark \ref{rem:longer_expansion},
\[
\omega(t)=\omega_0\circ\varphi^{-1}\in C\big([0,\infty),\widetilde{\mathcal{Z}}^{5,p}_1\big).
\]
This implies that $u:=\partial_z^{-1}\omega\in C\big([0,\infty),\mathcal{Z}^{6,p}_0\big)$.
Since there is no blow-up in the $\mathcal{Z}^{6,p}_0$-norm of the curve $u$, we conclude from
Theorem \ref{th:main} that 
$u\in C\big([0,\infty),\mathcal{Z}^{6,p}_0\big)\cap C^1\big([0,\infty),\mathcal{Z}^{5,p}_0\big)$
is a solution of the 2d Euler equation in $\mathcal{Z}^{6,p}_0$.
By repeating this process inductively, we obtain that 
\[
u\in C\big([0,\infty),\mathcal{Z}^{m,p}_0\big)\cap C^1\big([0,\infty),\mathcal{Z}^{m-1,p}_0\big)
\]
is a solution of the 2d Euler equation in $\mathcal{Z}^{m,p}_0$. 
Similarly, by applying Proposition \ref{prop:longer_expansion} inductively, we increase from $0$ to $N$ the order of 
the asymptotic expansion of $u$ and obtain that 
$u\in C\big([0,\infty),\mathcal{Z}^{m,p}_N\big)\cap C^1\big([0,\infty),\mathcal{Z}^{m-1,p}_N\big)$
is a solution of the 2d Euler equation in $\mathcal{Z}^{m,p}_N$.
This solution is unique and depends continuously on the initial data by Theorem \ref{th:main}.
The analyticity in time of the asymptotic coefficients of the solution follows from the last statement of Theorem \ref{th:main}.
The case when $m=4$ and $p>2$ follows from Remark \ref{rem:more_decay} and 
Remark \ref{rem:solutions_in_the_little_holder_space}.
\end{proof}

\appendix 
\section{Appendix: Auxiliary results} \label{appendix:technical}
In this Appendix we prove a number of technical results used in the main body of the paper.
Our first task is to study the properties of the auxiliary asymptotic space $\mathcal{Z}^{m,p}_N(B_R^c)$.
To this end we first study the properties of its reminder space $W^{m,p}_\delta(B_R^c)$ where 
$B_R^c\equiv\big\{{\rm x}\in\R^d\,\big|\,|{\rm x}|>R\big\}$ with $R>0$. For a non-negative real number $\beta\in\R$ denote by
$\left\lfloor{\beta}\right\rfloor$ its integer part, i.e. the largest integer number that is less or equal than $\beta$. 
Similarly, let $\{\beta\}:=\beta-\left\lfloor{\beta}\right\rfloor$ be the fractional part of $\beta\ge 0$. We have the following weighted version of 
the Sobolev embedding theorem.

\begin{proposition}\label{prop:main_technical}
Assume that $1<p<\infty$, $\delta\in\R$, and $R\ge 1$. Then
\begin{itemize}
\item[(i)] If $0\le m<d/p$ then for any $q\in\big[p,d/(\frac{d}{p}-m)\big]$ one has that
$W_{\delta-(d/p)}^{m,p}(B_R^c)\subseteq L_{\delta -(d/p)}^q(B_R^c)$. Moreover, there exists a constant
$C\equiv C(d,p,q,m,\delta)>0$ independent of the choice of $R\ge 1$ such that for any $f \in W_{\delta-(d/p)}^{m,p}(B_R^c)$,
\begin{equation*}
\norm{f}_{L_{\delta-(d/p)}^q(B_R^c)} \le C\norm{f}_{W^{m,p}_{\delta -(d/p)}(B_R^c)}.
\end{equation*}
If $m=d/p$ then the above statement holds for any $q\in\left[p,\infty\right)$.
\item[(ii)] For $m>d/p$ and for any integer $0\le k<m-(d/p)$ one has that 
$W_{\delta-(d/p)}^{m,p}(B_R^c)\subseteq C^k(B_R^c)$ and for any $|\alpha|\le k$,
\begin{equation*}
\sup_{{\rm x} \in B_R^c}\left(\x^{\delta+|\alpha|}\abnorm{D^{\alpha}f({\rm x})}\right) \le 
C\norm{f}_{W^{m,p}_{\delta-(d/p)}(B_R^c)},\quad D^\alpha\equiv\partial_{x_1}^{\alpha_1}\cdots\partial_{x_d}^{\alpha_d},
\end{equation*}
with a constant $C\equiv C(d,p,m,\delta)>0$ {\em independent} of the choice of $R\ge 1$.
\item[(iii)] Assume that $m>d/p$, $0\le k<m-(d/p)$, and $\delta+(d/p)>0$. 
Choose a real number $\gamma$ such that $0<\gamma<\big\{m-k-(d/p)\big\}$ for
$\big\{m-k-(d/p)\big\}\ne 0$ or $0<\gamma<1$ for $\big\{m-k-(d/p)\big\}=0$. Then, 
\[
W_\delta^{m,p}(\R^d)\subseteq C^{k,\gamma}_b(\R^d)
\]
and the embedding is bounded.
\end{itemize}
\end{proposition}

\begin{remark}
An important technical part of Proposition \ref{prop:main_technical} is that
the constants $C>0$ appearing in item (i) and (ii) are independent of the choice of $R\ge 1$.
\end{remark}

\begin{proof}[Proof of Proposition \ref{prop:main_technical}]
Let us first prove item $(i)$. Assume that $0\le m<d/p$. For a given open set $\mathcal{U}\subseteq\R^d$ denote
\begin{align*}
\vertiii{f}_{m,p,\delta;\mathcal{U}} := 
\left(\sum_{0 \le\alpha \le m } \int_{\mathcal{U}} \abnorm{\abnorm{\rm x}^{\delta+|\alpha|}D^{\alpha}f}^p\,d{\rm x}\right)^{1/p}.
\end{align*}
This norm is equivalent to 
$\|f\|_{W^{m,p}_\delta(\mathcal{U})}:=
\sum_{0 \le\alpha \le m }\big\|\abnorm{\rm x}^{\delta + |\alpha|}D^{\alpha}f\big\|_{L^p(\mathcal{U})}$.
For $R\ge 1$ consider the annulus $A_R := B_{2R}\setminus{\bar{B}}_R$ where $B_R$ denotes the open ball of radius $R$ 
in $\R^d$ centered at zero and ${\bar{B}}_R$ is its closure. By changing the variables in the corresponding integrals one easily
sees that for any $f\in C^\infty({\bar{A}}_R)$,
\begin{align}\label{eq:scaling_weighted_sobolev}
\vertiii{f_R}_{m,p,\delta;A_1}=R^{-\delta-(d/p)}\vertiii{f}_{m,p,\delta;A_R},
\end{align}
where $f_R({\rm x}) := f(R{\rm x})$ for ${\rm x}\in\R^d$. By using that $R\le|{\rm x}|\le 2R$ on $A_R$ we also see that
\begin{align*}
\sum_{0 \le\alpha\le m}\int_{A_R}\left(|{\rm x}|^{\delta -(d/p)+ |\alpha|} R^{\frac{d}{p}}|D^{\alpha}f|\right)^p\,d{\rm x} &
\le \vertiii{f}^p_{m,p,\delta;A_R}\\ &\le\sum_{0 \le \alpha \le m} 
\int_{A_R}\left(|{\rm x}|^{\delta -(d/p)+ |\alpha|} (2R)^{\frac{d}{p}}|D^{\alpha}f|\right)^p\,d{\rm x},
\end{align*}
which implies
\begin{align*}
R^{\frac{d}{p}}\vertiii{f}_{m,p,\delta-(d/p);A_R} \le \vertiii{f}_{m,p,\delta;A_R}\le
(2R)^{\frac{d}{p}}\vertiii{f}_{m,p,\delta-(d/p);A_R}. 
\end{align*}
This, together with \eqref{eq:scaling_weighted_sobolev}, then shows that for any $0\le m<d/p$ there exist constants $0<M_1<M_2$ 
independent of $R\ge 1$ such that for any $f\in C^\infty_c(A_R)$,
\begin{align}\label{WSpacePropEq3}
M_1R^{-\delta}\vertiii{f}_{m,p,\delta-(d/p);A_R}\le \vertiii{f_R}_{m,p,\delta;A_1}\le
M_2R^{-\delta}\vertiii{f}_{m,p,\delta-(d/p);A_R}. 
\end{align}
By the Sobolev embedding theorem for any given $q\in\big[p,d/(\frac{d}{p}-m)\big]$ there exists a constant 
$\widetilde{C}\equiv C(d,q)>0$ such that for any $g\in W^{m,p}(A_1)$,
\begin{align}\label{WSpacePropEq4}
\norm{g}_{L^q(A_1)}\le\widetilde{C}\norm{g}_{W^{m,p}(A_1)}. 
\end{align}
By applying \eqref{WSpacePropEq4} to $|{\rm x}|^{\delta}g({\rm x})$ with $g\in C^{\infty}_c(A_1)$ and then using that
$\abnorm{D^{\alpha}|{\rm x}|^{\delta}}\le C_{\alpha,\delta}|{\rm x}|^{\delta - |\alpha|}$ we obtain
\begin{align*}
\norm{g}_{L^q_\delta(A_1)}&\equiv\norm{\abnorm{\rm x}^{\delta}g}_{L^q(A_1)}
\le\widetilde{C}\sum_{|\alpha|\le m}\norm{D^{\alpha}(|{\rm x}|^{\delta}g({\rm x}))}_{L^p(A_1)} 
\le C_1\sum_{|\alpha|\le m}\sum_{\beta\le\alpha}\norm{\big(D^{\alpha-\beta}|{\rm x}|^{\delta}\big)
D^{\beta}g}_{L^p(A_1)}\\
&\le C_2\sum_{|\alpha|\le m}\sum_{\beta\le\alpha}\norm{|{\rm x}|^{\delta-|\alpha|+|\beta|}D^{\beta}g}_{L^p(A_1)}  
\le C_2\sum_{|\alpha|\le m}\sum_{\beta\le\alpha} \norm{|{\rm x}|^{\delta+|\beta|}D^{\beta}g}_{L^p(A_1)}\\
&\le C_3\sum_{|\beta|\le m}\norm{|{\rm x}|^{\delta+|\beta|}D^{\beta}g}_{L^p(A_1)} \le C_4 \vertiii{g}_{m,p,\delta;A_1}.
\end{align*}
By taking $g = f_R$ with $f\in C^\infty_c(A_R)$ we get 
$\vertiii{f_R}_{L^q_\delta(A_1)} \le C_4 \vertiii{f_R}_{p,m,\delta; A_1}$ 
which together with \eqref{WSpacePropEq3} and the fact that
$\vertiii{\cdot}_{L^q_\delta(\mathcal{U})}\equiv\vertiii{\cdot}_{0,q,\delta;\mathcal{U}}$ for any $\delta\in\R$ and an open set
$\mathcal{U}$ in $\R^d$ implies,
\begin{align*}
R^{-\delta}\vertiii{f}_{L^q_{\delta-\frac{d}{q}}(A_R)}\le C_5\vertiii{f_R}_{L^q_\delta(A_1)}
\le C_6 \vertiii{f_R}_{m,p,\delta;A_1} \le C_7 R^{-\delta}\vertiii{f}_{m,p,\delta-(d/p);A_R}.
\end{align*}
Hence, for any $f\in C^\infty_c(A_R)$,
\begin{align}\label{WSpacePropEq5}
\vertiii{f}_{L^q_{\delta-\frac{d}{q}}(A_R)} \le C \vertiii{f}_{m,p,\delta-(d/p);A_R},
\end{align}
with $C>0$ independent of $R\ge 1$.
Finally, we write $B_R^c=\bigcup_{k\ge 0}{\bar{A}}_{2^kR}$ and by using \eqref{WSpacePropEq5}
obtain that for any $f\in C^\infty_c(B_R^c)$,
\begin{align*}
\norm{f}_{L^q_{\delta - \frac{d}{q}}(B_R^c)} &= \left(\sum_{k=0}^{\infty} \int_{A_{2^kR}} 
\left(|{\rm x}|^{\delta-\frac{d}{q}}|f|\right)^q\right)^{\frac{1}{q}} = \left(\sum_{k=0}^{\infty}  
\vertiii{f}^q_{L^{q}_{\delta-\frac{d}{q}}(A_{2^kR})}\right)^{\frac{1}{q}}\\
&\le C\; \left(\sum_{k=0}^{\infty} \vertiii{f}^q_{m,p,\delta-(d/p);A_{2^kR}}\right)^{\frac{1}{q}} \le 
C\; \left(\sum_{k=0}^{\infty} \vertiii{f}^p_{m,p,\delta-(d/p);A_{2^kR}}\right)^{\frac{1}{p}} \\
&= C \; \vertiii{f}_{W^{m,p}_{\delta-(d/p)}(B_R^c)}
\end{align*}
where we used Lemma \ref{lem:inequality} below. The case when $m=d/p$ is treated in the same way. This proves item $(i)$.

In order to prove item $(ii)$, assume that $m>d/p$.
By the Sobolev embedding theorem and the inequality
$\abnorm{D^{\beta}|{\rm x}|^{\delta + |\alpha|}} \le C_ {\alpha,\delta} |{\rm x}|^{\delta +|\alpha|-|\beta|}$ we get that
for any $g\in C^\infty_c(A_1)$, an integer $0\le k<m-(d/p)$, and $|\alpha|\le k$, 
\begin{align}
\sup_{{\rm x}\in{A_1}}\Big(\abnorm{|{\rm x}|^{\delta+|\alpha|}D^{\alpha}g({\rm x})}\Big) 
&\le C_1\sum_{|\beta|\le m-|\alpha|}\norm{D^{\beta}\big(|{\rm x}|^{\delta+|\alpha|}D^{\alpha}g\big)}_{L^p(A_1)}\nonumber\\
&\le C_2\sum_{|\beta|\le m-|\alpha|}\sum_{\gamma\le\beta}\norm{|{\rm x}|^{\delta+|\alpha|+|\gamma|-|\beta|}
D^{\alpha+\gamma}g}_{L^p(A_1)}\label{eq:useful_chunk}\\
&\le C_3\sum_{|\alpha+\gamma|\le m}\norm{|{\rm x}|^{\delta+|\alpha+\gamma|}
D^{\alpha+\gamma}g}_{L^p(A_1)}\le C_4\vertiii{g}_{m,p,\delta;A_1}.\nonumber
\end{align}
By taking $g=f_R$ with $f\in C^\infty_c(A_R)$ in this inequality and then using that 
$D^{\alpha}f_R = R^{|\alpha|}(D^{\alpha}f)_R$ we obtain
\begin{align*}
R^{-\delta} \sup_{y \in A_R}\big(|{\rm y}|^{\delta + |\alpha|}\abnorm{D^{\alpha}f({\rm y})}\big)
\le C_3\vertiii{f_R}_{m,p,\delta-(d/p);A_1}=C_3 R^{-\delta}\vertiii{f}_{m,p,\delta-(d/p);A_R}
\end{align*}
where we used \eqref{eq:scaling_weighted_sobolev} in the last step.
Hence,
\begin{align*}
\sup_{{\rm y} \in A_R}\big(|{\rm y}|^{\delta + |\alpha|}|D^{\alpha}f({\rm y})|\big) \le 
C_3 \vertiii{f}_{m,p,\delta -(d/p);A_R}\le 
C_3 \vertiii{f}_{m,p,\delta -(d/p)}
\end{align*}
with a constant $C_3>0$ independent of the choice of $R\ge 1$.
By taking $\sup$ over $B_R^c$ we then conclude the proof of item $(ii)$ and
Proposition \ref{prop:main_technical}.

Let us now prove item (iii). Assume that $m>d/p$, $0\le k+(d/p)<m$, $\delta+(d/p)>0$, and 
choose $\gamma\in\R$ as described in the proposition.
It follows from item (ii) that $W^{m,p}_{\delta}(B_1^c)\subseteq C^k(B_1^c)$ and that for any $|\alpha|\le k$ and 
$f\in W^{m,p}_{\delta}(B_1^c)$ we have
\begin{equation*}
\sup_{{\rm x} \in B_1^c}\left(\x^{(\delta+(d/p))+|\alpha|}\abnorm{D^{\alpha}f({\rm x})}\right) \le 
C\norm{f}_{W^{m,p}_{\delta}(B_1^c)}<\infty.
\end{equation*}
Since $\delta+(d/p)>0$ we then obtain that $W^{m,p}_{\delta}(\R^d)\subseteq C^k_b(\R^d)$.
Let us now estimate the H\"older semi-norm of $D^\alpha f$ with $|\alpha|\le k$.
For $R\ge 1$ consider the open annulus $A_R$ in $\R^d$ defined above and let $|\cdot|_{0,\gamma;A_R}$ and 
$[\cdot]_{\gamma;{\bar{A}}_R}$ be the H\"older norm and semi-norm in the closure ${\bar{A}}_R$ of $A_R$ in $\R^d$.
For any $\delta\in\R$ and for any $g\in C^k({\bar{A}}_1)$ we have
\begin{equation}\label{eq:preparation1}
[g]_{\gamma;{\bar{A}}_1}=\big[|{\rm x}|^{-\delta}\big(|{\rm x}|^\delta g\big)\big]_{\gamma;{\bar{A}}_1}
\le 2^{|\delta|}\big[|{\rm x}|^\delta g\big]_{\gamma;{\bar{A}}_1}+K_\delta\big||{\rm x}|^\delta g\big|_{L^\infty({\bar{A}}_1)}
\le C_\delta\big| |{\rm x}|^\delta g\big|_{\gamma;{\bar{A}}_1}
\end{equation}
with constants depending only on the choice of $\delta\in\R$.
Here we used that for any non-empty $U\subseteq\R^d$ we have that
$[fg]_\gamma\le |f|_{L^\infty(U)}[g]_{\gamma;U}+|g|_{L^\infty(U)}[f]_{\gamma;U}$
for any $f,g\in C^{0,\gamma}(U)$. By the Sobolev embedding theorem in the domain $A_1$ and by the assumptions on the regularity exponents 
$m$, $k$, and $\gamma$, we obtain that for any $|\alpha|\le k$ there exists a constant $C>0$, as well as constants $C_1, C_2, C_3>0$, 
such that for any $g\in H^{m,p}({\bar{A}}_1)$ we have that
\begin{align}
\big| |{\rm x}|^\delta D^\alpha g\big|_{\gamma;{\bar{A}}_1}
&\le C\sum_{|\beta|\le m-|\alpha|}\norm{D^{\beta}\big(|{\rm x}|^\delta D^{\alpha}g\big)}_{L^p(A_1)}
\le C_1\sum_{|\beta|\le m-|\alpha|}\sum_{\gamma\le\beta}\norm{|{\rm x}|^{\delta+|\gamma|-|\beta|}
D^{\alpha+\gamma}g}_{L^p(A_1)}\nonumber\\
&\le C_2\sum_{|\alpha+\gamma|\le m}\norm{|{\rm x}|^{\delta+|\alpha+\gamma|}
D^{\alpha+\gamma}g}_{L^p(A_1)}\le C_3\vertiii{g}_{m,p,\delta;A_1}\label{eq:preparation2}
\end{align}
where we argued as in \eqref{eq:useful_chunk}.
It follows from \eqref{eq:preparation1} and \eqref{eq:preparation2} that there exists a constant $C>0$ such that
for any $g\in H^{m,p}({\bar{A}}_1)$ and for any $|\alpha|\le k$ we have that
\begin{equation}\label{eq:basic_holder}
[D^\alpha g]_{\gamma;{\bar{A}}_1}\le C\,\vertiii{g}_{m,p,\delta;A_1}.
\end{equation}
Now take $R\ge 1$ and let $f\in H^{m,p}({\bar{A}}_R)$. Then we have 
\begin{align}
[D^\alpha f_R]_{\gamma;{\bar{A}}_1}
&=\sup_{{\rm x},{\rm y}\in{\bar{A}}_1;{\rm x}\ne{\rm y}}
\frac{\big|(D^\alpha f_R)({\rm x})-(D^\alpha f_R)({\rm y})\big|}{|{\rm x}-{\rm y}|^\gamma}\nonumber\\
&=R^{|\alpha|}\sup_{{\rm x},{\rm y}\in{\bar{A}}_1;{\rm x}\ne{\rm y}}
\frac{\big|(D^\alpha f)(R{\rm x})-(D^\alpha f)(R{\rm y})\big|}{|{\rm x}-{\rm y}|^\gamma}\nonumber\\
&=R^{|\alpha|+\gamma}[D^\alpha f]_{\gamma;{\bar{A}}_R}.\label{eq:scaling_holder}
\end{align}
By combining \eqref{eq:scaling_weighted_sobolev} and \eqref{eq:scaling_holder} with \eqref{eq:basic_holder} we obtain that
there exists a constant $C>0$ such that for any $R\ge 1$, $|\alpha|\le k$, and for any $f\in H^{m,p}({\bar{A}}_R)$,
\begin{equation}\label{eq:basic_holder_R}
[D^\alpha f]_{\gamma;{\bar{A}}_R}\le 
\frac{C}{R^{\mu}}\,\|f\|_{W^{m,p}_\delta(A_R)}
\end{equation}
where $\mu:=\Big(\delta+(d/p)\Big)+\gamma+|\alpha|>0$. In addition, by the Sobolev embedding theorem in 
the unit ball $B_1$ we also have that there exists a constant, denoted again by $C>0$, such that 
for any $f\in H^{m,p}({\bar{B}}_1)$,
\begin{equation}\label{eq:basic_holder_1}
[D^\alpha f]_{\gamma;{\bar{B}}_1}\le C\,\|f\|_{W^{m,p}(B_1)}
\end{equation}
where ${\bar{B}}_1$ denotes the closure of $B_1$ in $\R^d$.
Note that
\[
\R^d={\bar{B}}_1\cup\Big(\bigcup_{k\ge 0}{\bar{A}}_{2^k}\Big).
\]
Now, take $f\in C^\infty_b(\R^d)$ and two points ${\rm x},{\rm y}\in\R^d$ such that ${\rm x}\ne{\rm y}$. 
The closed segment $[{\rm x},{\rm y}]:=\big\{{\rm x}+s({\rm y}-{\rm x})\,\big|\,0\le s\le 1\big\}$ intersects 
any given annulus ${\bar{A}}_{2^k}$ with $k\ge 0$ as well as the closed ball ${\bar{B}}_1$ at no more than 
two straight segments:
\[
{\bar{A}}_{2^k}\cap[{\rm x},{\rm y}]=I_k'\cup I_k'',\quad{\bar{B}}_1\cap[{\rm x},{\rm y}]=I_{-1},
\]
where $I_k'$ and $I_k''$ with $k\ge 0$, and $I_{-1}$, can be empty. If such a segment is not empty we set
\[
I_k'=[{\rm x}_k',{\rm y}_k'],\quad I_k''=[{\rm x}_k',{\rm y}_k'],\quad I_{-1}=[{\rm x}_{-1},{\rm y}_{-1}].
\]
Then,
\[
[{\rm x},{\rm y}]=I_{-1}\cup\bigcup_{k\ge 0}\big(I_k'\cup I_k''\big).
\]
This, together with \eqref{eq:basic_holder_R} and \eqref{eq:basic_holder_1}, then implies that
for any $f\in W^{m,p}_\delta(\R^d)$ we have that
\begin{align*}
\frac{\big|(D^\alpha f)({\rm x})-(D^\alpha f)({\rm y})\big|}{|{\rm x}-{\rm y}|^\gamma}&\le
\frac{\big|(D^\alpha f)({\rm x}_{-1})-(D^\alpha f)({\rm y}_{-1})\big|}{|{\rm x}_{-1}-{\rm y}_{-1}|^\gamma}+
\sum_{k\ge 0}\frac{\big|(D^\alpha f)({\rm x}_k')-(D^\alpha f)({\rm y}_k')\big|}{|{\rm x}_k'-{\rm y}_k'|^\gamma}\\
&+\sum_{k\ge 0}\frac{\big|(D^\alpha f)({\rm x}_k'')-(D^\alpha f)({\rm y}_k'')\big|}{|{\rm x}_k''-{\rm y}_k''|^\gamma}
\le[D^\alpha f]_{\gamma;{\bar{B}}_1}+2\sum_{k\ge 0}[D^\alpha f]_{\gamma;{\bar{A}}_{2^k}}\\
&\le C\Big(1+2\sum_{k\ge 0}\frac{1}{2^{\mu k}}\Big)\|f\|_{W^{m,p}_{\delta}}
\end{align*}
where we omit the terms corresponding to empty intervals in the second estimate above and use that for any 
$k\ge 0$ we have that $\|f\|_{W^{m,p}_{\delta}(A_{2^k})}\le\|f\|_{W^{m,p}_\delta(\R^d)}$ and a similar inequality for the ball $B_1$.
Hence, there exists a positive constant $C\equiv C_\mu<\infty$ such that for any $f\in W^{m,p}_\delta(\R^d)$ and for 
any $|\alpha|\le k$ we have that $[D^\alpha f]_\gamma\le C\,\|f\|_{W^{m,p}_{\delta}}$. Since we already proved that 
$W^{m,p}_\delta(\R^d)\subseteq C^k_b(\R^d)$, this completes the proof of item (iii).
\end{proof}

Denote by $\ell^p$, $p\ge 1$, the Banach space of complex-valued sequences $a=(a_l)_{l\ge 1}$ with finite $\ell^p$-norm
$\|a\|_{\ell^p}:=\big(\sum_{l\ge 1}|a_l|^p\big)^{1/p}$.
In the proof of Proposition \ref{prop:main_technical} we use the following simple lemma.

\begin{lemma}\label{lem:inequality}
For $1\le p\le q$ one has that $\ell^p\subseteq\ell^q$ so that for any $a\in\ell^p$ we have
$\|a\|_{\ell^q}\le\|a\|_{\ell^p}$.
\end{lemma}
\begin{proof}[Proof of Lemma \ref{lem:inequality}]
Assume that $1\le p\le q$. Without loss of generality we can assume that $\|a\|_{\ell^p}=1$. 
Then, $|a_l|\le 1$, and hence $|a_l|^q\le|a_l|^p$ for any $l\ge 1$.
We have,
\[
\|a\|_{\ell^q}^q=\sum_{l\ge 1}|a_l|^q\le \sum_{l\ge 1}|a_l|^p=\|a\|_{\ell^p}^p=1.
\]
This completes the proof of the lemma.
\end{proof}

Proposition \ref{prop:main_technical} allows us to prove the following multiplicative property of 
the weighted Sobolev space $W^{m,p}_\delta(B_R^c)$.

\begin{proposition}\label{prop:W_R}
For any real $\delta_1,\delta_2\in\R$ and integers $0 \le k \le l \le m$ with $m+l-k>d/p$ there exists a constant 
$C\equiv C(d,p,l,k,m,\delta_1,\delta_2)>0$ {\em independent} of the choice of $R\ge 1$ such that for any 
$f\in W^{m,p}_{\delta_1 -(d/p)}(B_R^c)$ and for any $g\in W^{l,p}_{\delta_2 -(d/p)}(B_R^c)$ we have 
that $fg\in W^{k,p}_{\delta_1 + \delta_2 -(d/p)}(B_R^c)$ and
\begin{align}\label{eq:W-product}
\norm{fg}_{W^{k,p}_{\delta_1 + \delta_2 -(d/p)}(B_R^c)} \le
C \; \norm{f}_{W^{m,p}_{\delta_1 -(d/p)}(B_R^c)}  \,  \norm{g}_{W^{l,p}_{\delta_2 -(d/p)}(B_R^c)}.
\end{align}
In particular, for $m>d/p$ and $\delta\in\R$ the weighted Sobolev space $W^{m,p}_\delta(B_R^c)$ is a Banach algebra.
\end{proposition}

\begin{remark}\label{rem:R^d}
Note that inequality \eqref{eq:W-product} as well as items (i) and (ii) of Proposition \ref{prop:main_technical} hold with 
$B_R^c$ replaced by $\R^d$ (see Lemma 1.2 in \cite{McOTo2}).
This also easily follows from Proposition \ref{prop:main_technical} (i), (ii), Proposition \ref{prop:W_R}, and Lemma \ref{lem:break}. 
\end{remark}

\begin{proof}[Proof of Proposition \ref{prop:W_R}]
Choose the parameters $\delta_1,\delta_2,k,l,m$ and $R\ge 1$ as in the statement of the proposition.
By the generalized H\"older inequality with weights (see Lemma \ref{lem:holder_with_weights} below) we have that
\begin{align}\label{eq:holder_with_weights}
\norm{fg}_{L^p_{\delta_1+\delta_2-(d/p)}(B_R^c)} \le \norm{f}_{L^{p_1}_{\delta_1-\frac{d}{p_1}}(B_R^c)} \, 
\norm{g}_{L^{p_2}_{\delta_2-\frac{d}{p_2}}(B_R^c)}
\end{align}
for any $f\in L^{p_1}_{\delta_1-\frac{d}{p_1}}(B_R^c)$ and $g\in L^{p_2}_{\delta_1-\frac{d}{p_2}}(B_R^c)$
with $p=\frac{1}{p_1}+\frac{1}{p_2}$, $1<p\le p_1, p_2\le\infty$.
First consider the case when $0\le m<d/p$. Then by Proposition \ref{prop:main_technical} we have the continuous embedding
$W^{m,p}_{\delta -(d/p)}(B_R^c) \subseteq L^q_{\delta-\frac{d}{q}}(B_R^c)$ for any given
$q\in\big[p,d/(\frac{d}{p}-m)\big]$. Now choose $p_1\in\big[p,d/(\frac{d}{p}-m)\big]$ and 
$p_2\in\big[p,d/(\frac{d}{p}-l)\big]$ such that $\frac{1}{p} = \frac{1}{p_1} + \frac{1}{p_2}$.
This can be done, since 
\begin{align*}
\frac{1}{d/(\frac{d}{p}-m)}+\frac{1}{d/(\frac{d}{p}-l)}\le\frac{1}{p},
\end{align*}
which amounts to $m+l>d/p$. (This holds in view of the assumption that $m+l-k>d/p$ with $k\ge 0$.)
Hence, by Proposition \ref{prop:main_technical} $(i)$, we have that
\[
\norm{f}_{L^{p_1}_{\delta_1-\frac{d}{p_1}}(B_R^c)}\le C_1\norm{f}_{W^{m,p}_{\delta_1-(d/p)}(B_R^c)}
\quad\text{\rm and}\quad
\norm{g}_{L^{p_2}_{\delta_2-\frac{d}{p_2}}(B_R^c)}\le C_1\norm{g}_{W^{m,p}_{\delta_2-(d/p)}(B_R^c)}
\]
for any $f\in W^{m,p}_{\delta_1 -(d/p)}(B_R^c)$ and $g\in W^{l,p}_{\delta_2 -(d/p)}(B_R^c)$.
By combining this with \eqref{eq:holder_with_weights}, we then conclude that for $m+l>d/p$ one has that
\begin{align}\label{eq:pre-inequality}
\norm{fg}_{L^p_{\delta_1+\delta_2-(d/p)}(B_R^c)} \le C_2 \norm{f}_{W^{m,p}_{\delta_1-(d/p)}(B_R^c)} \, 
\norm{g}_{W^{l,p}_{\delta_2-(d/p)}(B_R^c)}.
\end{align}
with a constant $C_2>0$ independent of the choice of $R\ge 1$.
In the case when $m=d/p$ we prove \eqref{eq:pre-inequality} by arguing in the same way and by selecting 
$p_1, p_2\in[p,\infty)$ such that $\frac{1}{p} = \frac{1}{p_1} + \frac{1}{p_2}$.
In the case when $m>d/p$ we choose $p_1=p$ and $p_2=\infty$ to conclude from \eqref{eq:holder_with_weights} that
\[
\norm{fg}_{L^p_{\delta_1+\delta_2-(d/p)}(B_R^c)} \le 
\norm{f}_{L^\infty_{\delta_1}(B_R^c)} 
\norm{g}_{L^{p_2}_{\delta_2-(d/p)}(B_R^c)}.
\]
By combining this with Proposition \ref{prop:main_technical} $(ii)$ we prove \eqref{eq:pre-inequality}. 
Hence, the inequality \eqref{eq:pre-inequality} holds if $m+l>d/p$.

In order to complete the proof of the proposition, take $0\le k\le l$ and functions 
$f\in W^{m,p}_{\delta_1 -(d/p)}(B_R^c)$ and 
$g\in W^{l,p}_{\delta_2 -(d/p)}(B_R^c)$. For any $|\alpha|\le k$ we have 
\begin{align*}
|{\rm x}|^{|\alpha|}D^{\alpha}(fg) = \sum_{\beta\le\alpha}\binom{\alpha}{\beta} \left(|{\rm x}|^{|\beta|}D^{\beta}f\right) \, 
\left(|{\rm x}|^{|\alpha  - \beta|}D^{\alpha - \beta}g\right)
\end{align*}
where $|{\rm x}|^{|\beta|}D^{\beta}f \in W^{m-|\beta|,p}_{\delta_1 + |\beta|  -(d/p)}(B_R^c)$ and 
$|{\rm x}|^{|\alpha-\beta|}D^{\alpha-\beta} \in W^{l-|\alpha-\beta|,p}_{\delta_2 + |\alpha-\beta|  -(d/p)}(B_R^c)$.
Note that $(m-|\beta|)+(l-|\alpha-\beta|)=m+l-|\alpha|\ge m+l-k>d/p$.
This, together with \eqref{eq:pre-inequality}, then implies that
\[
|{\rm x}|^{|\alpha|}D^{\alpha}(fg)\in L^p_{\delta_1 + \delta_2 -(d/p)}(B^c_R)
\]
for any $|\alpha|\le k$. It also shows that \eqref{eq:W-product} holds with a constant $C>0$ independent of 
the choice of $R\ge 1$.
\end{proof}

In the proof of Proposition \ref{prop:W_R} we used the following generalized H\"older inequality with weights.

\begin{lemma}\label{lem:holder_with_weights}
For any real $\delta_1,\delta_2\in\R$, $1<p\le p_1,p_2\le\infty$ with $\frac{1}{p}=\frac{1}{p_1}+\frac{1}{p_2}$, and $R>0$ 
we have that
\[
\|fg\|_{L^p_{\delta_1+\delta_2-(d/p)}(B_R^c)}\le
\|f\|_{L^{p_1}_{\delta_1-\frac{d}{p_1}}(B_R^c)}\|g\|_{L^{p_1}_{\delta_1-\frac{d}{p_2}}(B_R^c)}
\]
for any $f\in L^{p_1}_{\delta_1-\frac{d}{p_1}}(B_R^c)$ and $g\in L^{p_2}_{\delta_1-\frac{d}{p_2}}(B_R^c)$.
\end{lemma}

\begin{proof}[Proof of Lemma \ref{lem:holder_with_weights}]
Assume that $1<p<p_1,p_2<\infty$. Note that the relation $\frac{1}{p}=\frac{1}{p_1}+\frac{1}{p_2}$ implies that 
$\frac{1}{(p_1/p)}+\frac{1}{(p_2/p)}=1$. Then by using the (standard) H\"older inequality we obtain
\begin{eqnarray}
\|fg\|_{L^p_{\delta_1+\delta_2-(d/p)}(B_R^c)}^p&=&
\int_{B_R^c}\Big|fg |{\rm x}|^{\delta_1+\delta_2-(d/p)}\Big|^p\,d{\rm x}=
\int_{B_R^c}\Big|f|{\rm x}|^{\delta_1-\frac{d}{p_1}}\Big|^p\Big|g|{\rm x}|^{\delta_2-\frac{d}{p_2}}\Big|^p\,d{\rm x}\nonumber\\
&\le&\Big(\int_{B_R^c}\Big|f|{\rm x}|^{\delta_1-\frac{d}{p_1}}\Big|^{p_1}\,d{\rm x}\Big)^{p/p_1}
\Big(\int_{B_R^c}\Big|g|{\rm x}|^{\delta_1-\frac{d}{p_2}}\Big|^{p_2}\,d{\rm x}\Big)^{p/p_2}
\end{eqnarray}
which implies the generalized H\"older inequality with weights in the considered case.
It is a straightforward task to check that the inequality also holds with $p_1=p$ and $p_2=\infty$.
This completes the proof of the lemma.
\end{proof}

Now, we a ready to prove Lemma \ref{coro:strict_banach}. Assume that $d=2$ and identify $\R^2$ with the complex plane $\C$
as explained in the Introduction.

\begin{proof} [Proof of Lemma \ref{coro:strict_banach}.]
Assume that $m>2/p$ and that $R\ge 1$.
In the case when $u_1, u_2\in W^{m,p}_{\gamma_N}(B_R^c)$ we obtain from  Proposition \ref{prop:W_R} that
\begin{align}\label{eq:ring_property}
|u_1 u_2|_{\mathcal{Z}^{m,p}_N(B_R^c)} \le C |u_1|_{\mathcal{Z}^{m,p}_N(B_R^c)} |u_2|_{\mathcal{Z}^{m,p}_N(B_R^c)}.
\end{align}
Now, assume that $u_1 = \frac{a}{z^k{\bar z}^l}$, $k,l\ge 0$, and  $u_2 \in W^{m,p}_{\gamma_N}(B_R^c)$.
Then one easily sees that $u_1 u_2 \in W^{m,p}_{\gamma_N+k+l}(B_R^c)\subseteq W^{m,p}_{\gamma_N}(B_R^c)$. 
We also have 
\begin{align*}
\abnorm{u_1 u_2}^p_{\mathcal{Z}^{m,p}_N(B_R^c)} &= \sum_{0 \le |\alpha| \le m} \int_{B_R^c} 
\abnorm{|z|^{\gamma_N + |\alpha|}\partial^{\alpha}\left(\frac{a}{z^k{\bar z}^l}\,u_2\right)}^p\,dx\,dy\\
&\le \sum_{0\le|\alpha|\le m}|a|^p \int_{B_R^c}\left(|z|^{\gamma_N + |\alpha|} 
\sum_{\beta\le \alpha}\binom{\alpha}{\beta}\abnorm{\partial^{\alpha - \beta}\left(\frac{1}{z^k{\bar z}^l}\right)}
\abnorm{\partial^{\beta}u_2}\right)^p\,dx\,dy\\
&\le C_0\sum_{0\le|\alpha|\le m}|a|^p \int_{B_R^c}\left(|z|^{\gamma_N + |\alpha|} \sum_{\beta\le\alpha} 
\frac{1}{|z|^{k+l+|\alpha| - |\beta|}}\abnorm{\partial^{\beta}u_2}\right)^p\,dx\,dy\\
&\le |a|^p C_1\sum_{0\le|\alpha|\le m} \int_{B_R^c}\left(\sum_{\beta\le\alpha} |z|^{\gamma_N+k+l+|\beta|} 
\abnorm{\partial^{\beta}u_2}\right)^p\,dx\,dy\\
&\le C_2 |a|^p \abnorm{u_2}^p_{\mathcal{Z}^{m,p}_N(B_R^c)}
\end{align*}
where $C_2>0$ is independent of the choice of $R\ge 1$ and 
$\partial^\alpha\equiv\partial_z^{\alpha_1}\partial_{\bar z}^{\alpha_2}$.
Hence, \eqref{eq:ring_property} holds with a constant $C>0$ independent of the choice of $R\ge 1$.
In the case when $u_1=\frac{1}{z^{k_1}{\bar z}^{l_1}}$ and  $u_1=\frac{1}{z^{k_1}{\bar z}^{l_1}}$ with
$0\le (k_1+l_1)+(k_2+l_2)\le N$ and $k_1,l_1,k_2,l_2\ge 0$, we obviously have that
\begin{align*}
\left|\frac{a}{z^{k_1}\bar{z}^{l_1}}\frac{b}{z^{k_2}\bar{z}^{l_2}}\right|_{\mathcal{Z}^{m,p}_N(B_R^c)} 
=\left|\frac{a}{z^{k_1}\bar{z}^{l_1}}\right|_{\mathcal{Z}^{m,p}_N(B_R^c)}  
\left| \frac{b}{z^{k_2}\bar{z}^{l_2}}\right|_{\mathcal{Z}^{m,p}_N(B_R^c)}
\end{align*}
by the definition of the norm \eqref{eq:ZR-norm} on $\mathcal{Z}^{m,p}_N(B_R^c)$.
If $(k_1+l_1)+(k_2+l_2)\ge N+1$, one easily sees that $u_1u_2\in W^{m,p}_{\gamma_N}(B_R^c)$
and that \eqref{eq:ring_property} holds with a constant independent of the choice of $R\ge 1$.
The general case easily follows from the considered cases.
\end{proof}

We will also need the following simple lemma.

\begin{lemma}\label{coro:norms}
Assume that $m>2/p$. Then for any $u\in\mathcal{Z}^{m,p}_N$ we have
\begin{enumerate}
\item[(i)] $|u|_{\mathcal{Z}^{m,p}_N(B_R^c)}\to a_{00} $ as $R\to\infty$.
\item[(ii)] For $R\ge 2$ we have that $\sup_{B_R^c}|u| \le |u|_{\mathcal{Z}^{m,p}_N(B_R^c)}$.
\item[(iii)] For any $u\in\mathcal{Z}^{m,p}_N$ we have that $\|u\|_{L^{\infty}} \le C \norm{u}_{\mathcal{Z}^{m,p}_N}$. 
\end{enumerate}
\end{lemma}

\begin{proof}[Proof of Lemma \ref{coro:norms}]
Take $u\in\mathcal{Z}^{m,p}_N$ such that $u = \chi\sum_{0 \le k+l\le N}\frac{a_{kl}}{z^k\bar{z}^l} + f$ and 
$f\in W^{m,p}_{\gamma_N}$. 
Then, item $(i)$ follows directly from the definition \eqref{eq:ZR-norm} of the norm 
$|u|_{\mathcal{Z}^{m,p}_N(B_R^c)}$ together with the fact that $\norm{f}_{W^{m,p}_{\gamma_N}(B_R^c)}\to 0$
as $R\to\infty$. 
In order to prove $(ii)$ take $R\ge 2$. Then, by Proposition \ref{prop:main_technical} $(ii)$, we have
\begin{equation*}
\sup_{B_R^c}|u| \le \sum_{0 \le k+l\le N} \frac{|a_{kl}|}{R^{k+l}} + \sup_{B_R^c}|f| \le 
\sum_{0\le k+l\le N} \frac{|a_{kl}|}{R^{k+l}} + \norm{f}_{W^{m,p}_{\gamma_N}(B_R^c)} = 
|u|_{\mathcal{Z}^{m,p}_N(B_R^c)}.
\end{equation*}
In view of Remark \ref{rem:R^d} we also have
\begin{equation*}
\|u\|_{L^{\infty}}\le C\Big(\sum_{0 \le k+l\le N} |a_{kl}| + \|f\|_{L^\infty}\Big)\le\|u\|_{\mathcal{Z}^{m,p}_N}
\end{equation*}
with $C:=\max\big(1,\sup\chi\big)$.
This completes the proof of the lemma.
\end{proof}

The proof of the following lemma follows easily from the properties of the Fourier transform in $S'$.

\begin{lemma}\label{lem:commutator}
Assume that $u,\partial_z^{-1}u,\partial_z^{-1}\partial u\in S'(\R^2,\R^2)$ where $\partial$ stands for 
$\partial_z$ or $\partial_{\bar z}$. Then the commutator $\big[\partial,\partial_z^{-1}\big]u$ is a polynomial of $\bar z$
with complex coefficients.
\end{lemma}

\section{Appendix: Euler's equation in complex coordinates} \label{appendix:complex_form}
Using the complex structure on $\mathbb{C}$ we will rewrite the Euler equation on $\R^2$ in terms of 
the holomorphic component of the fluid velocity vector field $v$. For simplicity we assume in this Appendix that all functions and 
vector fields are $C^\infty$-smooth on $\R^2$. A vector field $v = A(x,y)\partial_x + B(x,y) \partial_y$ on 
$\R^2$ can be written as
\begin{align*}
v=a(z,\bar{z})\,\partial_z + b(z, \bar{z})\,\partial_{\bar{z}},
\end{align*}
where $a = A+i B$, $b = A-i B$, and $\partial_z:=\frac{1}{2}(\partial_x-i\partial_y)$, 
$\partial_{\bar z}:=\frac{1}{2}(\partial_x+i\partial_y)$ are the Cauchy operators. 
In this way the function $a : \C\to\C$, $a\equiv a(z,{\bar z})$, is the {\em holomorphic component} of the real vector field $v$.
Note that $v$ is completely determined by $a(z,\bar{z})$. We will use this fact and rewrite the 2d Euler equation in therms
of its holomorphic component and the Cauchy operators.

\medskip

A direct computation gives
\begin{align}
\divop {v}&=\partial_x A+\partial_y B=(\partial_z+\partial_{\bar{z}})A+i(\partial_z-\partial_{\bar{z}})B\nonumber\\ 
&=\partial_z(A+iB)+\partial_{\bar{z}}(A-iB)\nonumber\\
&=a_z+\bar{a}_{\bar{z}}=2\mathop{\rm Re}(a_z).\label{eq:div-complex}
\end{align}
By identifying $v$ with its holomorphic component $a$, we can write 
\[
\divop a=a_z+\bar{a}_{\bar{z}}.
\]
Similarly, we have
\begin{equation}\label{eq:curl}
\curlop a=\frac{1}{2i}\big(a_z-\bar{a}_{\bar{z}}\big).
\end{equation}
Note that for $a=a(z,\bar{z})$ divergence free we obtain that $\curlop a=-i\,\partial_za$. We also have
\begin{equation*}
v\cdot\nabla v=\big(a\partial_z + \bar{a}\partial_{\bar{z}}\big)(a)\,\partial_z+
\big(a\partial_z+\bar{a}\partial_{\bar{z}}\big)(\bar{a})\,\partial_{\bar{z}}=
(aa_z+\bar{a}a_{\bar{z}})\,\partial_z+(a\bar{a}_z+\bar{a}\bar{a}_{\bar{z}})\,\partial_{\bar{z}}.
\end{equation*}
For a smooth function $\p:\R^2\to\R$ we get
\begin{align*}
\nabla\p &=(\partial_x\p)\partial_x+(\partial_y\p)\partial_y =\p_x(\partial_z+\partial_{\bar{z}}) + 
i\p_y(\partial_z-\partial_{\bar{z}})\\
&=(\p_x+i\p_y)\,\partial_z+(\p_x -i\p_y)\,\partial_{\bar{z}}\\
&=2\p_{\bar{z}}\partial_z+2\p_z\partial_{\bar{z}}.
\end{align*}
As a direct consequence of these identities we obtain that the 2d Euler equation \eqref{eq:euler} can be rewritten as
\begin{equation}\label{EQEulerComplexVar}
u_t+(uu_z+\bar{u}u_{\bar{z}})=-2\p_{\bar{z}},\quad\divop u=u_z+\bar{u}_{\bar{z}}=0,
\end{equation}
where $u\equiv a$ is the holomorphic component of the fluid velocity $v$ and $\p: \C\to\R$ is the pressure.

\medskip

Finally, note that for a divergence free vector $v$ we have that
\begin{eqnarray}
\divop(aa_z+\bar{a}a_{\bar{z}})&=&\partial_z (aa_z+\bar{a}a_{\bar{z}})+\partial_{\bar{z}}(\bar{a}\bar{a}_{\bar{z}}+a\bar{a}_z)\nonumber\\ 
&=&(a_z)^2+(\bar{a}_{\bar{z}})^2+2\bar{a}_z a_{\bar{z}}+a\big(a_{zz}+(\bar{a}_{\bar{z}})_z\big)+
\bar{a}\big((\bar{a}_{\bar{z}})_{\bar{z}}+a_{z\bar{z}}\big)\nonumber\\
&=&2(a_z)^2+2|a_{\bar{z}}|^2,\label{eq:Q}
\end{eqnarray}
where we used that $0=\divop v=a_z+\bar{a}_{\bar{z}}$ by \eqref{eq:div-complex}.



\begin{thebibliography}{99}

\bibitem{AMRBook} R. Abraham, J. Marsden, and T. Ratiu, {\em Manifolds, tensor analysis, and applications}, 
$\bf 75$, Springer, 2012

\bibitem{AG} S. Alinhac, P. Gerard, {\em Pseudo-differential operators and the Nash-Moser theorem},
Graduate Studies in Mathematics, $\bf 81$, AMS, Providence, 2007

\bibitem{BardosTiti}  C. Bardos, E. Titi, {\em Euler equations for an ideal incompressible fluid},
Russian Math. Surveys, $\bf 62$(2007), no. 3, 409-451

\bibitem{Bran} L. Brandolese, {\em Space-time decay of Navier-Stokes flows invariant under rotations},
Math. Ann., $\bf 329$(2004), no. 4, 685-706

\bibitem{BranMe} L. Brandolese, Y. Meyer, {\em On the instantaneous spreading for the Navier-Stokes system in 
the whole space. A tribute to J. L. Lions.},  ESAIM Control Optim. Calc. Var., $\bf 8$(2002), 273-285

\bibitem{Cantor} M. Cantor, {\em Groups of diffeomorphisms of $\R^n$ and the flow of a perfect fluid},
Bull. Amer. Mmath. Soc., $\bf 81$(1975), 205-208

\bibitem{Con1} A. Constantin, {\em Existence of permanent and breaking waves for a
shallow water equation: a geometric approach}, Ann. Inst. Fourier, Grenoble, $\bf 50$(2000), no. 2,
321-362
 
\bibitem{BS} I. Bondareva, M. Shubin, {\em Equations of Korteweg-de Vries type in classes of increasing 
functions}, Trudy Sem. Petrovsk., $\bf 256$(1989), no. 14, 45-56

\bibitem{EM} D. Ebin, J. Marsden, {\em Groups of diffeomorphisms and the motion
of an incompressible fluid}, Ann. Math., $\bf 92$(1970), 102-163
 
\bibitem{KPST} T. Kappeler, P. Perry, M. Shubin, P. Topalov, {\em Solutions of mKdV in classes of functions 
unbounded at infinity}, J. Geom. Anal., $\bf 18$(2008), 443-477

\bibitem{IKT} H. Inci, T. Kappeler, P. Topalov, {\em On the Regularity of the Composition of Diffeomorphisms},
Memoirs of the AMS, $\bf 266$, no.1062, 2013

\bibitem{Kato1} T. Kato, {\em On classical solutions of the two-dimensional non-stationary Euler equation}
Arch. Rational Mech. Anal., $\bf 25$(1967), 188-200

\bibitem{KR} I. Kukavica, E. Reis, {\em Asymptotic expansion for solutions of the Navier-Stokes equations 
with potential forces}, J. Differential Equations, $\bf 250$(2011), no. 1, 607-622

\bibitem{Lockhart} R. Lockhart, Fredholm properties of a class of elliptic operators on non-compact manifolds,
Duke J. Math., $\bf 48$(1981), no. 1, 289-312

\bibitem{McOwen1} R. McOwen, {\em The behavior of the Laplacian on weighted Sobolev spaces}, 
Comm. Pure Appl. Math., $\bf 32$(1979) 783-795 

\bibitem{McOTo1} R. McOwen, P. Topalov, {\em Asymptotics in shallow water waves}, 
Discrete Contin. Dyn. Syst., $\bf 35$(2015), no. 7, 3103-3131

\bibitem{McOTo2} R. McOwen, P. Topalov, {\em Groups of asymptotic diffeomorphisms}, 
Discrete Cont. Dyn. Syst., $\bf 36$(2016), no. 11, 6331-6377

\bibitem{McOTo3} R. McOwen, P. Topalov, {\em Spatial asymptotic expansions in the incompressible Euler equation}, 
GAFA, $\bf 27$(2017), 637-675

\bibitem{MY} G. Misiolek, T. Yoneda, {\em Continuity of the solution map of the Euler equation 
in H{\"o}lder spaces and weak norm inflation in Besov spaces}, Trans. Amer. Math. Soc.,
$\bf 370$(2018), no. 7, 4709-4730

\bibitem{Montgomery} D. Montgomery, {\em On continuity in topological groups}, 
Bull. Amer. Math. Soc., $\bf 42$(1936), 879-882

\bibitem{DobrShaf} S. Dobrokhotov, A. Shafarevich, {\em Some integral identities and remarks on the decay
at infinity of the solutions of the Navier-Stokes Equations in the entire space},
Russian J. Math. Phys., $\bf 2$(1994), no. 1, 133-135

\bibitem{Serfati} P. Serfati, {\em Structure d'E.D.O. holomorphe en temps, pour l'\'equation d'Euler 
\`a faible r\'egularit\'e en espace}, PhD thesis, Chapitre III, Paris, 1992

\bibitem{Serfati1} P. Serfati, {\em \'Equation d'Euler et holomorphies \`a faible r\'egularit\'e spatiale},
C. R. Acad. Sci. Paris Sér. I Math.,  $\bf 320$(1995), no. 2, 175-180

\bibitem{Serfati2} P. Serfati, {\em Structures holomorphes \`a faible r\'egularit\'e spatiale en m\'ecanique des fluides},
J. Math. Pures Appl., $\bf 74$ (1995), no. 2, 95-104

\bibitem{Stein} E. Stein, {\em Note on singular integrals}, 
Proc. Amer. Math. Soc., $\bf 8$(1957), 250-254 

\bibitem{SteinWeiss} E. Stein, G. Weiss, {\em Fractional integrals on n-dimensional Euclidean space},
J. Math. and Mech., $\bf 7$(1958), no. 4, 503-514

\bibitem{SteinBook} E. Stein, {\em Singular integrals, and differentiability properties of functions}, 
Princeton, 1970

\bibitem{Subin} M. Subin, {\em Pseudodifferential Operators and spectral Theory}, Springer, Berlin, 2011

\bibitem{ST1} X. Sun, P. Topalov, {\em On the group of almost periodic diffeomorphisms and its exponential map},
IMRN, $\bf 2019$, to appear

\bibitem{singIntCZ} A. Calderon, A. Zygmund, {\em On the existence of certain singular integrals}, 
Acta Mathematica, {\bf 88}(1952), 85-139

\bibitem{LangDiffGeo}S. Lang, {\em Differential Manifolds},
Addison-Wesley Series in Mathematics, Boston, 1972
\end{thebibliography}
\end{document}